\documentclass[12pt,a4paper]{amsart}
\usepackage{amsfonts,amsmath,amssymb,amsthm,mathrsfs}
\usepackage{array,latexsym,enumerate,ifpdf}
\usepackage[usenames,dvipsnames]{color}
\usepackage{graphicx,epsfig}
\usepackage[top=2.5cm, bottom=2.5cm, left=2.5cm, right=2.5cm,twoside=false]{geometry}
\usepackage[all]{xy}
\numberwithin{equation}{section}

\makeatletter \def\subsection{\@startsection{subsection}{3}%
  \z@{.5\linespacing\@plus.7\linespacing}{.5\linespacing}%
  {\normalfont\itshape}} \makeatother

\makeatletter \renewenvironment{proof}[1][\proofname]{%
  \par\pushQED{\qed}\normalfont%
  \topsep6\p@\@plus6\p@\relax
  \trivlist\item[\hskip\labelsep\bfseries#1\@addpunct{.}]%
  \ignorespaces}{%
  \popQED\endtrivlist\@endpefalse} \makeatother

\theoremstyle{plain}
\newtheorem{theorem}{Theorem}[section]
\newtheorem{lemma}[theorem]{Lemma}
\newtheorem{proposition}[theorem]{Proposition}
\newtheorem{corrolary}[theorem]{Corollary}

\newtheorem{claim}{Claim} \makeatletter \@addtoreset{claim}{theorem}\makeatother

\theoremstyle{definition}
\newtheorem{definition}[theorem]{Definition}
\newtheorem{example}[theorem]{Example}
\newtheorem{noname}[theorem]{}
\newtheorem{remark}[theorem]{Remark}
\newtheorem{construction}[theorem]{Construction}
\newtheorem{notation}[theorem]{Notation}

\theoremstyle{remark}
\newtheorem*{smallremark}{Remark}
\newtheorem{case}{Case} \makeatletter \@addtoreset{case}{theorem}\makeatother

\newcommand{\bthm}{\begin{theorem}}
\newcommand{\bprop}{\begin{proposition}}
\newcommand{\blem}{\begin{lemma}}
\newcommand{\bcor}{\begin{corrolary}}
\newcommand{\brem}{\begin{remark}}
\newcommand{\bdfn}{\begin{definition}}
\newcommand{\bitem}{\begin{itemize}}
\newcommand{\bex}{\begin{example}}
\newcommand{\bno}{\begin{noname}}
\newcommand{\bsrem}{\begin{smallremark}}
\newcommand{\bnot}{\begin{notation}}
\newcommand{\bcon}{\begin{construction}}
\newcommand{\bca}{\begin{case}}
\newcommand{\bcl}{\begin{claim}}
\newcommand{\beq}{\begin{equation}}

\newcommand{\eeq}{\end{equation}}
\newcommand{\ecl}{\end{claim}}
\newcommand{\eca}{\end{case}}
\newcommand{\econ}{\end{construction}}
\newcommand{\enot}{\end{notation}}
\newcommand{\esrem}{\end{smallremark}}
\newcommand{\eno}{\end{noname}}
\newcommand{\eex}{\end{example}}
\newcommand{\eitem}{\end{itemize}}
\newcommand{\ethm}{\end{theorem}}
\newcommand{\eprop}{\end{proposition}}
\newcommand{\elem}{\end{lemma}}
\newcommand{\ecor}{\end{corrolary}}
\newcommand{\erem}{\end{remark}}
\newcommand{\edfn}{\end{definition}}
\newcommand{\benum}{\begin{enumerate}}
\newcommand{\eenum}{\end{enumerate}}

\newcommand{\wt}{\widetilde}
\newcommand{\cal}[1]{\mathcal{#1}}

\newcommand{\ds}{\displaystyle}
\newcommand{\ti}{\tilde}

\def\8{\infty}
\def\PP{\mathbb{P}}

\def\C{\mathbb{C}}
\def\Z{\mathbb{Z}}

\def\Q{\mathbb{Q}}

\def\xra{\xrightarrow}

\def\:{\colon}
\def\map{\dashrightarrow}

\def\med{\medskip}

\def\Bk{\operatorname{Bk}}

\def\Supp{\operatorname{Supp}}
\def\Pic{\operatorname{Pic}}

\def\Exc{\operatorname{Exc}}
\def\Spec{\operatorname{Spec}}
\def\NS{\operatorname{NS}}
\def\ind{\operatorname{ind}}
\def\tip{\operatorname{tip}}

\ifpdf \usepackage[linkbordercolor={0 0 1}]{hyperref} \else \usepackage[hypertex,linkbordercolor={0 0 1}]{hyperref} \fi

\begin{document}

\title[The Coolidge-Nagata conjecture]{The Coolidge-Nagata conjecture}

\author[Mariusz Koras]{Mariusz Koras}
\address{Mariusz Koras: Institute of Mathematics, University of Warsaw, ul. Banacha 2, 02-097 Warsaw}\email{koras@mimuw.edu.pl}
\author[Karol Palka]{Karol Palka}
\address{Karol Palka: Institute of Mathematics, Polish Academy of Sciences, ul. \'{S}niadeckich 8, 00-656 Warsaw, Poland}
\email{palka@impan.pl}

\thanks{Both authors were supported by the polish National Science Center, the first one from the grant No. 2013/11/B/ST1/02977 and the second one from the grant No. 2012/05/D/ST1/03227}


\begin{abstract}Let $E\subseteq \PP^2$ be a complex rational cuspidal curve contained in the projective plane. The Coolidge-Nagata conjecture asserts that $E$ is Cremona equivalent to a line, i.e.\ it is mapped onto a line by some birational transformation of $\PP^2$. In \cite{Palka-Coolidge_Nagata1} the second author analyzed the log minimal model program run for the pair $(X,\frac{1}{2}D)$, where $(X,D)\to (\PP^2,E)$ is a minimal resolution of singularities, and as a corollary he established the conjecture in case when more than one irreducible curve in $\PP^2\setminus E$ is contracted by the process of minimalization. We prove the conjecture in the remaining cases.
\end{abstract}

\maketitle

\section{Main result} All varieties considered are complex algebraic. Two subvarieties $X_1\subset X, X_2\subset X$ are \emph{equivalent in $X$} if there exists an automorphism $\varphi$ of $X$, such that $\varphi(X_1)=X_2$. In case $X$ is rational and $X_1, X_2$ are of codimension $1$ we say they are \emph{Cremona equivalent} if there exists a birational transformation $\varphi$ of $\PP^n$ mapping $X_1$ onto $X_2$. We are interested in studying the way a projective homology line, i.e.\ a curve having singular homology of $\PP^2$, and hence homeomorphic to $\PP^2$ in the Euclidean topology, can be embedded into the projective plane. By the adjunction formula (abstract) projective lines in $\PP^2$ have degree at most two, hence are Cremona equivalent. On the other hand, describing non-equivalent projective homology lines in $\PP^2$ is a hard problem with many connections (see \cite{FLMN}). Because a projective homology line has analytically irreducible singularities, it is nothing else than a rational cuspidal curve. There are infinitely many non-equivalent examples known and we are still far from understanding the situation completely. Here we prove the following conjecture.

\begin{theorem}[The Coolidge-Nagata conjecture]\label{thm:CN} Every complex rational cuspidal curve (i.e.\ every projective homology line) contained in the projective plane is Cremona equivalent to a line.
\end{theorem}

The conjecture is traditionally attributed to Coolidge and Nagata, who studied planar rational curves and their behaviour under the action of the Cremona group (see \cite[Book IV,\S II.2]{Coolidge} and Nagata \cite{Nagata}).\footnote{In general, the problem of determining which planar rational curves are Cremona equivalent to a line is known as the  'Coolidge-Nagata problem'.} It appears in an explicit form for instance in \cite[p. 234]{MaSa-cusp}. An analogous problem in the affine case has been solved. Indeed, by a celebrated result of Abhyankar-Moh \cite{AbhMoh_the_line_thm} and Suzuki \cite{Suzuki_AMSthm} every affine line in $\C^2=\Spec \C[x,y]$ is equivalent to $x=0$ and by a result of Lin-Zaidenberg \cite{LinZaid_line} every affine homology line in $\C^2$ other than $\C^1$ is equivalent to one of $x^n=y^m$ for some coprime positive integers $n>m\geq 2$ (see \cite{GurjarMiyanishi_AMS_and_LZ_thms}, \cite{Koras-ab_moh} or \cite{Palka-AMS-LZ} for proofs using the theory of open surfaces). It follows that all affine homology lines in $\C^2$ are Cremona equivalent.

The proof of the conjecture goes as follows. Suppose $\bar E\subset \PP^2$ is a rational cuspidal curve violating the Coolidge-Nagata conjecture. We may assume $\PP^2\setminus \bar E$ is of log general type (see \cite[2.4]{Palka-Coolidge_Nagata1}). Let $\pi_0\:(X_0,D_0)\to (\PP^2,\bar E)$ be a composition of a minimal sequence of blowups, such that the proper transform $E_0\subset X_0$ of $\bar E$ is smooth. This resolution is dominated by a minimal \emph{log resolution} $\pi\:(X,D)\to (\PP^2,\bar E)$ for which $D$, the total reduced transform of $\bar E$, is an snc-divisor. By the criterion of Kumar-Murthy \cite{Kumar-Murthy} (see \ref{prop:fundamentals}(ii)), which strengthens the original criterion by Coolidge, we have $\kappa(K_{X_0}+\frac{1}{2}E_0)\geq 0$, hence $\kappa(K_{X_0}+\frac{1}{2}D_0)\geq 0$. In \cite{Palka-minimal_models} and \cite{Palka-Coolidge_Nagata1} the second author analyzed the log minimal model program run for the pair $(X_0,\frac{1}{2}D_0)$ and he proved that the number $n$ of irreducible curves in $\PP^2\setminus \bar E$ contracted by the process of minimalization is at most one. This established the conjecture in particular in the case when $\bar E$ has more than two cusps. We follow this approach incorporating other tools developed independently by the first author. The key step, Theorem \ref{prop:n=0}, rules out the case $n=1$. The proof is hard, because for $n=1$ bounds coming from the log MMP are weaker. One important ingredient here is that the pushforwards of $2K_{X_0}+D_0$ and $2K_{X_0}+E_0$ on the minimal model are respectively nef and effective, hence their intersection is non-negative. The second one is that by the Kawamata-Viehweg vanishing theorem we have $K_X\cdot (K_X+D)=h^0(2K_X+D)>0$ so, because the process of minimalization is shown not to change the Euler characteristic of the open part of the surface, the logarithmic version of the Bogomolov-Miyaoka-Yau inequality gives strong bounds on the shape of the divisor $D$ (see \ref{lem:mmp_basics}(iv)), and hence on $D_0$ and the singularities of $\bar E$. Still, controlling possible shapes of $D_0$ together with the process of minimalization is the most difficult task. Once this is done, we know that the proces of minimalization of $(X_0,\frac{1}{2}D_0)$ contracts only curves in $D_0$. Then we analyze in turn the process of minimalization of $(X_0,\frac{1}{2}E_0)$ (see \ref{lem:(X,E)}). It produces new $(-1)$-curves, but their intersections with $D_0-E_0$ are harder to control. Often we need to rule out very concrete shapes of $D_0$. This is done using a detailed description of exceptional divisors over cusps of $\bar E$ in terms of their types and Hamburger-Noether pairs (see Subsection \ref{ssec:HN_and_types}).

The classification of projective homology lines in $\PP^2$ up to equivalence, not just up to Cremona equivalence, is a more difficult task. One of the long standing conjectures is that they cannot have more than four singular points. We will use the tools we created to prove the latter conjecture in a forthcoming article.

\

\tableofcontents

\vfill\eject

\section{Preliminaries}\label{sec:Prelim}

\subsection{Surfaces and divisors} We keep the notation of \cite[\S2-\S3]{Palka-Coolidge_Nagata1}, the reader is advised to consult it for details.

Let $X$ be a smooth projective surface. The Picard rank of $X$ is denoted by $\rho(X)$ and the canonical divisor by $K_X$. Fix a reduced effective divisor $D$ on $X$. We define the \emph{discriminant of $D$} as $d(D)=\det(-Q(D))$, where $Q(D)$ is the intersection matrix of $D$. We put $d(0)=1$. By $\#D$ we denote the number of irreducible components of $D$. If $T$ is an (irreducible) component of $D$ we define the \emph{branching number of $T$ in $D$} as $\beta_D(T)=T\cdot (D-T)$. If a reduced effective divisor $T$ equals $T_1+\ldots+T_k$, where $T_i$ are smooth, $T_i\cdot T_{i+1}=1$ and $T_i\cdot T_j=0$ for $j>i+1$, then we call $T$ a \emph{chain}. Curves are always assumed to be irreducible and reduced. When we say two curves meet once, twice, etc.\ we mean that their intersection number is respectively $1, 2$, etc. A chain of rational curves with successive self-intersections $a_1,\ldots,a_k$ is denoted by $[-a_1,\ldots,-a_n]$. A $(-1)$-curve in $D$ is called \emph{superfluous} if it meets at most two other components of $D$, each at most once.

Assume $D$ is not a rational chain. A \emph{twig of $D$} is a chain contained in $D$ which contains a tip of $D$ (a component with $\beta_D\leq 1$) and no branching component of $D$ (no component with $\beta_D\geq 3$). A twig comes with a natural linear order of components in which the tip is the first component. Assuming $T=[a_1,\ldots,a_n]$ is a rational twig (not necessarily maximal) of $D$ with $a_i\geq 2$  we define the \emph{inductance of $T$} as $$\ind (T)=\frac{d(T-\tip(T))}{d(T)}.$$ We put $\ind(0)=0$. We define $\Bk_D T$, the \emph{bark} of $T$ with respect to $D$, as the unique (effective) $\Q$-divisor supported on $\Supp T$, such that $$\Bk_D T\cdot R=\beta_D(R)-2$$ for every component $R$ of $T$, equivalently that $\Bk_D T\cdot R$ equals $-1$ if $R$ is the tip of $D$ contained in $T$ and is zero otherwise. It is easy to show that $(\Bk_D T)^2=-\ind (T).$ A \emph{$(-2)$-twig} is a twig consisting of $(-2)$-curves. By $(2)_{k}$ we denote a sequence $2,2,\ldots,2$ of of length $k$. For instance $[5,(2)_3]=[5,2,2,2]$. The divisors $D$ of most interest to us are snc-boundaries of affine surfaces of log general type. In this case we define $\Bk D$ and $\ind(D)$ as the sums of respective quantities summed up over all maximal rational twigs of $D$. Barks are closely related to the negative part of the Zariski decomposition of $K_X+D$. For a general definition and details the reader should consult \cite[2.3.5]{Miyan-OpenSurf}.

We will need the following form of the logarithmic Bogomolov-Miyaoka-Yau inequality. Let $(X,D)$ be a smooth pair, i.e.\ a pair consisting of a smooth projective surface $X$ and an effective reduced snc-divisor $D$. We say that $D$ and the pair $(X,D)$ are \emph{snc-minimal} if a contraction of any $(-1)$-curve in $D$ maps it onto a divisor which is not snc. For any divisor $T$ on $X$ we denote the Iitaka-Kodaira dimension of $T$ by $\kappa(T)$.

\begin{lemma}\label{lem:BMY} Let $(X,D)$ be a smooth snc-minimal pair, such that $\kappa(K_X+D)\geq 0$. Assume $X\setminus D$ is affine and contains no affine lines. Then $$(K_X+D)^2+\ind(D)\leq 3\chi(X\setminus D).$$
\end{lemma}

\begin{proof}The divisor $D$ is connected and supports an ample divisor, so it does not have a negative definite intersection matrix. By \cite[2.5(ii)]{Palka-exceptional} $(K_X+D-\Bk D)^2\leq 3\chi(X\setminus D)$. But since $(X,D)$ is snc-minimal and $X\setminus D$ contains no affine lines, \cite[6.20]{Fujita-noncomplete_surfaces} says that $\Bk D$ is the negative part of the Zariski decomposition of $K_X+D$. In particular, $(K_X+D)^+=K_X+D-\Bk D$ and $(K_X+D-\Bk D)^2=(K_X+D)^2-(\Bk D)^2=(K_X+D)^2+\ind(D)$.
\end{proof}

The following lemma will be used frequently to bound the inductance.

\blem\label{lem:ind>=} Let $T=T_1+\ldots+T_k$ be a rational twig of $D$ (with $T_1$ being the tip of $D$), such that $T_i^2\leq -2$ for all $i$. We have:\benum[(i)]
    \item if $k\geq 2$ then $d(T)=(-T_1^2)d(T-T_1)+d(T-T_1-T_2)$,
    \item $\frac{1}{d(T_1)}=\ind (T_1)\leq \ind (T_1+T_2)\leq \ind(T_1+T_2+T_3)\leq \ldots\leq \ind (T)$.
\eenum
\elem

\begin{proof} (i) follows from elementary properties of determinants. (ii) By (i) we have $$\ind([a_1,\ldots,a_n])=1/(a_1-\ind([a_2,\ldots,a_n])),$$ so we prove that $\ind(T)\geq \ind (T-T_k)$ by induction with respect to $k$.
\end{proof}

If $\alpha\:X\to X'$ is a birational morphism of surfaces we put $\rho(\alpha)=\rho(X)-\rho(X')$ and we denote the reduced exceptional divisor of $\alpha$ by $\Exc \alpha$. The proper transform of a curve $C\subset X'$ on $X'$ is denoted by $(\alpha^{-1})_*C$. If for a curve $C\subset X$ we have $C\cdot \Exc \alpha>0$ then we say that $\alpha$ \emph{touches $C$}. In case $C\cdot (\alpha^*\alpha_*C-C)=1$ it \emph{touches $C$ once}.

\subsection{Exceptional divisors over cusps}\label{ssec:HN_and_types}

For the convenience of the reader we repeat, after \cite[\S 3]{Palka-Coolidge_Nagata1}, the definition of Hamburger-Noether pairs (characteristic pairs) associated with a cusp. For a detailed treatment see \cite{Russell2}. As an input data take an analytically irreducible germ of a singular curve $(\chi,q)$ on a smooth surface and a (germ of a) curve $C$ passing through $q$, smooth at $q$. Put $(C_1,\chi_1,q_1)=(C,\chi,q)$, $c_1=(C_1\cdot \chi_1)_{q_1}$, where $(\ \cdot\ )_{q_1}$ denotes the local intersection index at $q_1$, and choose a local coordinate $y_1$ at $q_1$ in such a way that $Y_1=\{y_1=0\}$ is transversal to $C_1$ at $q_1$ and $p_1=(Y_1\cdot\chi_1)_{q_1}$ is not bigger than $c_1$. Blow up over $q_1$ until the proper transform $\chi_2$ of $\chi_1$ intersects the reduced total transform of $C_1+Y_1$ not in a node. Let $q_2$ be the point of intersection and let $C_2$ be the last exceptional curve. Put $c_2=(C_2\cdot \chi_2)_{q_2}$. We repeat this procedure and we define successively $(\chi_i,q_i)$ and $C_i$ until $\chi_{h+1}$ is smooth for some $h\geq 1$. This defines a sequence $$\binom{c_1}{p_1}, \binom{c_2}{p_2}, \ldots, \binom{c_h}{p_h},$$ depending on the choice of $C$. It follows from the definition that $c_i\geq p_i$, $\gcd(c_i,p_i)=c_{i+1}$ (where $c_{h+1}=1$) and that $p_1$ is the first and maximal number in the sequence of multiplicities of $q\in \chi$. Since $\chi$ is irreducible, the total exceptional divisor contains a unique $(-1)$-curve.

Because of the forced condition $c_i\geq p_i$ the sequence is usually longer than the sequence of Puiseux pairs. Although it is defined for any initial curve $C$, in this article we will choose for $C$ a smooth germ maximally tangent to $\chi$ (note that because $\chi$ is singular its intersection with smooth germs passing through $q$ is bounded from above). Then $c_1>p_1$. For this choice of $C$ we refer to the above sequence as the sequence of \emph{Hamburger-Noether pairs} (or \emph{characteristic pairs}) \emph{of the (minimal log) resolution of $(\chi_1,q_1)$}. It is convenient to extend the definition to the case when $(\chi_1,q_1)$ is smooth by defining its sequence of characteristic pairs to be $\binom{1}{0}$. By $\binom{u}{u}_k$ we mean a sequence of pairs $\binom{u}{u},\ldots,\binom{u}{u}$ of length $k$. For $i\leq h$ let $(\mu_j)_{j\in I_i}$ be the non-increasing sequence of multiplicities of successive centers for the sequence of blowups as above leading from $\chi_i$ to $\chi_{i+1}$. The sequence $(\mu_j)_{j\in I_1},\ldots,(\mu_j)_{j\in I_h}$ is the \emph{multiplicity sequence of the singularity $(\chi,q)$}. Note that the composition of blowups corresponding to multiplicities bigger than $1$ is the minimal weak resolution of singularities.

Let now $\pi\:X\to X'$ be a proper birational morphism of smooth surfaces, such that the exceptional divisor $Q=\Exc \pi$ contains a unique $(-1)$-curve $U$. If $U$ is not a tip of $Q$ then $\pi$ is a minimal log resolution of a germ of a singular curve $(\chi,q)$ on $X'$, namely the image of a smooth germ transversal to $U$, so we define the \emph{sequence of characteristic pairs of $Q$} to be the one of $(\chi,q)$. The sequence of characteristic pairs of a zero divisor is defined to be empty. In case $U$ is a tip of $Q$ let $(X,Q)\to (Y,Q')$ be a composition of a minimal number of contractions, say $m$, of $(-1)$-curves in $Q$ and its successive images, such that $Q'$ contains no $(-1)$-tip. If $(\binom{c_i}{p_i})_{i\leq h}$ is the sequence of characteristic pairs for $Q'$ then the sequence of characteristic pairs of $Q$ is by definition $(\binom{c_i}{p_i})_{i\leq h},\binom{1}{1}_m$. The sequence depends only on the intersection matrix of $Q$.

\blem\label{lem:HN_pairs} Let $\pi\:X\to X'$ be a birational morphism of smooth projective surfaces, such that $Q=(\Exc \pi)_{red}$ contains a unique $(-1)$-curve. Let $(\binom{c_i}{p_i})_{i\leq h}$ be the sequence of characteristic pairs of $Q$. If $\Gamma$ is a curve on $X$ meeting $Q$ only in the $(-1)$-curve then:
\begin{align}
-K_{X'}\cdot \pi_*\Gamma+K_{X}\cdot \Gamma &=\ds (Q\cdot \Gamma )(c_{1}+p_1+p_2+\ldots+p_h-1),\\
(\pi_*\Gamma)^2-\Gamma^2 &=(Q\cdot \Gamma)^2(c_1p_1+c_2p_2+\ldots+c_hp_h)
\end{align}
\elem

\begin{proof}Let $(\mu_i)_{i\in I_j}$ be the non-increasing sequence of multiplicities of successive centers for the sequence of blowups as above leading from $\chi_j$ to $\chi_{j+1}$. The corresponding multiplicities for the proper transforms of the germ $(\Gamma,\Gamma\cap Q)$ are $(Q\cdot \Gamma)\mu_i$, so by elementary properties of a blowup we have $$-K_{X'}\cdot \pi_*\Gamma+K_{X}\cdot\Gamma=\ds\sum_{j\in I_i,i\leq h}(Q\cdot \Gamma)\mu_{j}\text{\ \ and\ \ } (\pi_*\Gamma)^2-\Gamma^2=\ds\sum_{j\in I_i,i\leq h}((Q\cdot \Gamma)\mu_{j})^2.$$ By induction with respect to $\max(c_i,p_i)$ we have $$\ds \sum_{j\in I_i}\mu_j=c_i+p_i-\gcd(c_i,p_i) \text{\ \ and\ \ } \ds \sum_{j\in I_i}\mu_j^2=c_ip_i.$$ The lemma follows.
\end{proof}

We now define the notion of a \emph{type} for $Q$, which is especially useful for small values of $K\cdot Q$. Recall that given a reduced snc-divisor $V$ we say that the blowup with a center on $V$ is \emph{inner} (for $V$) if the center belongs to exactly two components of $V$, otherwise it is \emph{outer}. Let's write $\pi\:X\to X'$ as a composition of blowups $\pi=\sigma_1\circ\ldots\circ\sigma_{\#Q}$. We can think of $Q$ as being created from $q=\pi(Q)\in X'$ by the sequence $\sigma_1,\ldots,\sigma_{\#Q}$ of blowups, where we start with the first exceptional divisor and each time we replace it with the subsequent reduced total transform. First of all, we define the type of a zero divisor to be $(0)$.

First assume $U$ is not a tip of $Q$. It follows that the last blowup is inner. Group the members of the above sequence into maximal alternating blocks (of positive length) of outer blowups and of inner blowups. Treat $\sigma_1$ as part of the first block of outer blowups. Let $k_1, r_1, k_1, \ldots r_m$ be the lengths of subsequent blocks of outer and inner blowups respectively. Then $k_1\geq 2$ (because $\sigma_2$ is outer by definition), $\sigma_{(k_1+r_1+\ldots+k_{i})+1},\ldots,\sigma_{(k_1+r_1+\ldots+k_i)+r_i}$ is the $i$'th block of inner blowups and we have $\sum_{i=1}^m (k_i+r_i)=\#Q$. We then say that $Q$ is \emph{of type $(r_1,\ldots,r_m)$}. From the definition it follows that for each $i$ the exceptional divisor of the composition of the blowups belonging to the $i$'th blocks of outer or inner blowups is a chain containing a tip of $Q$ and at most one branching component of $Q$. We call its proper transform on $X$ the \emph{$i$'th branch of $Q$}. The proper transform of the last exceptional curve $\Exc(\sigma_{k_1+r_1+\ldots+k_i+r_i})$ of the $i$'th branch is a branching component of $Q$ if $i<m$ and it is the unique $(-1)$-curve of $Q$ otherwise (which is a branching component of $D$).

If $U$ is a tip of $Q$ then we take the contraction $\alpha\:(X,Q)\to (Y,Q')$ as above and we define the type of $Q$ to be the one of $Q'$. The branches of $Q$ are the proper transforms of branches of $Q'$ and there is one more branch contracted by $\alpha$. Note that the type is a sequence of positive integers of length $1$, unless $Q=[(2)_s,1]$ for some $s\geq 0$.

Recall that the Fibonacci numbers are defined by $F_1=F_2=1$, $F_{n+1}=F_{n}+F_{n-1}$.

\blem\label{lem:type}Let $Q$ be as in \ref{lem:HN_pairs} and let $(r_1,\ldots,r_m)$ be its type. Let $U$ and $\mu(U)$ be respectively the $(-1)$-curve of $Q$ and its multiplicity in $\pi^{-1}(\pi(Q))$. Let $F_n$ denote the $n$'th Fibonacci number. Then:\benum[(i)]
    \item $K\cdot (Q-U)=r_1+\ldots+r_m$,
    \item $\mu(U)\leq F_{K\cdot Q+3}=F_{r_1+\ldots r_m+2}.$
\eenum
\elem

\begin{proof} We prove the statements by induction with respect to the number of components of $Q$. Let $\sigma\:\bar Y\to Y$ be a blowup with a center on the $(-1)$-curve of $Q$ and let $\bar Q$ be the total reduced transform of $Q$. If $\sigma$ is outer for $Q$ then $\bar Q=\sigma^*Q$, so $K_{\bar Y}\cdot \bar Q=K_Y\cdot Q$, hence the intersection with the canonical divisor, the type and the multiplicity are the same for $Q$ and $\wt Q$.

We may therefore assume $\sigma$ is inner for $Q$. Then $\bar Q$ is of type $(r_1,\ldots,r_m+1)$ and $\bar Q=\sigma^*Q-\Exc \sigma$, so $K_{\bar Y}\cdot \bar Q=K_Y\cdot Q+1=r_1+\ldots+r_{m-1}+(r_m+1)$, which proves (i). Let $\bar U$ be the $(-1)$-curve of $\bar Q$, let $U'$ be the component of $\bar Q-\bar U-(\sigma^{-1})_*U$ meeting it and let $\sigma'\:(Y,Q)\to (Y',Q')$ be the contraction of $U$. Denoting by $\mu_T(\ )$ the multiplicity of a component in the irreducible decomposition of a divisor $T$ we have $\mu_{\bar Q}(\bar U)=\mu_Q(U)+\mu_{Q'}(U')\leq F_{K_Y\cdot Q+3}+F_{K_{Y'}\cdot Q'+3}\leq F_{K_Y\cdot Q+4}=F_{K_{\bar Y}\cdot {\bar Q}+3}$. For the induction to work it remains to prove the inequality in case $Q=[1]$ and $Q=[3,1,2]$. In the former we have $K\cdot Q+3=2$ and $\mu(U)=1=F_2$ and in the latter $K\cdot Q+3=3$ and $\mu(U)=2=F_3$.
\end{proof}

\blem\label{lem:Q_with_small_KQ}  Let $Q$ be a rational chain which is contractible to a smooth point and contains a unique $(-1)$-curve. Let $k$ denote a non-negative integer. Then $Q$ is of type $(r)=K\cdot Q+1$ and if $r\leq 4$ then $Q$ is one of the following: \benum[(i)]

\item $r=0$: $[(2)_k,1]$.

\item $r=1$: $Q=[(2)_k,3,1,2]$.

\item $r=2$: (iii.1) $[(2)_k,4,1,2,2]$,  (iii.2) $[(2)_k,3,2,1,3]$.

\item $r=3$: (iv.1) $[(2)_k,5,1,2,2,2]$,  (iv.2) $[(2)_k,4,2,1,3,2]$, (iv.3) $[(2)_k,3,3,1,2,3]$, (iv.4) $[(2)_k,3,2,2,1,4]$.

\item $r=4$: (v.1) $[(2)_k,6,1,2,2,2,2]$, (v.2) $[(2)_k,5,2,1,3,2,2]$, (v.3) $[(2)_k,4,3,1,2,3,2]$, (v.4) $[(2)_k,4,2,2,1,4,2]$, (v.5)  $[(2)_k,3,4,1,2,2,3]$, (v.6) $[(2)_k,3,3,2,1,3,3]$,\\ (v.7) $[(2)_k,3,2,3,1,2,4]$ (v.8) $[(2)_k,3,2,2,2,1,5]$.
\eenum
\elem

\begin{proof} As we have seen in the proof of \ref{lem:type}, an inner blowup increases $K\cdot Q$ by $1$. For every chain of type $(r)$ with $r\geq 1$ there are exactly two choices of the center of an inner blowup to produce a chain with a unique $(-1)$-curve of type $(r+1)$. The lemma follows.
\end{proof}

\subsection{Cuspidal curves}\label{ssec:cuspidal} From now on let $\bar E\subseteq \PP^2$ be a rational cuspidal curve and let $\pi\:(X,D)\to (\PP^2,\bar E)$ be the minimal log resolution of singularities. By definition $X$ is a smooth projective surface and $D$ is a simple normal crossing divisor which contains no superfluous $(-1)$-curves, i.e.\ $(-1)$-curves which meet at most two other components of $D$, each at most once. The proper transform of $\bar E$ on $X$ is denoted by $E$ and the minimal weak resolution of singularities ('weak' means that we only require the proper transform of $\bar E$ to be smooth) by $\pi_0\:(X_0,D_0)\to (\PP^2,\bar E)$. Clearly, there exists a morphism $\psi_0\:(X,D)\to (X_0,D_0)$, such that $\pi_0\circ\psi_0=\pi$. Let $q_1,\ldots,q_c$ be the cusps of $\bar E$ and let $Q_j$ be the exceptional reduced divisors of $\pi$ over $q_j$. We have $D-E=Q_1+\ldots+Q_c$, $\wt Q_j=\psi_0(Q_j)$ and $E_0=\psi_0(E)$.

 A cusp of $\bar E$ which is locally analytically isomorphic to the singular point of $x^2=y^{2k+3}$ at $0\in\Spec \C[x,y]$ for some $k\geq 0$ is called \emph{semi-ordinary}. It is \emph{ordinary} in case $k=0$. For a semi-ordinary cusp $Q_j=[(2)_{k},3,1,2]$ and $\wt Q_j=[1,(2)_{k}]$ with the $(-1)$-curve of $\wt Q_j$ being tangent to $E_0$.

\blem[\cite{Palka-Coolidge_Nagata1}, 3.3]\label{cor:I and II} Let $\rho\:(Y,D_Y)\to (\PP^2,\bar E)$ be any weak resolution of singularities such that for each $j$ the divisor $Q_j=\rho^{-1}(q_j)$ contains a unique $(-1)$-curve and $E_Y=(\rho^{-1})_*\bar E$ meets $Q_j$ only in this $(-1)$-curve. Let $(\binom{c_{j,i}}{p_{j,i}})_{i\leq h_j}$ be the sequence of characteristic pairs of $Q_j$. Put $\gamma_Y=-E_Y^2$, $d=\deg \bar E$, $\rho_j=Q_j\cdot E_Y$, $M(q_j)=c_{j,1}+\ds \sum_{i=1}^{h_j}p_{j,i}-1$ and $I(q_j)=\ds \sum_{i=1}^{h_j}c_{j,i}p_{j,i}$. Then \benum[(i)]

\item $\gamma_Y-2+3d=\ds \sum_j \rho_j M(q_j),$

\item $\gamma_Y+d^2=\ds \sum_j \rho_j^2 I(q_j),$

\item $(d-1)(d-2)=\ds \sum_j \rho_j(\rho_j I(q_j)-M(q_j)).$
\eenum \elem

For a proof of the following remark see for instance \cite[3.5]{Palka-Coolidge_Nagata1}.

\brem\label{rmk:ind>half} For every $j$ the contribution of the maximal twigs of $D$ contained in the first branch of $Q_j$ is strictly bigger than $\frac{1}{2}$. \erem

\bprop[\cite{Palka-Coolidge_Nagata1}, 2.4]\label{prop:fundamentals} Assume $\bar E\subset \PP^2$ is a rational cuspidal curve violating the Coolidge-Nagata conjecture. Put $p_2(\PP^2,\bar E)=h^0(2K_X+D)$. Then:
\benum[(i)]
    \item  $\PP^2\setminus \bar E=X\setminus D$ is a $\Q$-acyclic surface of log general type.
    \item $2K_X+E\geq 0$. In particular, $p_2(\PP^2,\bar E)>0$ and $\kappa(K_X+\frac{1}{2}D)\geq 0$.
    \item $\PP^2\setminus \bar E$ contains no curve isomorphic to $\C^1$.
    \item $K_X\cdot (K_X+D)=p_2(\PP^2,\bar E)$.
\eenum\eprop

\bcor\label{cor:fE>=4} Let $\alpha\:\PP^2\map Z$ be a rational map to a smooth surface, such that $E_Z=\alpha_*\bar E\neq 0$. \benum[(i)]
    \item If $f$ is a fiber of a $\PP^1$-fibration of $Z$ then $f\cdot E_Z\geq 4$.
    \item If $f=U_1+U_2+\ldots+U_k$ where $U_i$'s are distinct, $U_1,U_k$ are $(-1)$-curves and $U_2,\ldots,U_{k-1}$ are $(-2)$-curves such that $U_i\cdot U_{i+1}>0$ for each $i<k$ then $f\cdot E_Z\geq 4$.
    \item If $f=U_1+2U_2+U_3$ where $U_1,U_3$ are $(-2)$-curves and $U_2$ is a $(-1)$-curve meeting $U_1$ and $U_3$ then $f\cdot E_Z\geq 4$.
    \item  If $E_Z$ is smooth then $E_Z^2\leq -4$.
\eenum
\ecor

\begin{proof} Let $\ti\alpha\:(\ti Z,E_{\ti Z})\to (Z,E_Z)$, where $E_{\ti Z}$ is the proper transform of $\bar E$, be a resolution of base points of $\alpha$ which dominates the minimal log resolution $(X,D)$. By \ref{prop:fundamentals}(ii) $2K_X+E\geq 0$, so $2K_{\ti Z}+E_{\ti Z}\geq 0$ and hence $2K_Z+E_Z\geq 0$. It follows that for every nef divisor $f$ we have $f\cdot E_Z\geq -2f\cdot K_Z$. For (i) we may assume $f$ is a smooth $0$-curve, hence $f\cdot E_Z\geq -2f\cdot K_Z=4$. In (ii) and (iii) $f$ is nef (note the assumptions do not imply that $f_{red}$ is a chain), which easily gives the inequalities. For (iv), since $E_Z$ is not in the fixed part of $|2K_Z+E_Z|$, we have $0\leq E_Z\cdot (2K_Z+E_Z)=2E_Z\cdot (K_Z+E_Z)-E_Z^2$, so $E_Z^2\leq 4(p_a(E_Z)-1)=-4$.
\end{proof}

\blem[\cite{Palka-Coolidge_Nagata1} 4.6] \label{lem:min(mu)>=4} If $\bar E\subset \PP^2$ violates the Coolidge-Nagata conjecture and the cusps $q_2,\ldots, q_c\in \bar E$ have multiplicity two (equivalently, they are semi-ordinary) then $q_1$ has multiplicity at least four.
\elem

\subsection{Elliptic fibrations}We will need information about fibers of elliptic fibrations.

\blem \cite[Theorem 3.3]{Kumar-Murthy}. \label{lem:elliptic_ruling} Let $E$ be a smooth rational curve on a smooth rational surface $Y$. If $E^2=-4$ and $C$ is a $(-1)$-curve for which $E\cdot C=2$ then $|E+2C|$ induces an elliptic fibration of $Y$. In particular, if $2K_Y+E\sim 0$ then any $(-1)$-curve on $Y$ gives such a fibration. Moreover, in the latter case the fibration has no section and singular fibers other than $E+2C$ consist of $(-2)$-curves. \elem

A fiber of an elliptic fibration is \emph{minimal} if it contains no $(-1)$-curves.

\blem\label{lem:elliptic_fibers} Let $p\:X\to B$ be an elliptic fibration of a smooth projective surface $X$ and let $F$ be a reduction of some singular fiber. \benum[(i)]
    \item If $F$ is minimal and reducible then it  consists of $(-2)$-curves.
    \item If $F$ is minimal but not snc then its snc-minimal resolution either consists of a $(-4)$-curve and a $(-1)$-curve meeting twice or it is a rational tree consisting of a branching $(-1)$-curve and three tips of self-intersections $-d_1,-d_2,-d_3$ with $d_i>0$ and $\sum\frac{1}{d_i}=1$.
    \item All fibers of $p$ are minimal if and only if $K_X^2=0$.
\eenum \elem

\begin{proof} For (i), (iii) and the Kodaira classification of singular fibers see \cite[\S V.7]{BHPV}. (ii) From the classification of singular fibers we know that $F$ is either a nodal or unicuspidal rational curve or a pair of tangent lines or a triple of lines meeting at a common point. We check easily that after resolving singularities of these divisors we get divisors as above.
\end{proof}

\subsection{The log MMP for $(X_0,\frac{1}{2}D_0)$}\label{ssec:logMMP}
Assume $\bar E\subset \PP^2$ violates the Coolidge-Nagata conjecture. By \ref{prop:fundamentals}(i) the complement $\PP^2\setminus \bar E$ is of log general type. In \cite[\S3]{Palka-minimal_models} we studied minimal models related to minimal weak resolutions $\pi_0\:(X_0,D_0)\to (\PP^2,\bar E)$ of such curves. For a detailed discussion in the context of the Coolidge-Nagata conjecture see \cite[\S4]{Palka-Coolidge_Nagata1}. Let us recall some properties of these models. For our purposes the following definitions will be sufficient.

A (rational) chain $S=[2,\ldots,2,3,1,2]$ is a \emph{semi-ordinary ending} of a reduced effective divisor $T$ if $S\cdot (T-S)=1$ and the unique intersection point belongs to the $(-1)$-curve of $S$. Note that $D_0$ has no semi-ordinary endings, because all its $(-1)$-curves are tangent to $E_0$. On the other hand, $D$ has a semi-ordinary ending whenever $\bar E$ has a semi-ordinary cusp. By a line on a complex affine surface we mean any curve isomorphic to $\C^1$. Starting from $(X_0,D_0)$ we will define a sequences of log surfaces $(X_i,D_i)$, $i=0,\ldots,n$. We need some notation.

\bnot\label{def:Delta,Upsilon,Dflat} Let $(X_i,D_i)$ be a pair consisting of a smooth projective surface $X_i$ and a reduced $\Z$-divisor $D_i$ with smooth components and no superfluous $(-1)$-curves, such that $X_i\setminus D_i$ is affine and $D_i$ contains no semi-ordinary endings \benum[(i)]

\item By $\Delta_i$ we denote the sum of all maximal $(-2)$-twigs of $D_i$.

\item Let $\Upsilon_i$ be the sum of $(-1)$-curves $L$ in $D_i$, for which either $\beta_{D_i}(L)=3$ and $L\cdot \Delta_i=1$ or $\beta_{D_i}(L)=2$ and $L$ meets exactly one component of $D_i$.

\item Decompose $\Delta_i$ as $\Delta_i=\Delta^+_i+\Delta^-_i$, where $\Delta^+_i$ consists of these $(-2)$-twigs of $D_i$ which meet $\Upsilon_i$.

\item Put $D_i^\flat=D_i-\Upsilon_i-\Delta^+_i-\Bk_{D_i} \Delta^-_i$.
\eenum\enot

Put $K_i=K_{X_i}$. As for now we have defined only $(X_0,D_0)$. Here the assumptions stated above are satisfied by \ref{prop:fundamentals}. The divisor $\Upsilon_0+\Delta_0^+$ consists of exceptional divisors of $\pi_0$ over semi-ordinary cusps, with $\Upsilon_0$ consisting of the $(-1)$-curves. To define $(X_i,D_i)$ for $i>0$ we proceed as follows.

\blem[\cite{Palka-minimal_models}, 3.5, 4.1] \label{lem:Ai} Let $(X_i,D_i)$ be as in \ref{def:Delta,Upsilon,Dflat}. Assume that $X_i\setminus D_i$ contains no lines and that $\kappa(K_{X_i}+\frac{1}{2}D_i)\geq 0$. Assume also that the components of $\Upsilon_i$ are disjoint and each meets at most one connected component of $\Delta_i^+$. \benum[(i)]
    \item If $2K_i+D_i^\flat$ is not nef then there exists a $(-1)$-curve $A_i\not\subset D_i$, such that \beq\label{eq:Ai} A_i\cdot (\Upsilon_i+\Delta_i^+)=0 \text{\ \ and\ \ } A_i\cdot (D_i-\Delta^-_i)=A_i\cdot\Delta^-_i=1,\eeq and the component of $\Delta^-_i$ meeting $A_i$ is a tip of $\Delta^-_i$.
    \item If $\psi_{i+1}\:(X_i,D_i)\to (X_{i+1},D_{i+1})$ is the composition of successive contractions of superfluous $(-1)$-curves in $D_i+A_i$ and its images then $(X_{i+1},D_{i+1})$  satisfies the above assumptions.
\eenum
\elem

The lemma follows essentially from the fact that if $\alpha_i\:(X_i,D_i)\to (Y_i,D_{Y_i})$ is the contraction of $\Upsilon_i+\Delta_i$ then either $K_{Y_i}+\frac{1}{2}D_{Y_i}$ is nef, and hence $\alpha^*_i(K_{Y_i}+\frac{1}{2}D_{Y_i})=K_i+\frac{1}{2}D_i^\flat$ is nef, or there is a log extremal $(K_{Y_i}+\frac{1}{2}D_{Y_i})$-negative ray not contained in $D_{Y_i}$. In the latter case one takes the lift of the ray for $A_i$. Clearly, the process stops at some point, because contractions decrease the Picard rank. We denote the index $i$ of the first $2K_i+D_i^\flat$ which is nef by $n$ and refer to it as the \emph{length} of the process of minimalization $\psi=\psi_n\circ\ldots\circ\psi_1\:(X_0,\frac{1}{2}D_0)\to (X_n,\frac{1}{2}D_n)$. Note that $X_{i+1}\setminus D_{i+1}$ is an open subset of $X_i\setminus D_i$ with the complement isomorphic to $\C^*$. By construction $\psi$ does not contract $E_0$. We put $E_{i+1}=\psi_{i+1}(E_i)$. Let $\varphi_i\:(X_i',D_i')\to (X_i,D_i)$ be the minimal log resolution. We have $\varphi_0=\psi_0$. We define $\psi_{i+1}'\:(X_i',D_i')\to (X_{i+1}',D_{i+1}')$ as a lift of $\psi_i$. We put $E'_{i+1}=\psi_{i+1}(E_i')$, where $E_0'=E$. We obtain a commutative diagram:

$$ \xymatrix{{} & {(X,D)}\ar@{>}[r]^-{\psi_1'}\ar@{>}[d]^{\psi_0=\varphi_0}\ar@{>}[ld]_-{\ \pi} &{(X_1',D_1')}\ar@{>}[r]^-{\psi_2'\ }\ar@{>}[d]^{\varphi_1} &{\ldots}\ar@{>}[r]^-{\psi_n'\ } &{(X_n',D_n')}\ar@{>}[d]^{\varphi_n}\\
 {(\PP^2,\bar E)} & {(X_0,D_0)}\ar@{>}[r]^-{\psi_1}\ar@{>}[l]^-{\ \pi_0} &{(X_1,D_1)}\ar@{>}[r]^-{\psi_2\ } &{\ldots}\ar@{>}[r]^-{\psi_n\ } &{(X_n,D_n)}}$$

Since $\PP^2\setminus \bar E$ is $\Q$-acyclic, we have $\#D_0=\rho(X_0)$. Then the definition of $\psi_{i+1}$ gives \beq \#D_i=\rho(X_i)+i\label{eq:rho(Xi)+i=|D|}\eeq and, because the contractions in $\psi_{i+1}$'s are inner for $D_i+A_i$,
\beq \label{eq:psi_is_innner} \psi_{i+1}^*(K_{i+1}+D_{i+1})=K_i+D_i+A_i.\eeq

\bdfn We call the pairs $(Y_n,\frac{1}{2}D_{Y_n})$ and $(X_n',\frac{1}{2}D_n')$ constructed above a \emph{minimal model} and respectively an \emph{almost minimal model} of $(X_0,\frac{1}{2}D_0)$. \edfn

Note that at each step there may bo more than one choice for $A_i$, and hence the length ($n$) and the pair $(X_n,D_n)$ refer to some fixed process of minimalization. We simply work with a fixed choice.

The divisor $D_0$, which we will be working with, has smooth components but some of them are tangent to $E_0$. This choice is related to the Kumar-Murthy criterion \ref{prop:fundamentals}(ii) and the special role played by the component $E_0$. We introduce the following numbers controlling the degree of tangency.

\bnot Let $q_1,\ldots, q_c$ be the cusps of $\bar E$.  Assume $j\in \{1,\ldots,c\}$.
\benum[(i)]
    \item Write $\tau_j$ for the number of times $\psi_0$ touches $E$. Equivalently, $\tau_j$ is the number of curves over the cusp $q_j$ contracted by $\psi_0$.
    \item We put $s_j=1$ if $\psi_0$ contracts a twig of $D$ over $q_j$ (which is necessarily a $(-2)$-twig) and $s_j=0$ otherwise.
    \item Put $\tau_j^*=\tau_j-s_j-1$, $\tau^*=\sum_{j=1}^c \tau_j^*$ and $s=\sum_{j=1}^c s_j$.
    \item For $k=0,1$ let $n_k$ be the number of contracted $A_i$'s, i.e.\ the $(-1)$-curves defined above, for which $A_i\cdot E_i=k$.
\eenum
\enot

Because $\tau_j\geq 2$, we have $\tau_j^*\geq 0$. By \eqref{eq:Ai} $A_i\cdot E_i\leq 1$, so $n=n_0+n_1$. Elementary computations (\cite[4.3]{Palka-minimal_models}) give \beq\label{eq:Kn(Kn+Dn)} K_n\cdot (K_n+D_n)=p_2(\PP^2,\bar E)-c-\tau^*-n \eeq and \beq\label{eq:En(Kn+Dn)} E_n\cdot (K_n+D_n)=2c-2+\tau^*+n_1.\eeq

\vskip 0.5cm
In principle, the length of the minimalization process of $(X_0,D_0)$ could be arbitrarily big. The following result proved in Part I of the article shows in particular that this is not so.

\bthm[\cite{Palka-Coolidge_Nagata1}, 1.1, 1.2, 1.3, 5.4, 7.3]\label{thm:CN1} Let $(X_0,D_0)\to (\PP^2,\bar E)$ be the minimal weak resolution of a rational cuspidal curve  violating the Coolidge-Nagata conjecture. Put $p_2(\PP^2,\bar E)=h^0(2K_X+D)$. Then \  \benum[(a)]
    \item the length of the minimalization process of $(X_0,\frac{1}{2}D_0)$ is at most $1$, i.e.\ $n\leq 1$.
    \item Either $p_2(\PP^2,\bar E)=3$ or $p_2(\PP^2,\bar E)=4$, $n=0$ and $\bar E$ has only one cusp.
    \item $\bar E$ has at most two cusps.
    \item $|K_n\cdot (K_n+E_n)|\leq p_2(\PP^2,\bar E)-2$.
\eenum\ethm

The core of the proof of \ref{thm:CN1} is the inequality \beq\label{eq:basic_inequality} (2K_n+E_n)\cdot (2K_n+D_n^\flat)\geq 0,\eeq which follows from the fact that $2K_n+E_n$ is effective (due to the Kumar-Murthy criterion \ref{prop:fundamentals}(ii)) and $2K_n+D_n^\flat$ is nef by the construction of $(X_n,D_n)$.

\med 
\section{Process of length $n=1$.}\label{sec:n=1} We keep the notation from the previous section. From now on we assume, for a contradiction, that $\bar E\subseteq \PP^2$ violates the Coolidge-Nagata conjecture, i.e.\ it is a rational cuspidal curve which is not Cremona equivalent to line. By \ref{prop:fundamentals} $\PP^2\setminus \bar E$ is of log general type and $2K_X+E\geq 0$.  Recall that $\pi_0\:(X_0,D_0)\to (\PP^2,\bar E)$ is the minimal weak resolution of singularities and $n$ is the length of the process of minimalization of $(X_0,\frac{1}{2}D_0)$ as defined in Section \ref{sec:Prelim}. In this section we prove the following theorem.

\bthm\label{prop:n=0} If $\bar E\subseteq \PP^2$ violates the Coolidge-Nagata conjecture then $2K_0+\frac{1}{2}D_0^\flat$ is numerically effective, i.e.\ $n=0$.
\ethm

By \ref{thm:CN1}(a) in the process of minimalization of $(X_0,\frac{1}{2}D_0)$ at most one curve not contained in $D_0$ is contracted, i.e.\ $n\leq 1$. Up to the end of this section we suppose therefore that $n=1$. In particular $\psi_1=\psi$ and $\psi_1'=\psi'$. We put $A=A_0$ and we denote the unique $(-2)$-twig of $D_0$ met by $A$ by $\Delta_A$. The component of $\Delta_A$ meeting $A$ (and its proper transform on $X$) is denoted by $T_A$ (see Fig.\ \ref{fig:1}). By \ref{lem:Ai} $T_A$ is a tip of $\Delta_A$, but not necessarily a tip of $D_0$. The proper transform of $\Delta_A$ on $X$ (which is also a maximal $(-2)$-twig) will be also denoted by $\Delta_A$. Since $\psi_0\:X\to X_0$ does not touch $A$, its proper transform on $X$, which we denote by $A'$, is also a $(-1)$-curve and it meets exactly one $(-2)$-twig of $D$.
\begin{figure}[h]\centering\includegraphics[scale=0.5]{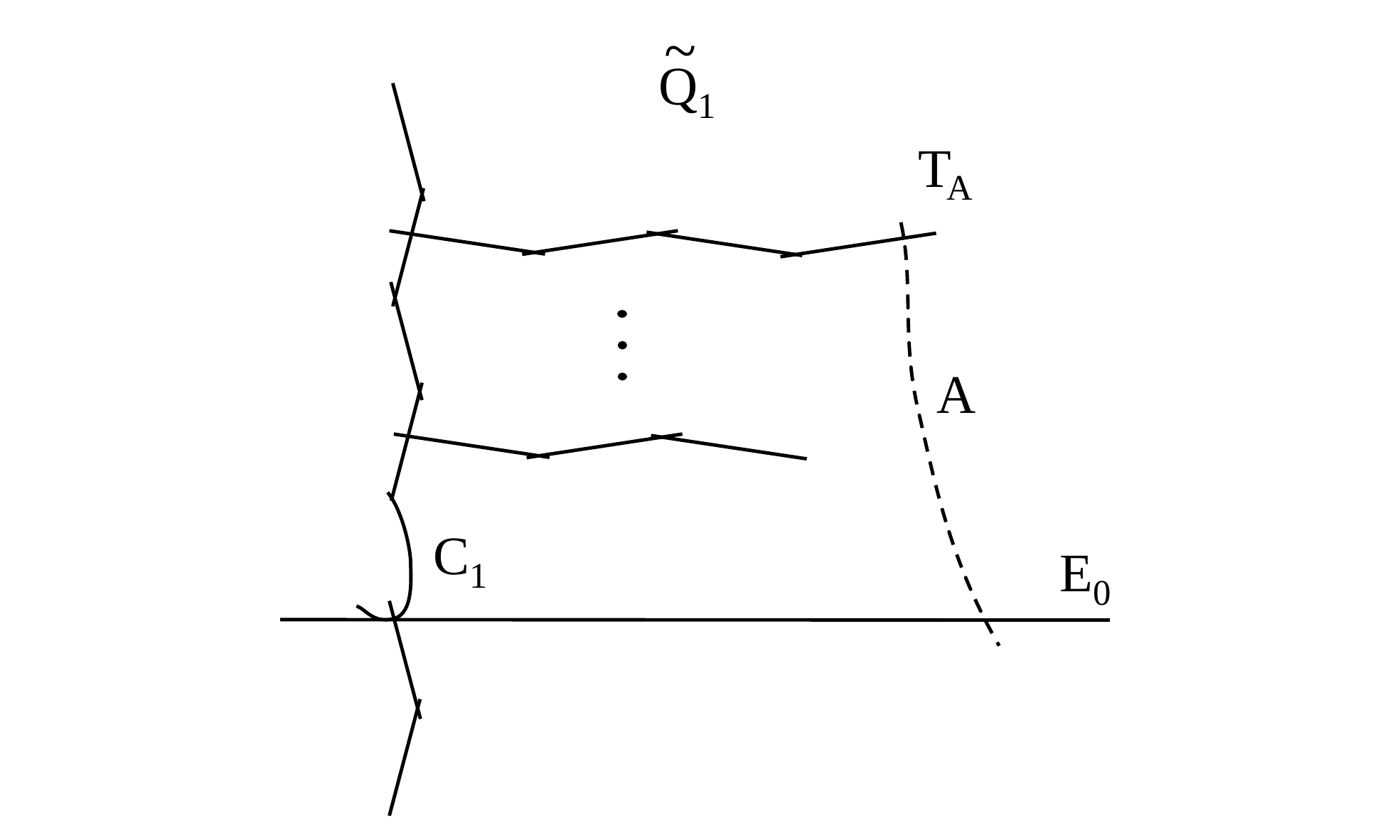}\caption{The divisor $D_0=\wt Q_1+\ldots+\wt Q_c+E_0$ on $X_0$. Case $c=A\cdot E_0=1$.}  \label{fig:1}\end{figure}

We have now the following diagram, with $2K_1+D_1^\flat$ being numerically effective.

$$ \xymatrix{{} & {(X,D)}\ar@{>}[r]^-{\psi'}\ar@{>}[d]^{\psi_0}\ar@{>}[ld]_-{\ \pi} &{(X_1',D_1')}\ar@{>}[d]^{\varphi_1} \\
 {(\PP^2,\bar E)} & {(X_0,D_0)}\ar@{>}[r]^-{\psi}\ar@{>}[l]^-{\ \pi_0} &{(X_1,D_1)}}$$

By the definition of $\psi$ we see that all components of $D_0$ contracted by $\psi$ are contained in maximal twigs of $D_0$ meeting $A$. We say that $\psi$ is of \emph{type II} if it contracts both components of $D_0$ meeting $A$; otherwise it is of \emph{type I}. Contractions of type II are difficult to analyze, because in principle they may contract both maximal twigs met by $A$, including components of very negative weights, in which case it is harder to recover them having only the information about the minimal model $(X_1,D_1)$.

\bnot Let $q_1,\ldots, q_c$ be the cusps of $\bar E$.  Assume $j\in \{1,\ldots,c\}$. \benum[(i)]
    \item Denote by $C_j$ the $(-1)$-curve of $D_0-E_0$ over $q_j\in \bar E$. Put $\cal C=C_1+\ldots+C_c$.
    \item Put $C_j'=\psi_*C_j$ ($C_j'\neq 0$ as every $C_j$ is tangent to $E_0$).
    \item Let $\cal C_+$ and $\cal C_{exc}$ be the sums of these $C_j'$'s whose self-intersection is non-negative or stays equal to $(-1)$ respectively.
    \item Denote by $\cal L$ the sum of (non-superfluous) $(-1)$-curves in $D_1$ not contained in $\cal C_{exc}$.
    \item Put $$R_1=D_1-E_1-\cal C_+-\cal C_{exc}-\cal L.$$
\eenum
\enot

Note that $C_j'$ is a component of $\cal C_+$ whenever $\psi$ touches $C_j$. We refer to the components of $\cal L$ as $(-1)$-curves \emph{created by $\psi$}. We will see that for $n\leq 1$ there is in fact at most one component in $\cal L$. Let us recall some basic properties of the process of minimalization.

\blem\label{lem:mmp_basics} Let $(X_i,D_i)$ and $(X_i',D_i')$, $i=0,1$ be as above. \benum[(i)]
    \item $D_1'$ is snc-minimal,
    \item For every component $U$ of $D_i-\Delta_i$ we have $U\cdot \Delta_i\leq 1$.
    \item For every component $U$ of $D_0-E_0$ $$\psi(U)\cdot E_1\leq U\cdot E_0+1.$$ If the equality holds then $\psi$ is of type $I$ and $A\cdot E_0=1$. Moreover, $\psi$ touches $U$ exactly once and either $U$ is the component of $\Delta_0^-$ met by $A$ or there is a unique connected component of $\Delta_0^-$ meeting $U$ and this component is contracted by $\psi$.
    \item $\ind(D_i')\leq 5-p_2(\PP^2,\bar E)-i$,
    \item $\psi$ creates at most one $(-1)$-curve, i.e.\ $\#\cal L\leq 1$.
    \item Components of $R_1$ intersect non-negatively with $K_1$.
    \item If $(C_j')^2\geq 0$ then $K_1\cdot C_j'+\tau^*_j\geq 0$.
    \item If $\cal L$ (as above) is a component of $\Upsilon_1$ then $(\psi^{-1})_*\cal L\cdot E_0=0$ and $\cal L\cdot E_1=A\cdot E_0\leq 1$.
\eenum
\elem

\begin{proof} For (i) (which relies on the fact that $2K_X+D\geq 0$) and (ii)-(iii) see \cite{Palka-minimal_models} 3.7 and 4.1(vi)-(vii) respectively. For (iv), which follows mostly from \ref{lem:BMY} and \ref{prop:fundamentals}(iv) see \cite[4.2(iv)]{Palka-Coolidge_Nagata1}. For (v)-(vii) see 4.3(ii),(v)-(vi) loc.\ cit.\ (in all cases the fact that $2K_X+E\geq 0$ is used). For (viii) put $\cal L'=(\psi^{-1})_*\cal L$ and suppose that $\cal L'\cdot E_0>0$. Since $\cal L'$ is not one of $C_j$'s, it meets, say, $C_1$ at its intersection point with $E_0$, so $\cal L'\cdot E_0=\cal L'\cdot C_1=1$. If $\beta_{D_1}(\cal L)=2$ then $A\cdot \cal L'=0$ and $\cal L'$ meets no twig of $D_0$, hence $\psi$ does not touch $\cal L'$, which is impossible. Therefore, $\beta_{D_1}(\cal L)=3$, so since $\cal L$ is a component of $\Upsilon_1$, it meets some connected component $T$ of $\Delta_1^+$. Because $\psi$ is inner for $D_0+A$, the point $\psi(A)$ does not lie on $T$, because otherwise $D_0+A$ would not be connected. It follows that $\psi^*T$ is a connected component of $\Delta_0^-$ meeting $\cal L'$ and not touched by $\psi$. But then $\beta_{D_0}(\cal L')\geq 3$ and, since $\psi$ does not contract $C_1$ or $E_0$, it follows that $\psi$ does not touch $\cal L'$; a contradiction. Thus, $\cal L'\cdot E_0=0$. By (iii) $\cal L\cdot E_1\leq 1$ and equality holds if and only if $A\cdot E_0=1$.
\end{proof}

\bnot Put $\eta=\#\Upsilon_1-\#\Upsilon_0$. Note that $\#\Upsilon_0$ is the number of semi-ordinary cusps of $\bar E$ and that $\eta\leq 1$. If $\eta=1$ then $\Upsilon_1=\psi_*\Upsilon_0+\cal L$ and by \ref{lem:mmp_basics}(viii) $\cal L\cdot E_1=(\psi^{-1})_*\cal L\cdot E_0\leq 1$. Define $\eta_0$ to be $1$ if $\eta=1$ and $\cal L\cdot E_1=0$, and to be $0$ otherwise. Put $\eta_1=\eta-\eta_0$. Put $\gamma_i=-E_i^2$. By \ref{cor:fE>=4}(iv) $\gamma_0\geq\gamma_1\geq 4$.
\enot

\med Because $n=1$, we obtain the following corollary.

\bcor\label{cor:n=1_basics} Put $\zeta=K_1\cdot (K_1+E_1)$. We have: \benum[(i)]
    \item $p_2(\PP^2,\bar E)=3$, and hence $\ind(D)\leq 2$ and $\ind(D_1')\leq 1$,
    \item $\zeta\in\{-1,0,1\},$
    \item $\gamma_1+\tau^*+(n_0-\eta_0)\leq 5+2\zeta+\eta$.
    \item $\ds \sum_{j:(C_j')^2\geq 0}(K_1\cdot C_j'+\tau_j^*)+\#\cal C_++\ds\sum_{j:(C_j')^2=-1}\tau_j^*+K_1\cdot R_1=2+\#\cal L-\zeta,$
\eenum
\ecor

\begin{proof}(i) follows from \ref{thm:CN1}(b) and \ref{lem:mmp_basics}(iv). (ii) follows from \ref{thm:CN1}(d) and (a).

(iii) By \eqref{eq:basic_inequality} $$0\leq (2K_1+D_1^\flat)\cdot(2K_1+E_1)=(2K_1+D_1)\cdot(2K_1+E_1)-\Upsilon_1\cdot(2K_1+E_1).$$ We have $$\Upsilon_1\cdot (2K_1+E_1) =\Upsilon_0\cdot(2K_0+E_0)-2\eta_0-\eta_1=-2\eta_0-\eta_1$$ and $$(2K_1+D_1)\cdot(2K_1+E_1)=2\zeta-K_1\cdot E_1+2K_1\cdot(K_1+D_1)+E_1\cdot (K_1+D_1).$$ By \eqref{eq:Kn(Kn+Dn)} and \eqref{eq:En(Kn+Dn)} the latter expression equals $2\zeta-\gamma_1+2p_2(\PP^2,\bar E)-\tau^*-2+n_1$.

 (iv) We have $D_1-E_1=R_1+\cal C_++\cal C_{exc}+\cal L$ and by \eqref{eq:Kn(Kn+Dn)} $K_1\cdot (D_1-E_1)=2-\zeta-\tau^*-c$. Since $K_1\cdot \cal C_{exc}=-(c-\#\cal C_+)$ and $K_1\cdot \cal L=-\#\cal L$, (iv) follows.
\end{proof}

Note that all terms on the left hand sides of (iii) and (iv) are non-negative and right hand sides are strongly bounded from above. As above, put $\zeta=K_1\cdot (K_1+E_1)$.

\blem\label{lem:zeta>=0} $\zeta\in \{0,1\}$. \elem

\begin{proof} Suppose $\zeta=-1$. By \ref{cor:n=1_basics}(iii) we have $\gamma_1+\tau^*+(n_0-\eta_0)\leq 3+\eta$, so since $\gamma_1\geq 4$, it follows that $\gamma_1=4$, $\tau_*=0$, $\eta=1$ (hence $\#\cal L=1$) and $n_0=\eta_0$.
Now \ref{lem:mmp_basics}(vii) implies that $\cal C_+=0$, because otherwise for some $j$ we would have $K_1\cdot C_j'+\tau_j^*\leq \tau_j^*-2<0$. Then \ref{cor:n=1_basics}(iv) reads as $K_1\cdot R_1=4$. By \ref{cor:n=1_basics}(i) $\ind(D_1')\leq 1$. Since $\tau_*=0$, we have $s_j=1$ and $\tau_j=2$ for every $j$, so the maximal twigs of $D$ contracted by $\psi_0$ are $(-2)$-twigs, hence each of them contributes to $\ind(D_1')$ at least $\frac{1}{2}$. If $\cal L\cdot E_1=1$ then making a contraction $\varphi\:X_1\to Z$ of $\cal L$ we get $(\varphi_*E_1)^2=E_1^2+1=-3$, which is impossible by \ref{cor:fE>=4}(iv). Thus $\cal L\cdot E_1=0$ and hence $A\cdot E_0=0$. Note also that if $c=2$ then the contribution of twigs of $D_1'$ contracted by $\varphi_1$ equals $1$, so since $\ind(D_1')\leq 1$, $D_1'$ has no other twigs, which implies that both $\wt Q_j$ are chains. However, if both $\wt Q_j$ are chains then, because of $\tau^*=0$, they are both part of $\Upsilon_0+\Delta_0^+$, which is impossible for $n=1$. Therefore, $c=1$. We infer that $E$ and $E_1'$ are $(-6)$-curves.

We have $\ind(D_1')\leq 1$. Because $E_1'$ is a tip of $D_1'$, this implies that $D_1'$ can have at most one $(-2)$-tip, so $\Delta_1=0$. Since $\cal L$ is a component of $\Upsilon_1$, $\beta_{D_1}(\cal L)=2$ and $\cal L$ meets exactly one component of $D_1-\cal L$, hence $A$ meets only, say, $\wt Q_1$. Let $\ind'$ be the contribution to $\ind(D_1')$ of twigs of $D_1$ (equivalently, of twigs of $D_1'$ not contracted by $\psi_0$ and other than $E_1$). Clearly, $\ind'\leq 1-\frac{1}{2}-\frac{1}{6}=\frac{1}{3}$. Since $\Delta_1=0$, the tips of corresponding twigs have self-intersections at most $(-3)$, so in fact by \ref{lem:ind>=} there can be at most one twig in $D_1$ and it is a $(-3)$-curve. Suppose there is one. Let $U\subset D_0$ be the branching component of $D_0$ met by the proper transform of this twig. The divisor $D_0-E_0$ is a fork (has three maximal twigs and one branching component) and after the contraction of the maximal twig having $C_1'$ as a tip it becomes $[3,1,a_1,\ldots,a_k]$ for some $a_i\geq 2$ and $k\geq 1$. Since the latter chain contacts to a smooth point, we have $[a_1,\ldots,a_k]=[2,3,(2)_{k-2}]$. By \ref{lem:mmp_basics}(v)-(vi) $A$ meets the tip of the latter twig and its $(-3)$-curve. Then $k=3$ and $K_1\cdot R_1=1$; a contradiction.

Therefore $D_1$ has no twigs, hence $D_0$ has exactly one twig, so $\wt Q_1$ is a chain. Then $\wt Q_1=[1,(2)_k]$ for some $k\geq 0$, so $\Delta_0^-=0$ and hence $n=0$; a contradiction.
\end{proof}

\bprop\label{prop:A0E0=0} $A\cdot E_0=0$. \eprop

\begin{proof} Suppose $A\cdot E_0=1$. In this case, since $E_0$ is not contracted, $\psi$ is of type I, and it is easy to recover $D$ given the information on $D_1$. Indeed, because all centers of blowups constituting $\psi\circ \psi_0\:X\to X_1$ belong to the proper transforms of $E$, we have $(\psi\circ \psi_0)^*(K_1+E_1)=K+E$, so $K\cdot (K+E)=K_1\cdot (K_1+E_1)=\zeta$. Then $K\cdot (D-E)=K\cdot (K+D)-K\cdot (K+E)=p_2(\PP^2,\bar E)-\zeta$, so $K\cdot (D-E)=3-\zeta,$ and hence \beq\label{eq:KR} K\cdot R=3+c-\zeta,\eeq where $R$ is $D-E$ with the $(-1)$-curves (there is exactly one over each cusp) subtracted. The components of $R$ intersect non-negatively with $K$, so this is a very restrictive condition on $R$, because we have already bounded $c$ and $\zeta$. We have $\#D_1=\rho(X_1)+1=11-K_1^2=11-(\zeta-K_1\cdot E_1)=9+\gamma_1-\zeta$, hence \beq\label{eq:b2(D1)}\#D_1=9-\zeta+\gamma_1\geq 13-\zeta.\eeq

Since $\psi$ is of type I, either $T_A$ is a tip of $D_0$ and then $\psi$ is the contraction of $\Delta_A+A$ or $T_A$ is not a tip of $D_0$ and then $\psi$ is simply the contraction of $A$. The remaining part of the proof is quite long, but essentially it boils down to a repeated usage of \ref{cor:n=1_basics} and \ref{cor:fE>=4}.

\bcl If $\eta\neq 0$ then $\cal L$ meets $\Delta_1^+$ and $\psi$ contracts only $A$. \ecl

\begin{proof} If $T_A$, the component of $\Delta_A\subset \Delta_0$ met by $A$, is not a tip of $D_0$ then $\psi$ is the contraction of $A$ and hence $\psi_*(\Delta_A-T_A)$ is a part of $\Delta_1^+$, so we are done. Assume $T_A$ is a tip of $D_0$. Since $A\cdot E_0=1$, $\psi$ is the contraction of $A+\Delta_A$. By \ref{lem:mmp_basics}(ii) the unique component of $D_0-\Delta_A$ meeting $\Delta_A$ meets no other connected component of $\Delta_0$. Thus its image does not meet $\Delta_1$, hence $\eta=0$; a contradiction.
\end{proof}

 Let $U_j$ be the $(-1)$-curve of $Q_j$ and let $\ind_j$ be the contribution of the maximal twigs of $D$ contained in $Q_j$ to $\ind(D)$. By \ref{rmk:ind>half} $\ind_j>\frac{1}{2}$. Clearly, $\ind(D)\geq \sum \ind_j$ and the equality holds if $c\geq 2$ (if $c=1$ then $E$ is a maximal twig of $D$).  The contribution of the unique twig of $D$ meeting $A$ and other than $E$ in case $c=1$ is denoted by $\ind_A$ (this twig contains $T_A$ and $\Delta_A$). Since $A\cdot E=1$, we have \beq \sum_{j=1}^c \ind_j-\ind_A\leq \ind(D_1')\leq 1.\label{eq:ind_AE=1}\eeq

\bcl $c=1$. \ecl

\begin{proof} By \ref{thm:CN1}(c) $c\leq 2$.  Assume $c=2$. In case $\eta\neq 0$ the contributions of the connected component of $\Delta_1^+$ meeting $\cal L$ and the $Q_i$ not meeting $A$ are at least $\frac{1}{2}$ and more than $\frac{1}{2}$ respectively, hence $\ind(D_1')>1$, which is impossible. Thus $\eta=0$ and hence $T_A$ is a tip of $D$ and by \ref{cor:n=1_basics}(iii) $$\gamma_1+\tau^*\leq 5+2\zeta.$$ Let $(r_1,\ldots,r_{m_1})$ and $(w_1,\ldots,w_{m_2})$ be the types of $Q_1$ and $Q_2$ as defined in Subsection \ref{ssec:HN_and_types}. Put $r=\sum r_i$ and $w=\sum w_i$.

Suppose $\min(r,w)=1$, say $w=1$. Then $Q_2=[(2_k),3,1,2]$, so $A$ does not meet $Q_2$ and $\ind_2\geq \frac{5}{6}$. We obtain $\ind_1-\ind_A\leq \frac{1}{6}$. There is at least one twig contributing to $\ind_1-\ind_A$, because $A$ meets only one twig of $D$. By \ref{lem:ind>=} it contains a $\leq(-6)$ -tip, whose intersection with $K$ is at least $4$. On the other hand, $K\cdot (Q_1-U_1)=K\cdot R-1=4-\zeta\leq 4$ by \eqref{eq:KR}, hence $\zeta=0$ and $Q_1-U_1$ consists of a $(-6)$-tip and some number of $(-2)$-curves. This is possible only if $Q_1=[6,1,(2)_4]$, so $\tau^*\geq\tau_1^*=3$. But $\gamma_1\geq 4$, so we get a contradiction with the inequality above.

Suppose $\zeta=0$. We get $\tau^*\leq 5-\gamma_1\leq 1$. However, if $\tau^*=0$ then $s_1=s_2=1$, so the contribution of twigs of $D_1'$ contracted by $\varphi_1$ is at least $1$ and, since $A\cdot E_0=1$, there is at least one more twig contributing to $\ind(D_1')$, which again contradicts the inequality $\ind(D_1')\leq 1$. Thus $\tau^*=1$ and hence $\gamma_1=4$. By \ref{lem:type} and \eqref{eq:KR} $r+w=K\cdot R=5$. We may assume $A\cdot Q_2=0$. Then $A\cdot Q_1=1$, so in particular $q_1\in \bar E$ is not a semi-ordinary cusp. We claim $Q_1$ contains no other $(-2)$-twig than $\Delta_A$. Indeed, otherwise the contributions to $\ind(D_1')$ of the other $(-2)$-twig and the twigs in $Q_2$ are at least $\frac{1}{2}$ and more than $\frac{1}{2}$ respectively, which is a contradiction. Two corollaries follow. Firstly, $w=2$. Indeed, if $r=2$ then, since $q_1$ is not a semi-ordinary cusp, $Q_1$ is branched, hence of type $(1,1)$, so it contains at least two $(-2)$-tips; a contradiction. Secondly, $\tau_2^*=0$. Indeed, otherwise $\tau_1^*=0$ and hence $s_1=1$, which means that there is a $(-2)$-twig in $Q_1$ other than $\Delta_A$. Therefore, $w=2$ and $\tau_2^*=0$. By \ref{lem:Q_with_small_KQ} $Q_2$ is not a chain, so it is of type $(1,1)$. Then it contains at least two $(-2)$-tips, so $\ind(D_1')>\ind_2\geq 1$; a contradiction.

Thus $\zeta=1$ and $r,w\geq 2$. By \ref{lem:type} and \eqref{eq:KR} $r+w=K\cdot R=4$, so $r=w=2$. Suppose, say, $Q_1$ is branched. Then it is of type $(1,1)$ so the maximal twigs of $D$ contained in $Q_1$ are $T_1=[2]$, $T_2=[2]$ and $T_3=[(2)_k,3]$ for some $k\geq 0$. We have $\ind_1\geq \frac{1}{2}+\frac{1}{2}+(1-\frac{2}{2k+3})\geq \frac{4}{3}$. Since $\ind(D_1')\leq 1$, $A$ meets $Q_1$. Then $1>\ind(D_1')-\ind_2=\ind_1-\ind_A$, so $A$ meets $T_1+T_2$ and $k=0$. We obtain $\ind_1-\ind_A=\frac{5}{6}$, which gives $\ind_2\leq \frac{1}{6}$. In particular, $Q_2$ is not branched. By \ref{lem:Q_with_small_KQ}(iii) $\ind_2>\frac{1}{3}$; a contradiction.
\end{proof}

\bcl $\zeta=0$. \ecl

\begin{proof} Suppose $\zeta=1$. We have now $\ind_1-\ind_A\leq \ind(D_1')\leq 1$ and $K\cdot R=3$. If $Q_1$ is of type $(1,1,1)$ then it has four maximal twigs and three of them are $(-2)$-tips, at least two of which are not met by $A$, so $\ind_1-\ind_A>\frac{1}{2}+\frac{1}{2}=1$, which is impossible.

Suppose $Q_1$ is of type $(1,2)$. Then the maximal twigs of $D$ contained in the first branch of $Q_1$ are $T_1=[2]$ and $T_2=[(2)_k,3]$ for some $k\geq 0$ and by \ref{lem:Q_with_small_KQ}(iii) the one contained in the second branch of $Q_1$ is $T_3=[2,2]$ or $T_3=[3]$. In both cases we have $A\cdot T_3=0$, because if $T_3=[2,2]$ then $T_3$ is contracted by $\varphi_1$.

Consider the case $T_3=[2,2]$. Since $\ind_1-\ind_A\leq 1$, we see that $A$ meets $T_1$ and $k=0$. The second branch of $Q_1$ is $[2,2,1,4,(2)_u]$ for some $u\geq -1$. Here and later we use the convention $[a_1,\ldots,a_{n-1},a_n,(2)_{-1}]=[a_1,\ldots,a_{n-1}]$. Then $\#D_1=u+4$, so by \eqref{eq:b2(D1)} $u=4+\gamma_1\geq 8$. We have $D_1-E_1=[3,1,(2)_u,1]$ and $E_1$ meets this divisor once transversally in the middle $(-1)$-curve and once (with tangency index $3$) in the $(-1)$-tip ($C_1'$). Let $\alpha\:X_1\to Z$ be the contraction of the subchain consisting of the middle $(-1)$-curve and two $(-2)$-curves. Then $f=\alpha_*T_2$ is a $0$-curve with $f\cdot \alpha_*E_1=3$, which contradicts \ref{cor:fE>=4}(i).

Consider the case $T_3=[3]$. We have $A\cdot T_2=0$. Indeed, otherwise $T_A$ (the component of $\Delta_A\subset T_2$ meeting $A$) is not a tip, because $\pi(A)\subset \PP^2$ cannot be a $0$-curve, hence $\ind(D_1')>1$, which is false. Thus $A$ meets $T_1$. Now \eqref{eq:ind_AE=1} gives $k\leq 1$. The second branch of $Q_1$ is $[3,1,2,3,(2)_u]$ for some $u\geq -1$, so $\#D_1=k+u+5$. By \eqref{eq:b2(D1)} $u=3+\gamma_1-k\geq 6$. We have $D_1-E_1=[(2)_k,3,1,(2)_u,2,1]$ and $E_1$ meets this divisor once transversally in the middle $(-1)$-curve and once in the common point of the $(-1)$-tip with the $(-2)$-curve. As before, let $\alpha\:X_1\to Z$ be the contraction of the subchain consisting of the middle $(-1)$-curve and two $(-2)$-curves. The image of the $(-3)$-curve contained in $T_2$ is a $0$-curve whose intersection with $\alpha_*E_1$ is $3$; a contradiction with \ref{cor:fE>=4}(i).

Suppose $Q_1$ is of type $(2,1)$. The maximal twig of $D$ contained in the second branch of $Q_1$ is $T_3=[2]$. Since $T_3$ is contracted by $\varphi_1$, it does not meet $A$. By \ref{lem:Q_with_small_KQ} the maximal twigs contained in the first branch are either $T_1=[2,2]$ and $T_2=[(2)_k],4]$ or $T_1=[2]$ and $T_2=[(2)_k,3,2]$. As before, we see that since $\pi(A)\subset \PP^2$ is not a $0$-curve, $A$ does not meet the tip of $D$ contained in $T_2$. From \eqref{eq:ind_AE=1} we infer that $T_2\cdot A=0$, $k=0$ and $A$ meets the tip of $D$ contained in $T_1$. Then $D_1-E_1$ is $[4,1,(2)_u,1]$, $u\geq 0$ in the first case and $[3,2,1,(2)_u,1]$, $u\geq 0$ in the second case. In both cases $E_1\cdot (D_1-E_1)=3$. Taking the subchain $f=[1,(2)_u,1]$ we have $f\cdot E_1\leq 3$. But $f$ is a total transform of a $0$-curve, so it is a fiber of a $\PP^1$-fibration of $X_1$, which is again a contradiction with \ref{cor:fE>=4}(i).

Finally, suppose $Q_1$ is a chain. Since $K\cdot R=3$, $Q_1$ is as in \ref{lem:Q_with_small_KQ}(iv). Denote the maximal twig containing $[(2)_k]$ by $T_1$ and the second one by $T_2$. Suppose $A$ meets $T_2$. Then $A$ meets a $(-2)$-twig of $D_0$ contained in $\varphi(T_2)$, which is possible only if $Q_1$ is as in (iv.2). But then $C_1'$ (which is the image of the $(-3)$-curve on $X_1$) is a $0$-curve with $C_1'\cdot E_1=3$; a contradiction. Thus $A$ meets $T_1$. It does not meet the tip of $T_1$, because otherwise $\pi(A)\subset \PP^2$ would be a $0$-curve. Then $k\geq 2$, $\eta=1$ and by \ref{lem:Ai}(i) $A$ meets the tip of $[(2)_k]$ which is not the tip of $D$. Then $\Delta_1^+$ has $k-1\geq 1$ components. Since $\ind(D_1')\leq 1$, we check that $Q_1$ is of type (iv.3) or (iv.4) and that $k\leq 4$, which in these two cases gives $\#D_0\leq \#D=k+6\leq 10$. Because $\psi$ contracts only $A$, \eqref{eq:b2(D1)} gives $\#D_0=\#D_1=8+\gamma_1\geq 12$; a contradiction.
\end{proof}

By \eqref{eq:KR} and \ref{cor:n=1_basics}(iii) we infer that $K\cdot R=4$ and $\gamma_1+\tau_1^*\leq 5+\eta$.

\bcl $Q_1$ has at most one branching component. \ecl

\begin{proof} Let $(r_1,\ldots,r_m)$ be the type of $Q_1$. By \ref{lem:type} $\sum r_i=K\cdot R=4$, so $Q_1$ has at most four branches. If it has four then it is of type $(1,1,1,1)$, so has a least four $(-2)$-tips, hence $\ind_1-\ind_A\geq \frac{3}{2}$, which is impossible by \eqref{eq:ind_AE=1}. Thus $Q_1$ has at most three branches.

Suppose it has three. Then it is of type $(1,1,2)$ or $(1,2,1)$ or $(2,1,1)$. Note that every branch of $Q_1$ with $r_i=1$ contains a maximal twig of $D$ which is an irreducible $(-2)$-curve. Since $A$ meets exactly one maximal twig of $D$ contained in $Q_1$, $A$ meets one of these $(-2)$-curves, because otherwise $\ind_1-\ind_A>\frac{1}{2}+\frac{1}{2}$, which would contradict \eqref{eq:ind_AE=1}. Thus, $\ind_A=\frac{1}{2}$ and $\ind_1\leq \frac{3}{2}$. By \ref{lem:Q_with_small_KQ}(iii) for the type $(2,1,1)$ the first branch of $Q_1$ contains either the twigs $[(2)_k,4]$ and $[2,2]$ or $[(2)_k,3,2]$ and $[3]$ and the remaining maximal twigs contained in $Q_1$ are $[2]$ and $[2]$. Then we compute $\ind_1\geq \frac{1}{4}+\frac{2}{3}+1$ in the first case and $\ind_1\geq \frac{2}{5}+\frac{1}{3}+1$ in the second case; a contradiction. For the types $(1,1,2)$ and $(1,2,1)$ we check that $\ind_1\geq \frac{1}{3}+1+\frac{1}{3}$; a contradiction. Therefore, $Q_1$ has at most two branches, hence at most one branching component.
\end{proof}

\bcl $Q_1$ has a unique branching component. \ecl

\begin{proof} Assume $Q_1$ is a chain. Then it is as in \ref{lem:Q_with_small_KQ}(v). As before, we denote the maximal twig containing $[(2)_k]$ by $T_1$ and the other one by $T_2$.

Suppose $A$ meets $T_2$. Then $A$ meets the twig $\Delta_A$ contained in $\psi_0(T_2)$. This is possible only in cases (v.2), (v.3) and (v.4). In cases (v.2) and (v.3) we have $\tau_1=2$ and then either $C_1'$ is a $0$-curve with $C_1'\cdot E_1=3$ or (this is possible only for (v.2)) we have a chain $f=[1,1]$ containing $C_1'$ with $f\cdot E_1=3$. By \ref{cor:fE>=4} it follows that $Q_1$ is as in (v.4). In this case $\eta=0$ and $\tau_1^*=2$, so $\gamma_1\leq 5-\tau_1^*=3$; a contradiction.

Therefore, $A$ meets $T_1$. Since $\pi(A)\subset \PP^2$ is not a $0$-curve, $A$ does not meet the tip of $D$ contained in $T_1$, which implies that $k\geq 2$, $\eta=1$ and $\Delta_1^+=[(k-1)]$. Because $\ind(D_1')\leq 1$, the contribution of $T_2$ is at most $\frac{1}{k}\leq \frac{1}{2}$, so $Q_1$ is not as in (v.1), which gives $\tau_1^*\geq 1$, and hence $\gamma_1\leq 5$. Let $\alpha\:X_1\to Z$ be the contraction of $\cal L+\Delta_1^+$. Then $(\alpha_*E_1)^2\geq E_1^2+2\geq -3$; again a contradiction with \ref{cor:fE>=4}(i).
\end{proof}

We denote the maximal twigs of $D$ contained in the first branch of $Q_1$ by $T_1$ and $T_2$, with (in the notation of \ref{lem:Q_with_small_KQ}) $T_1$ being the one containing $[(2)_k]$. The maximal twig contained in the second branch is denoted by $T_3$.

\bcl  $A$ meets the $(-2)$-tip of $D$ contained in $T_2$. \ecl

\begin{proof} Suppose $\eta=1$. By Claim 1 $\Delta_1^+\neq 0$, so if $\alpha\:X_1\to Z$ is the contraction of $\cal L+\Delta_1^+$ then $-(\alpha_*E_1)^2\leq \gamma_1-2$, hence by \ref{cor:fE>=4}(iv) $\gamma_1\geq 6$. Then $\tau_1^*\leq 5+\eta-\gamma_1\leq 0$, so $\tau_1^*=0$. Then $D_1$ has three maximal twigs and two of them end with a $(-2)$-tip, so $\ind(D_1')>1$; a contradiction.  Thus $\eta=0$. It follows that $A$ meets some tip $T_A$ of $D$ other than $E$. Clearly, $T_A$ is not a component of $T_1$, because otherwise $\pi(A)$ would be a $0$-curve contained in $\PP^2$. Suppose it is a component of $T_3$. Then $\Delta_A$ is a $(-2)$-twig contained in $\varphi_1(T_3)$. By \ref{lem:Q_with_small_KQ} this is impossible if $r_2\leq 2$ or if $T_3=[2,3]$, so $Q_1$ is of type $(1,3)$ and $T_3=[3,2]$. Then $C_1'$ is a $0$-curve with $C_1'\cdot E_1=3$; a contradiction with \ref{cor:fE>=4}(i).
\end{proof}

By Claim 5 $Q_1$ is of type $(r_1,r_2)$ and by \ref{lem:type} $r_1+r_2=K\cdot R=4$. Since $\eta=0$, by Claim 3 we have $\gamma_1+\tau_1^*\leq 5$.

Consider the case when $Q_1$ is of type $(1,3)$. We have $T_1=[(2)_k,3]$ and $T_2=[2]$. Since $\tau_1^*\leq 5-\gamma_1\leq 1$, the second branch of $Q_1$ is $[(2)_u,4,2,1,3,2]$ or $[(2)_u,3,3,1,2,3]$. We have $1\geq \ind(D_1')=1-\frac{2}{2k+3}+\frac{2}{5}$, so $k\leq 1$. By \eqref{eq:b2(D1)} $k+u+6=\#D_1=9+\gamma_1\geq 13$, so $u\geq 6$. Note that the image of the branching $(-2)$-curve of $Q_1$ is a $(-1)$-curve in $D_1$ meeting $E_1$ transversally in one point. Let $\alpha\:X_1\to Z$ be the contraction of a subchain of $D_1-E_1$ consisting of this $(-1)$-curve and two $(-2)$-curve from the second branch and let $f$ be the image of the $(-3)$-curve from $T_1$. Then $f$ is a $0$-curve with $f\cdot \alpha_*E_1=3$; a contradiction with \ref{cor:fE>=4}(i).

Consider the case when $Q_1$ is of type $(3,1)$. Since $A$ meets a $(-2)$-tip of $T_2$, by \ref{lem:Q_with_small_KQ}(iv) $T_2=[2,2,2]$ or $T_2=[2,3]$. Also, \eqref{eq:ind_AE=1} gives $k=0$. In the first case $D_1-E_1=[5,1,(2)_u,1]$, so taking the subchain $f=[1,(2)_u,1]$ we have a fiber of a $\PP^1$-fibration of $X_1$ with $f\cdot E_1=1+\tau_1=3$, which contradicts \ref{cor:fE>=4}(ii). Thus  $T_2=[2,3]$. By \eqref{eq:b2(D1)} $u+6=\#D_1=9+\gamma_1$, so $u=\gamma_1+3$. The characteristic pairs of $\psi_0(Q_1)$ are $\binom{7}{5},\binom{1}{1}_{u+1}$, so by \ref{cor:I and II}(iii) $(\deg \bar E-1)(\deg \bar E-2)= 2(2(7\cdot 5+u+1)-(7+5+u))= 120+2u=126+2\gamma_1$. Thus $(\deg \bar E-1)(\deg \bar E-2)\in \{134,136\}$; a contradiction.

Finally, consider the case when $Q_1$ is of type $(2,2)$. By \ref{lem:Q_with_small_KQ}(iii) $\tau_1^*=1$ and, since $A$ meets a $(-2)$-tip of $T_2$, we have $T_2=[2,2]$ and $T_1=[(2)_k,4]$. Also, $T_3=[2,2]$ or $T_3=[3]$. It follows that $\gamma_1=4$ and $k+u+7-\tau_1=\#D_1=9+\gamma_1=13$, so $u=6+\tau_1-k$. Suppose $k\neq 0$. Because $\ind(D_1')\leq 1$, the latter is possible only if $T_3=[3]$ and $k=1$, so $D_1-E_1=[(2)_k,4,1,(2)_u,2,1]=[2,4,1,(2)_8,1]$. Let $B$ be the sixth component of this (ordered) chain. The contraction of $D_1-E_1-B$ maps $X_1$ onto $\PP^2$ and maps $B$ onto a smooth curve of self-intersection $2$; a contradiction. Thus $k=0$ and hence $u=6+\tau_1$.

If $T_2=[2,2]$ then $\tau_1=3$ and the characteristic pairs of $\psi_0(Q_1)$ are $\binom{4}{3},\binom{1}{1}_{u+1}$, so by \ref{cor:I and II}(iii) $(\deg \bar E-1)(\deg \bar E-2)= 3(3(12+u+1)-(7+u))= 96+6u=150$; a contradiction. If $T_2=[3]$ then $\tau_1=2$ and the characteristic pairs of $Q_1$ are $\binom{12}{9},\binom{3}{3}_{u+1},\binom{3}{2}$. Then $I(q_1)=12\cdot 9+9(u+1)+6=195$ and $M(q_1)=12+9+3(u+1)+2-1=50$. By \ref{cor:I and II}(iii) $(\deg \bar E-1)(\deg \bar E-2)=145$; a contradiction.
\end{proof}

\bprop\label{prop:zeta=1} $\zeta=1$. \eprop

\begin{proof} By \ref{lem:zeta>=0} $\zeta\in\{0,1\}$. Suppose $\zeta=0$. By \ref{prop:A0E0=0} $A\cdot E_0=0$, so \ref{lem:mmp_basics}(iii) says that for every component $V$ of $D_1-E_1$ we have $V\cdot E_1=(\psi^{-1})_*V\cdot E_0$. In particular, if $\cal L\neq 0$ (recall that $\#\cal L\leq 1$ by \ref{lem:mmp_basics}) then $\cal L\cdot (2K_1+E_1)=\cal L'\cdot E_0-2\leq -1$, where $\cal L'=(\psi^{-1})_*\cal L$. Moreover, if $\cal L$ is a component of $\Upsilon_1$ then \ref{lem:mmp_basics}(viii) says that $\cal L'\cdot E_0=0$. It follows that $\eta=\eta_0$. By \ref{cor:n=1_basics}(iii) \begin{equation*}\gamma_1+\tau^*\leq 4+2\eta.\end{equation*} Recall that $\gamma_1\geq 4$.

\setcounter{claim}{0}
\bcl $\cal C_+=0$. \ecl

\begin{proof} Suppose $(C_1')^2\geq 0$. By \ref{lem:mmp_basics}(vii) $\tau_1^*\geq 2+(C_1')^2\geq 2$, so the above inequality gives $\gamma_1=4$, $\eta=1$ and $\tau^*=\tau_1^*=2$. Then $C_1'$ is a $0$-curve. Clearly, $\cal L$ meets $C_1'$, so $f=C_1'+\cal L$ is nef. Then $0\leq f\cdot (2K_1+E_1)\leq -5+C_1'\cdot E_1=-5+\tau_1$, so $\tau_1^*\geq \tau_1-2\geq 3$; a contradiction.
\end{proof}

Now \ref{cor:n=1_basics}(iv) reads as \begin{equation}\label{eq:K1R1_case_zeta=1}\tau^*+K_1\cdot R_1=2+\#\cal L.\end{equation}

\bcl $\eta=0$. \ecl

\begin{proof} Assume $\eta=1$. Then $\gamma_1+\tau^*\leq 6$ and $\cal L\cdot E_1=\cal L'\cdot E_0=0$. Suppose first that $\cal L$ meets some $C_j'$, say $C_1'$. Then $f=C_1'+\cal L$ is nef, so $0\leq f\cdot (2K_1+E_1)=-4+\tau_1=\tau_1^*+s_1-3$, so $\tau_1^*\geq 3-s_1\geq 2$ and hence $\tau_1^*=2$ and $s_1=1$. We have $\Delta_1^+=0$, because otherwise the contributions of $\Delta_1^+$ and of the $(-2)$-twig of $D_1'$ contracted by $\varphi_1$ would add up to more than $1$, contradicting \ref{lem:mmp_basics}(iv). By the definition of $\Upsilon_1$ it follows that $\beta_{D_1}(\cal L)=2$, so $\cal L\cdot C_1'=2$. But this implies that $\psi$ touches $C_1$; a contradiction with Claim 1. Therefore, $\cal L\cdot C_j'=0$ for every $j\leq c$. If $\cal L\cdot \Delta_1^+=0$ put $\Delta_{\cal L}=0$, otherwise let $\Delta_{\cal L}$ be the connected component of $\Delta_1^+$ meeting $\cal L$.

Consider the case when $\cal L$ meets two components of $D_1-\Delta_1^+$, say $B_1$ and $B_2$. Then $\Delta_{\cal L}\neq 0$. Clearly, $B_1$ and $B_2$ are components of $R_1$. Denote by $\bar B_1$, $\bar B_2$, $\bar R_1$, $\bar K_1$ and $\bar E_1$ the push-forwards of $B_1$, $B_2$, $R_1$, $K_1$ and $E_1$ by the contraction of $\cal L+\Delta_\cal L$. A component of $R_1$ meets $E_1$ at most once, hence a component of $\bar R_1$ meets $\bar E_1$ at most once. If, say, $\bar B_1^2\geq 0$ then $\bar B_1$ is a nef divisor intersecting $2\bar K_1+\bar E_1$ negatively. Since $2\bar K_1+\bar E_1$ is effective, we infer that $\bar B_1^2\leq -1$ and $\bar B_2^2\leq -1$. Because $K_1\cdot (B_1+B_2)\leq K_1\cdot R_1\leq 3$, we have $-4-\bar B_1^2-\bar B_2^2=\bar K_1\cdot (\bar B_1+\bar B_2)\leq -1$, hence, say, $\bar B_1^2=-1$ and $\bar B_2^2\in \{-1,-2\}$. Then $\bar B_1+\bar B_2$ is a nef divisor, so its intersection with $2\bar K_1+\bar E_1$ is non-negative. It follows that $\bar E_1\cdot (\bar B_1+\bar B_2)\geq 2$, hence both $(\psi^{-1})_*B_i$ meet $E_0$. Then $c=2$ and $s_1=s_2=0$, so $\tau^*\geq 2$. We get $\gamma_1=4$. Now the contraction of $\bar B_1$ maps $\bar E_1$ onto a $(-3)$-curve, which contradicts \ref{cor:fE>=4}(iv).

Thus we may assume that $\cal L$ meets only one component $B_1$ of $D_1-\Delta_1^+$ and $B_1\subset R_1$. Let $B_0$ be the proper transform of $B_1$ on $X_0$. Clearly, $B_0$ meets $\cal L'$. For $f=B_1+2\cal L$ we have $f\cdot (2K_1+E_1)=2(K_1\cdot B_1-2)+B_1\cdot E_1<2(K_1\cdot B_1-1)$. It follows that $K_1\cdot B_1\geq 2$, otherwise $f$ is nef and its intersection with $2K_1+E_1$ is negative, which is impossible. Since $K_1\cdot B_1\leq 3-\tau^*\leq 3$, we get $K_1\cdot B_1\in\{2,3\}$, hence $B_1^2\in\{-4,-5\}$. We may assume $B_0\subset \wt Q_1$.

Suppose $A\cdot B_0=0$. Then $\beta_{\wt Q_1}(B_0)=\beta_{D_1}(B_1)\geq 3$. Let $\alpha\:X_0\to Z$ be the composition of successive contractions of $(-1)$-curves in $\wt Q_1$ and its images until the unique $(-1)$-curve meets the image of $B_0$. Since $\alpha_*B_0$ is a branching component of $\alpha_*\wt Q_1$ with $b=-(\alpha_*B_0)^2=-B_0^2\geq 5$, the $(-1)$-curve is a part of a twig $T=[1,(2)_{b-2}]$ of $\alpha_*\wt Q_1$. It follows that $D_1'$ contains the twig $[(2)_{b-3}]$. This in turn implies that $s_j=0$ and consequently that $\tau_j^*>0$ for every $j$. Indeed, otherwise, since $b-3\geq 2$, $\ind(D_1')\geq \frac{1}{2}+\frac{2}{3}>1$, which contradicts \ref{cor:n=1_basics}(i). Because $0\leq K_1\cdot (R_1-B_1)\leq 3-\tau^*-2= 1-\tau^*\leq 1$, we infer that $\tau^*=1$, $c=1$, $B_1^2=-4$ and that $R_1-B_1$ consists of $(-2)$-curves. Now the inequality $\ind(D_1')\leq 1$ implies that $D_1'$ has exactly one maximal twig other than $E_1'$. Since $s_1=0$, $C_1'$ is not a tip of $D_1$, so $D_1-E_1-\cal L=[4,1,(2)_k]$ for some $k\geq b-2\geq 3$. We have $k+3=\#(D_1-E_1)=\rho(X_1)=10-K_1^2=10+K_1\cdot E_1=8+\gamma_1$, so $k=\gamma_1+5$. But $D_1'$ contains a twig $[(2)_{k-1}]$, so $(1-1/k)+\frac{1}{-(E_1')^2}\leq\ind(D_1')\leq 1$ and hence $k\leq -(E_1')^2=\gamma_1+2$; a contradiction.

Thus $A\cdot B_0=1$.  Then $\psi$ is of type $I$, hence $\cal L'$ meets $B_0$ and $\cal L'$ is a $(-2)$-curve touched by $\psi$. By the definition of $\Upsilon_1$ and by \ref{lem:mmp_basics}(ii) $\cal L'$ is in fact a part of some maximal $(-2)$-twig $T=[(2)_t]$, $t\geq 1$ of $D_0$ meeting $B_0$. We have $b=-B_0^2\geq -B_1^2+1\geq 5$. By \eqref{eq:K1R1_case_zeta=1} $$\tau^*+K_1\cdot (R_1-B_1)=5+B_1^2\leq 1,$$ so $\tau^*\leq 1$. Suppose $R_1-B_1$ consists of $(-2)$-curves. Then $\wt Q_1-C_1-B_0$ consists of $(-2)$-curves. It follows that with $\alpha\:X_0\to Z$ as above $\alpha_*\wt Q_1$ is a chain, hence it is of type $[(2)_t,b,1,(2)_{b-2}]$. Then $D_1'$ contains a twig $[(2)_{b-3}]$, whose contribution to $\ind(D_1')$ is at least $\frac{2}{3}$. Since $\ind(D_1')\leq 1$, we get $c=1$ and $s_1=0$, so $\tau^*=1$, $\tau_1=2$ and $B_1^2=-4$. In this case $C_1$ is not a tip of $\wt Q_1$, so, since $R_1-B_1$ consists of $(-2)$-curves, $\wt Q_1-B_0-T$ is a chain $[1,2,\ldots,2]$. Then the contraction of $D_1-E_1-B_1$ maps $X_1$ onto $\PP^2$ and $E_1$ onto a unicuspidal curve with a cusp of multiplicity at most three, which is in contradiction with \ref{lem:min(mu)>=4}. Thus $R_1-B_1$ does not consist only of $(-2)$-curves. By the equation above $B_1^2=-4$, $\tau^*=0$ and $R_1-B_1$ consists of one $(-3)$-curve and some number of $(-2)$-curves. Since $\ind(D_1)\leq 1$, it follows that $c=1$ (hence $E_1'$ is a tip of $D_1'$) and that $D_1'$ contains no $(-2)$-tips other than the one contracted by $\varphi_1$. Let $T'=[1,2,2,\ldots,2]$ be the maximal twig of $\wt Q_1$ containing $C_1$. The divisor $\wt Q_1$ has two other maximal twigs: $[(2)_{t},b,(2)_{t_1}]$ and $[3,(2)_{t_2}]$ for some $t_1\geq -1$, $t_2\geq 0$, so $1\geq \ind(D_1')\geq \frac{1}{2}+\frac{1}{\gamma_1+2}+(1-\frac{2}{2t_2+3})$. By \ref{cor:n=1_basics}(iii) $\gamma_1\leq 6$, so $t_2=0$. Because $\wt Q_1$ contracts to a smooth point, we get $t_1=1$ and then $b=3$; a contradiction.
\end{proof}

Since $\eta=0$, the inequality $\gamma_1+\tau^*\leq 4+2\eta$ gives $\gamma_1=4$ and $\tau^*=0$, hence $K_1\cdot(R_1+\cal L)=K_1\cdot R_1-\#\cal L=2$. Because $\tau^*=0$, the contribution to $\ind(D_1')$ of the $(-2)$-twigs of $D$ contracted by $\varphi_1$  is at least $\frac{1}{2}c$. If $c=2$ then, since $\ind(D_1')\leq 1$, $D_1'$ has no other maximal twigs, which implies that both $\wt Q_j=\psi(Q_j)$ are chains ending with a $(-1)$-curve tangent to $E_0$. But in the latter case both cusps are semi-ordinary, which is not possible for $n>0$. Therefore, $c=1$. Let $p\:X_1\to \PP^1$ be the elliptic fibration given by the linear system $|2C_1'+E_1|$ (see \ref{lem:elliptic_ruling}) and let $H$ be the unique component of $D_1-E_1$ meeting $C_1'$. It is the unique horizontal component of $D_1$. The proper transform of $H$ on $X_0$, being the unique component of $D_0-E_0$ meeting $C_1$, is a $(-2)$-curve, hence $H^2\geq -2$. If $H^2\geq -1$ or if $H^2=-2$ and $H$ meets some vertical $(-1)$-curve other than $C_1$ then we easily find a nef divisor ($H$ or $C_1'+H$ or $C_1'+H+\cal L$) intersecting $2K_1+E_1$ negatively, which contradicts $2K_1+E_1$ being effective. Therefore, $H^2=-2$ and $H$ meets no vertical $(-1)$-curves other than $C_1'$.

\bcl There is a $(-1)$-curve $L\not\subset D_1$ with $L\cdot D_1=2$. \ecl

\begin{proof}
We have $K_1\cdot R_1=2+\#\cal L>0$, so there is a vertical component of $R_1$ which is a $(k)$-curve for some $k\leq -3$. We have $K_1^2=\zeta-K_1\cdot E_1=-2$, so by \ref{lem:elliptic_fibers} there exists a vertical $(-1)$-curve $L$ other than $C_1'$ and if $\alpha\: X_1\to Z$ is the contraction of $L+C_1'$ then the fibers of the induced elliptic fibration of $Z$ are all minimal. Let $F$ be the fiber of $p$ containing $L$.

Suppose $\beta_F(L)>2$. Then $\alpha_*F$ is not snc, so by \ref{lem:elliptic_fibers} $F-L$ is a disjoint sum of three rational curves $U_1, U_2, U_3$ with $\sum \frac{1}{d(U_i)}=1$. Since $H$ is a $2$-section of $p$, some $U_i$, say $U_3$, does not meet $H$. It follows that $L$ is a component of $D_1$, otherwise either $D_1$ would not be connected or $X_1\setminus D_1$ would contain a complete curve, which is impossible, because $X_1\setminus D_1$ is affine. Since $D_1$ contains no superfluous $(-1)$-curves, all $U_i$ are components of $D_1$. Therefore, $L=\cal L$ and we have $K_1\cdot (U_1+U_2+U_3)\leq K_1\cdot R_1=2+\#\cal L=3$. The condition $\sum \frac{1}{d(U_i)}=1$ gives $U_i^2=-3$ for $i=1,2,3$. Note that since $\ind(D_1')\leq 1$, the contribution to $\ind(D_1')$ of twigs of $D_1$ other than the twig $T=[2]$ contracted by $\varphi_1$ is at most $\frac{1}{2}$, which implies that $H$ meets $U_1$ and $U_2$ (each once) and that $D_1'$ has no maximal twigs other than $U_3$ and $T$. But $p_a(H+U_1+L+U_2)=1$, so because $p_a(D_1')=1$, we get that in fact $D_1=E_1+C_1+H+F_0$. However, $\#(D_1-E_1)=\rho(X_1)=10-K_1^2=12-\zeta=12$; a contradiction.

Thus $\beta_F(L)\leq 2$. Suppose $L$ is a component of $D_1$. Since $D_1$ contains no superfluous $(-1)$-curves, we see that $F=L+U$, where $U$ is a $(-4)$-curve and $U\cdot L=2$. Then $K_1\cdot (R_1-U)=\#\cal L=1$, so $\alpha_*(R_1-U)$ contains a vertical $(-3)$-curve and hence some singular reducible fiber of the induced elliptic fibration of $Z=\alpha(X)$ does not consist of $(-2)$-curves, which is in contradiction to \ref{lem:elliptic_fibers}(i). Thus $L\not \subset D_1$. Since $L\cdot H=0$, $U$ is a component of $D_1$. By \ref{prop:fundamentals}(iii) $L\cdot D_1=2$.
\end{proof}

Since by \ref{prop:fundamentals}(iii) $\PP^1\setminus \bar E$ contains no affine lines, $X_1\setminus D_1$ contains no affine lines, so in fact $L$ meets $D_1$ at two different points, transversally. The proper transform of $L$ on $X_1'$, which we denote by the same letter, is a $(-1)$-curve with the same properties. We compute $(K_{X_1'}+D_1'+L)^2=(K_{X_1'}+D_1')^2+1=K_{X_1'}\cdot (K_{X_1'}+D_1')+1=K\cdot (K+D)-n+1=p_2(\PP^2,\bar E)=3$ and $\chi(X_1'\setminus (D_1'+L))=\chi(X_1'\setminus D_1')=1$. By \ref{lem:BMY} $\ind(D_1'+L)+(K_{X_1'}+D_1'+L)^2\leq 3\chi(X_1'\setminus (D_1'+L))$, hence $\ind(D_1'+L)=0$. But $D_1'$ has a $(-2)$-tip $T$ (contracted by $\psi_0$) which is not met by $L$, so $\ind(D_1'+L)\geq \frac{1}{2}$; a contradiction.

\end{proof}

Now we rule out the case $n=\zeta=1$.

\begin{proof}[Proof of Theorem \ref{prop:n=0}] By \ref{prop:zeta=1} $\zeta=1$. By \ref{prop:A0E0=0} $A\cdot E_0=0$. Then \ref{lem:mmp_basics}(viii) says that if $\cal L$ is a component of $\Upsilon_1$ then $\cal L\cdot E_1=0$, hence $\eta_0=\eta$. Now \ref{cor:n=1_basics} gives

\beq \gamma_1+\tau^*\leq 6+2\eta,\label{eq:n=zeta=1_gamma1}\eeq and
\beq\sum_{j:(C_j')^2\geq 0}(K_1\cdot C_j'+\tau_j^*)+\#\cal C_++\sum_{j:(C_j')^2=-1}\tau_j^*+K_1\cdot R_1=1+\#\cal L.\label{eq:n=zeta=1_K1R1}\eeq
Recall that $\cal L'=(\psi^{-1})_*\cal L$.

\setcounter{claim}{0}
\bcl $\cal C_+=0$. \ecl

\begin{proof} Suppose $(C_1')^2\geq 0$. Assume $\#\cal L=1$. The divisor $\cal L+C_1'$ is nef. By \eqref{eq:n=zeta=1_K1R1} and \ref{lem:mmp_basics}(vii) $0\leq K_1\cdot C_1'+\tau_1^*\leq \#\cal L=1$ and $\tau_1^*\geq 2+(C_1')^2\geq 2$, so $(2K_1+ E_1)\cdot C_1'\leq K_1\cdot C_1'+1-\tau_1^*+\tau_1=K_1\cdot C_1'+2+s_1\leq s_1$. It follows that $0\leq (C_1'+\cal L)\cdot (2K_1+E_1)\leq \cal L\cdot E_1-2+s_1$, so $\cal L\cdot E_1+s_1\geq 2$. Since $A\cdot E_0=0$, we have $\cal L\cdot E_1=\cal L'\cdot E_0\leq 1$. Then the inequalities become equalities, so $s_1=K_1\cdot C_1'+\tau_1^*=1$, $K_1\cdot C_1'=-2$ and $\cal L'\cdot E_0=1$. We infer that $\tau_1^*=3$ and that $\cal L$ meets some $C_j'$ for some $j>1$. But the latter implies that $\eta=0$, so \eqref{eq:n=zeta=1_gamma1} gives $\tau_1^*\leq 6-\gamma_1\leq 2$; a contradiction.

We obtain $\#\cal L=0$. Then $\eta=0$ and $\#\cal C_++K_1\cdot R_1\leq 1$. It follows that $\#\cal C _+=1$, $K_1\cdot C_1'+\tau_1^*=0$ and $R_1$ consists of $(-2)$-curves. Again, $\tau_1^*\geq 2$, which by \eqref{eq:n=zeta=1_gamma1} implies that $\gamma_1=4$ and $\tau_1^*=2$. Then $K_1\cdot C_1'=-2$, so $C_1'$ is a $0$-curve. This means that $\psi$ touches $C_1$ exactly once. Since $\eta=0$, this is possible only if $A\cdot C_1=0$ and $\psi$ contracts some twig $V$ of $D_0$ meeting $C_1$. Because $A\cdot E_0=0$, we have $C_1'\cdot E_1=C_1\cdot E_0=\tau_1$, so $0\leq C_1'\cdot (2K_1+E_1)=-4+\tau_1= s_1-1$, so $s_1=1$. But then $Q_1=V+C_1$, so since $\psi$ contracts $V$, $A$ meets $V$ once, and hence it meets $Q_1+E_1$ once. Then $A$ meets some component of $D_1-E_1-Q_1$, so $c\geq 2$. But by \eqref{eq:n=zeta=1_gamma1} $\tau_2^*=0$, so $\ind(D_1')\geq \frac{3}{4}+\frac{1}{2}>1$; a contradiction.
\end{proof}

Since $\cal C_+=0$, \eqref{eq:n=zeta=1_K1R1} reads as \beq \tau^*+K_1\cdot R_1=1+\#\cal L.\label{eq:n=zeta=1_K1R1_v2}\eeq

\bcl $\cal \tau^*\leq 1$. \ecl

\begin{proof} By the above inequality $\tau^*\leq 2$. Suppose $\tau^*=2$. Then $K_1\cdot R_1=0$, so $R_1$ consists of $(-2)$-curves. Also, $\#\cal L=1$. By \ref{lem:mmp_basics}(iii) for every component $V$ of $D_1-E_1$ we have $$V\cdot E_1=(\psi^{-1})_*V\cdot E_0$$ and the latter number is at most $1$ for $V\neq C_j'$.

Suppose there is a component $M$ of $R_1$, such that $M\cdot \cal L\geq 2$. Then $2\cal L+M$ is nef, so $0\leq (2\cal L+M)\cdot (2K_1+E_1)=M\cdot E_1+2\cal L\cdot E_1-4$, so $M\cdot E_1+2\cal L\cdot E_1\geq 4$. But $M\cdot E_1\leq 1$ and $\cal L\cdot E_1\leq 1$; a contradiction. Suppose there exist two components $M$, $M'$ of $R_1$ meeting $\cal L$. Then $M+2\cal L+M'$ is nef, so $0\leq (M+2\cal L+M')\cdot (2K_1+E_1)=-4+(M\cdot E_1+2\cal L\cdot E_1+M'\cdot E_1)$, so all of $M, M', \cal L$ meet $E_1$. Then the proper transforms of $M, M', \cal L$ on $X_0$ meet $E_0$. Since these transforms are contained in $D_0-E_0-\cal C$, they are contained in different connected components of $D_0-E_0$, so $c\geq 3$; a contradiction. We obtain $R_1\cdot \cal L\leq 1$.

Since by Claim 1 $\psi$ does not touch any $C_j'$, we have $\cal L\cdot C_j'=\cal L'\cdot C_j\leq 1$, so $\cal L$ meets each of $R_1$, $C_1',\ldots, C_c'$ at most once. This forces $\cal L\cdot E_1=1$. Indeed, otherwise $\cal L\cdot E_1=0$, so $\cal L$ meets at least two $C_j'$'s, and hence $\cal L'$ meets at least two $C_j$'s, which is impossible. Say $\cal L\cdot C_1'=1$. Then $\cal L'\cdot C_1=1$, so $s_1=0$. The divisor $C_1'+\cal L$ is nef, so $0\leq (C_1'+\cal L)\cdot (2K_1+E_1)=\tau_1-3$, so $\tau_1^*=\tau_1-1\geq 2$. Since $\cal L\cdot E_1=1$, \ref{lem:mmp_basics}(viii) says that $\eta=0$, so $\gamma_1\leq 6-\tau^*\leq 4$. Then $\gamma_1=4$ and the contraction of $\cal L$ maps $E_1$ onto a $(-3)$-curve; a contradiction with \ref{cor:fE>=4}(iv).
\end{proof}

\bcl $\cal L=0$. \ecl

\begin{proof} Suppose $\#\cal L=1$.

Suppose $\cal L$ meets $C_j'$, say, for $j=1$. Then again the divisor $C_1'+\cal L$ is nef, so $0\leq (C_1'+\cal L)\cdot (2K_1+E_1)=-4+\tau_1+\cal L\cdot E_1$. We have $\cal L\cdot E_1=\cal L'\cdot E_0\leq 1$, so $\tau_1\geq 3$. By Claim 2 $\tau^*\leq 1$, so $\tau_1=3$ and $s_1=1$. But $s_1=1$ implies that $\cal L\cdot E_1=0$, so the initial inequality fails; a contradiction. It follows that $\cal L\cdot C_j'=0$ for every $j$ and hence that $\cal L\cdot E_1=0$. Let $M$ be a component of $R_1$ meeting $\cal L$. Put $M'=(\psi^{-1})_*M$.

Suppose $M\cdot \cal L>1$. Say $M'\subset \wt Q_1$. By \eqref{eq:n=zeta=1_K1R1_v2} $K_1\cdot M\leq 2$, so $M^2\geq -4$. Then the divisor $M+2\cal L$ is nef, so $0\leq (2\cal L+M)\cdot (2K_1+E_1)=-4+2K_1\cdot M+M\cdot E_1$. We have $M\cdot E_1=M'\cdot E_0\leq 1$, so $K_1\cdot M\geq 2$. But $K_1\cdot R_1\leq 2$, so we get $K_1\cdot M=2$, $K_1\cdot (R_1-M)=0$ and by \eqref{eq:n=zeta=1_K1R1_v2} $\tau^*=0$. Since $\ind(D_1')\leq 1$, $c=1$. Now the curve $C_1$ is a tip of $\wt Q_1$, $M^2=-4$ (hence $(M')^2\leq -5$) and $R_1-M$ consists of $(-2)$-curves. Also, $\psi(A)$ is one of the points of intersection of $M$ and $\cal L$. If in $D_1-E_1$ there is a chain $f=[1,2,\ldots,2,1]$ containing both $\cal L$ and $C_1'$ then $f$ is nef and $f\cdot (2K_1+E_1)\leq -4+\tau_1+\cal L\cdot E_1+R_1\cdot E_1=(\psi^{-1})_*R_1\cdot E_1-2<0$, which is impossible. It follows that the connected component of $\wt Q_1-M'$ containing $C_1$ is a chain $T=[1,2,\ldots,2]$ not touched by $\psi$. Since $\wt Q_1$ contracts to a point, $T\neq C_1$ and $M'$ does not meet the $(-2)$-tip of $T$, which is therefore a tip of $D_1'$. Then $\ind(D_1')\geq \frac{1}{2}+\frac{1}{2}+\frac{1}{-(E_1')^2}>1$; a contradiction.

Therefore, every component of $R_1$ meets $\cal L$ at most once. Because $D_1$ contains no superfluous $(-1)$-curves, we see that $\cal L$ meets at least three components of $R_1$. If at least two of them, say $M_1$, $M_2$, have self-intersection bigger than $(-3)$ then $M_1+2\cal L+M_2$ is nef and intersects $2K_1+E_1$ negatively, which contradicts the effectiveness of $2K_1+E_1$. Since $K_1\cdot R_1\leq 2$, we get that $\cal L$ meets exactly three components of $R_1$, say $M_1$, $M_2$, and $M_3$, and we have $M_1^2=-2$, $M_2^2=M_3^2=-3$. Then the divisor $M_1+3\cal L+M_2+M_3$ is nef, so its intersection with $2K_1+E_1$ is non-negative. We obtain $(M_1+M_2+M_3)\cdot E_1\geq 2$, so $R_1\cdot E_1\geq 2$. It follows that $(\psi^{-1})_*R_1\cdot E_0\geq 2$, so $s_1=s_2=0$. Then $\tau^*\geq 2$, which contradicts Claim 2.
\end{proof}

Now \eqref{eq:n=zeta=1_gamma1} and \eqref{eq:n=zeta=1_K1R1_v2} give
$\gamma_1+\tau^*\leq 6$ and $\tau^*+K_1\cdot R_1=1.$ We may, and shall, assume $\tau_1^*\geq \tau_2^*$. Then $\tau_2^*=0$ and $\tau^*=\tau_1^*\leq 1$.

\bcl $\tau_1=2$. \ecl

\begin{proof} Suppose $\tau_1>2$. Then $\tau_1^*\geq 2-s_1\geq 1$, so by \eqref{eq:n=zeta=1_K1R1_v2} $\tau_1=3$, $s_1=1$ and $R_1$ consists of $(-2)$-curves. The contribution to $\ind(D_1')$ from the maximal twigs of $D_1'$ contracted by $\varphi_1$ is $\frac{2}{3}+(c-1)\frac{1}{2}$, so since $\ind(D_1')\leq 1$, we have $c=1$. Therefore, $D_1'$ consists of one $(-1)$-curve, one $(-4)$-curve and some number of $(-2)$-curves. We infer that $D_1'$ has no tips other than $E_1'$ and the one contracted by $\varphi_1$. Indeed, otherwise the contribution of the additional twig to $\ind(D_1')$ would have to be at most $1-\frac{2}{3}-\frac{1}{\gamma_1+3}\leq \frac{1}{3}-\frac{1}{8}<\frac{1}{4}$, which is impossible. Since the arithmetic genus of $D_1'$ is one, the arithmetic genus of $D_1'-E_1'$, and hence of $D_1-E_1$, is also one. Since $C_1'$ is the unique tip of $D_1-E_1$, we infer that $D_1-E_1$ has a unique branching component. Thus there exists a component $B$ of $D_1-E_1$ such that $D_1-E_1-B$ is a chain. The contraction of $D_1-E_1-B$ maps $X_1$ onto $\PP^2$ and $E_1$ onto a unicuspidal curve with a cusp of multiplicity three. This contradicts \ref{lem:min(mu)>=4}.
\end{proof}

Now $\tau^*=\tau^*_1=1-s_1$, so \eqref{eq:n=zeta=1_gamma1} and \eqref{eq:n=zeta=1_K1R1_v2} give
\beq \gamma_1\leq 5+s_1\label{eq:n=zeta=1_gamma1_v3}\eeq and
\beq K_1\cdot R_1=s_1.\label{eq:n=zeta=1_K1R1_v3}\eeq
If $s_1=0$ we denote the component of $\wt Q_1-C_1$ meeting $E_0$ by $L'$ and we put $L=\psi(L')$. By the Noether formula $\rho(X_1)=10-K_1^2=10+K_1\cdot E_1-\zeta=\gamma_1+7\geq 11$. We get $\#R_1=\#(D_1-E_1)-c-1\geq \rho(X_1)-3\geq 8$.

\bcl $c=1$. \ecl

\begin{proof} Suppose $c=2$. Consider the case $s_1=1$. The contribution to $\ind(D_1')$ from the twigs contracted by $\varphi_1$ is $c\cdot \frac{1}{2}=1$, so since $\ind(D_1')\leq 1$, $D_1'$ has two maximal twigs. Then $D$ has at most four maximal twigs, so $\wt Q_1$ and $\wt Q_2$ are chains with $(-1)$-curves as tips. Then both cusps of $\bar E$ are semi-ordinary, so $\Delta_0^-=0$, which is impossible for $n=1$. Therefore, $s_1=0$, so by \eqref{eq:n=zeta=1_K1R1_v3} $R_1$ consists of $(-2)$-curves. Because $\tau_1=\tau_2=2$, we infer that $D_1'-E_1'$ consists of two $(-1)$-curves (both meeting $E_1'$), three $(-3)$-curves and some number of $(-2)$-curves.

Suppose $D_1'$ has more than two maximal twigs. Since $\ind(D_1')\leq 1$, it has exactly three and all of them are $(-3)$-curves. But the twig contained in $\psi'(Q_2)$ which is contracted by $\varphi_1$ is a $(-2)$-curve; a contradiction. It follows that $D_1'$ has at most two maximal twigs, hence $D$ has at most four maximal twigs. Then $\wt Q_1$ and $\wt Q_2$ are chains. Since $\wt Q_2$ ends with a $(-1)$-curve, $\wt Q_2-C_2'$ is a $(-2)$-chain, hence $q_2\in \bar E$ is semi-ordinary and $\wt Q_2-C_2'$ is a part of $\Delta_0^+$. Then $A$ does not meet $\wt Q_2$. Now the contribution of the twigs of $D_1'$ contained in $Q_2$ to $\ind(D_1')$ is at least $\frac{5}{6}$, so $\psi'(Q_1)$ has no tips and hence $\psi(\wt Q_1)$ is a cycle. The contraction of $D_1-L$ maps $X_1$ onto $\PP^2$ and $\bar E$ onto a bicuspidal curve with cusps of multiplicity two and three. This is impossible by \ref{lem:min(mu)>=4}.
\end{proof}

\bcl $s_1=0$. \ecl

\begin{proof}Suppose $s_1=1$. Then $R_1\cdot E_1=0$ and $R_1$ consists of one $(-3)$-curve $V$ and some number of $(-2)$-curves. Also, $\gamma_1\leq 6$. Since $C_1$ is the only component of $\wt Q_1$ meeting $E_0$, we have $\gamma_0+d^2=\tau_1^2 I(q_1)$, where $d=\deg \bar E$. Then $\gamma_0+d^2\equiv 0\mod 4$, so $\gamma_0$ is congruent to $0$ or $3$ modulo $4$. Because $\gamma_0=\gamma_1\in\{4,5,6\}$, we get $\gamma_1=4$. Let $\alpha\: X_1\to Z$ be the contraction of $C_1$ and let $p\:Z\to \PP^1$ be the elliptic fibration induced by the linear system $|\alpha_*E_1|$. We have $K_1^2+K_1\cdot E_1=\zeta=1$, so $K_1^2=-1$. Then $K_Z^2=0$. By \ref{lem:elliptic_fibers} all fibers of $p$ are minimal. Then singular fibers of $p$ are either irreducible or consist of $(-2)$-curves. It follows that $V$ meets $C_1'$. But then $(\psi^{-1})_*V$, which is the unique component of $\wt Q_1-C_1$ meeting $C_1$, has self-intersection smaller than $(-2)$, hence $\wt Q_1$ cannot be contracted to a (smooth) point; a contradiction.
\end{proof}

\bcl $D_1'$ has at most two maximal twigs.  \ecl

\begin{proof} From \eqref{eq:n=zeta=1_K1R1_v3} and Claim 6 it follows that $R_1$ consists of $(-2)$-curves. Because $(\tau_1,s_1)=(2,0)$, we obtain that $D_1'-E_1'$ consists of a $(-1)$-curve meeting $E_1'$, of two $(-3)$-curves and some number of $(-2)$-curves. Recall that $L'$ is the component of $\wt Q_1-C_1$ meeting $E_0$ and $L=\psi(L')$. The total reduced transform of $L$ on $X_1'$ is $[1,2,3]$. Let $W_2$ be the subchain $[2,3]$ of the latter chain and let $W_1$ be the proper transform of $C_1'$ on $X_1'$. Clearly, $W_1^2=-3$ and $D_1'-E_1'-C_1'-W_1-W_2$ consists of $(-2)$-curves. Moreover, at most one of $W_1$, $W_2$ is a maximal twig of $D_1'$. We have $\gamma=-E^2=-(E_1')^2 =\gamma_1+2\in\{6,7\}$.  Since $\ind(D_1')\leq 1$, $D_1'$ has at most one $(-2)$-tip. If $D_1'$ has more than three maximal twigs then their contribution to $\ind(D_1')$ is at least $\frac{1}{\gamma_1+2}+\frac{1}{3}+\frac{1}{3}+\frac{1}{2}>1$, which is impossible. Therefore, $D_1'$ has at most three maximal twigs.

Suppose $D_1'$ has three maximal twigs. Then one of them is $W_1=[3]$, otherwise $1\geq \ind(D_1')\geq \frac{1}{\gamma_1+2}+\frac{2}{5}+\frac{1}{2}$, which is impossible, because $\gamma_1+2<10$. It follows that the maximal twigs of $D_1'$ are: $E_1'=[\gamma_1+2]$, $W_1=[3]$ and $V_0=[2]$. In particular, $C_1'$ is a tip of $D_1-E_1$, so $C_1$ is a tip of $\wt Q_1$. Then the last Hamburger-Noether pair of the log resolution of $(\PP^2,\bar E)$ is $\binom{3}{2}$. Now \ref{cor:I and II}(i) used for $Y=X$ (hence $\rho_i=0$ in the notation of the lemma) gives $3|\gamma$. Because $\gamma=\gamma_1+2\in\{6,7\}$, we get $\gamma_1=4$ and hence $K_1^2=\zeta-K_1\cdot E_1=-1$. We get $\#(D_1-E_1)=\rho(X_1)=11$.

At this point we know the self-intersections of all components of $D_1$ and their number. Still, to arrive at a contradiction we need to determine the shape of the dual graph of $D_1$ and then recover $D_0$.

The components of $D_1$ generate $\Pic (X_1)$, so since $3K_1+E_1-C_1'$ intersects trivially with all these components, we have $$3K_1+E_1\equiv C_1',$$ where $\equiv$ denotes the numerical equivalence. Let $p\:X_1\to \PP^1$ be the elliptic fibration induced by the linear system $|2C_1'+E_1|$. Since $K_1^2=-1$, all fibers except $f=2C_1'+E_1$ are minimal. Because $L\cdot (2C_1'+E_1)=3$, $L$ is a $3$-section of $p$. The divisor $V=D_1-E_1-C_1'-L$ is vertical. We have $\rho(X_1)=11$ and $\#V=9$, so if $V$ does not contain some fiber then the $11$ components of $V+C_1+E_1$ are independent in $\NS_\Q(X_1)$ and $V+C_1+E_1$, being vertical, has a semi-negative intersection matrix. The latter is impossible by the Hodge index theorem, so we infer that the support of $V$ contains some reduced fiber $F_V$. We have $V_0\subset V$ (we identify $\varphi_1(V_0)$ and $V_0$ because $\varphi_1$ does not touch $V_0$). Also, $p_a(V+L)=p_a(D_1'-E_1')=1$ and $L$, $V_0$ are the only tips of $V+L$.

Suppose $V_0\cdot L=1$. Then $V=V_0+F_V$ and $L\cdot F_V=L\cdot V_0=1$. We have $p_a(F_V)=p_a(V+L)=1$, so by the Kodaira classification $F_V$ is a $(-2)$-cycle. The point $\psi(A)$ belongs to $F_V$, and more precisely to the branching component $B$ of $D_1-E_1$ contained in $V$. Since $D_0-E_0$ contracts to a point, after the contraction of the proper transform of $C_1+L'+(\psi^{-1}_*V_0$ (on $X_0$) the proper transform of $B$ becomes a $(-1)$-curve. It follows that $(\psi^{-1})_*B$ is a $(-3)$-curve, so $\psi$ touches it once. But $\psi$ contracts more than $A$, because otherwise both components of $D_0$ met by $A$ would be $(-3)$-curves, which is impossible. Hence the proper transform of $B$ on $X$ is branching in $D$. It follows that the Hamburger-Noether pairs of the log resolution of $(\PP^2,\bar E)$ are $\binom{6a}{6b}, \binom{6}{3}, \binom{3}{2}$ for some $a>b\geq 1$. From \ref{cor:I and II}(ii) we get $\gamma+d^2\equiv 0\mod 4$. But $\gamma_1=4$, so $\gamma=6$; a contradiction.

We obtain $V_0\cdot L=0$. Then $V$ is connected, so $V=F_V$. Now $V$ is an snc-divisor with $9$ components which has a $(-2)$-tip and has not more than three $(-2)$-tips (it contains more than one when $L$ is a part of a $(-2)$-cycle in $D_1-E_1$). From the Kodaira classification of minimal fibers, we see that $V$ is a fiber of type $II^*$, i.e.\ it is a $(-2)$-fork with twigs of length $1$, $2$, and $5$. Moreover, $L$ meets $V$ twice, exactly in the tips of the twigs of length $2$ and $5$. Since $\psi$ is inner for $D_0+A$, the divisor $D_0+A-E_0$ is an snc divisor consisting of a cycle (having at least $10$ components) and two twigs attached to it: $V_0$ and $C_1'$ (see Fig. \ref{fig:2}).
\begin{figure}[h]\centering\includegraphics[scale=0.4]{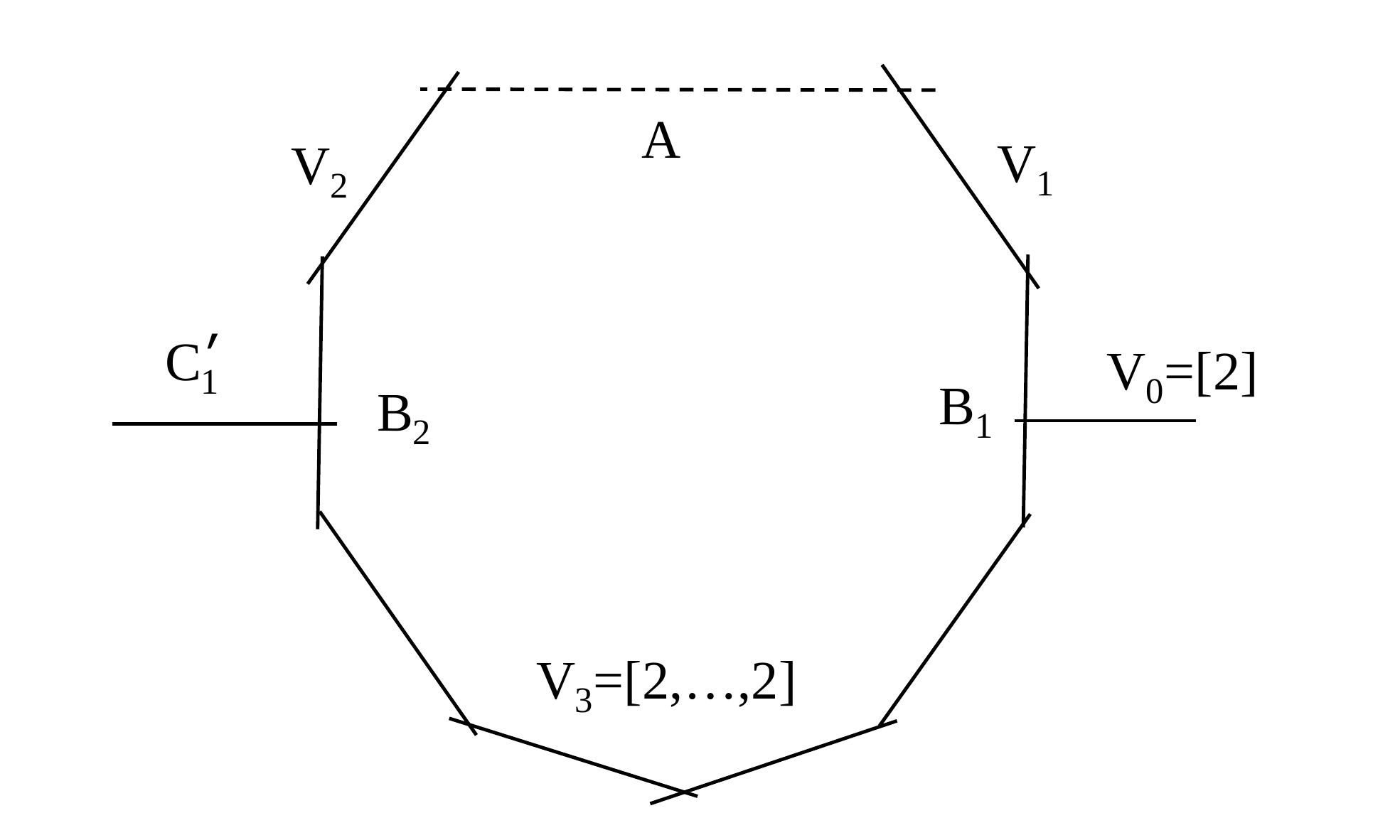}\caption{The divisor $D_0-E_0$ on $X_0$. Proof of Thm. \ref{prop:n=0}, Claim 7.}  \label{fig:2}\end{figure}
Let $B_1$ denote the component of this cycle meeting $V_0$ and let $V_1$ be the second maximal twig of $D_0$ meeting $B_1$ (put $V_1=0$ if $B_1$ is not a branching component of $D_0$). We have $L'\neq B_1$ and $L'\cdot C_1=1$. Let $V_2\neq C_1$ be the maximal twig of $D_0$ meeting $L'$. We denote the chain between $B_1$ and $L'$ by $V_3$. Since this chain is not touched by $\psi$, we have $V_3=[(2)_{v_3}]$ for some $v_3\geq 0$. By definition $D_0-E_0-C_1-V_0+A=V_2+L'+V_3+B_1+V_1+A$ is a cycle.

Consider the case $V_1\neq 0$, i.e.\ $A\cdot B_1=0$. Since $L'$ is the only component of $\wt Q_1$ meeting $C_1$, it is a $(-2)$-curve. Because $R_1$ consists of $(-2)$-curves, it follows that $\psi$ does not touch $L'$ and hence that $A\cdot B_2=0$, so $V_2\neq 0$ and hence $B_1$ and $L'$ are the branching components of $D_0-E_0$. Then the divisor $C_1+L'+V_3+V_2+B_1$ contracts to a point, so $V_2=[(2)_{v_2},v_3+2]$ for some $v_2\geq 0$ and $B_1^2=-v_2-3\leq -3$. In particular, $\psi$ touches $B_1$, so it contracts $V_1$. Similarly, because $V_0=[2]$, we get $V_1=[(2)_{v_1},3]$ for some $v_1\geq 0$. Note that the last component of $V_2$ is not contracted by $\psi$, because $\psi$ does not touch $L'$. Since $\psi$ contracts $V_1$, for $v_1=0$ we get $v_2=1$ or $v_2=v_3=0$, hence $\psi$ touches $L'$ or $\psi(B_2)^2=-3$, both possibilities being already ruled out. Thus $v_1\neq 0$. Because $A\cdot \Delta_0=1$, we get $v_2=0$. Then $\psi$ contracts exactly $A+[(2)_{v_1}]$, so $\psi(B_1)^2=B_1^2=-3$; a contradiction.

Consider the case $V_1=0$, i.e.\ $A\cdot B=1$. Now $L'$ is the unique branching component of $D_0-E_0$. Then $\psi$ contracts exactly $A$ and a maximal $(-2)$-twig contained in $V_2$. Because $\psi$ does not touch $B_2$, we have $V_2=[(2)_{v_2},3,(2)_{v_2'}]$ for some $v_2\geq 1$. The maximal twig of $D_0-E_0$ containing $B_1$ is $[2,v_2+3,(2)_{v_3}]$. We check easily that if $v_3=0$ then $D_0-E_0$ is not contractible to a point. Therefore, $v_3\geq 1$. The contractibility forces $v_2'=0$ and $v_3=1$. Then $\#(D_1-E_1)=6$, which is impossible, because $\#(D_1-E_1)=\rho(X_1)=11$.
\end{proof}

\bcl $D-E$ is a chain. \ecl

\begin{proof} Suppose $D-E$ contains a branching component. Its image on $X_0$, which we denote by $B$, is a branching component of $D_0-E_0$, because $s_1=0$. Now $D_0-E_0$ has at least three tips, so $D_1-E_1$ has at least one tip. From the previous claim it follows that it has exactly one, so $D_0-E_0$ has exactly three and $A$ meets two of them. Denote the maximal twigs of $D_0-E_0$ by $V_1$, $V_2$ and $V_3$, where $V_2$ meets $A$ and $V_1$ contains $C_1$. Suppose $C_1$ is a tip of $D_0-E_0$. Then $A$ meets the tip of $V_3$. Let $U$ be a component of $V_2$ meeting $\psi(B)$ which is a part of a cycle contained in $D_0-E_0$. The contraction of $D_1-E_1-U$ maps $X_1$ onto $\PP^2$ and $\bar E$ onto a unicuspidal curve with a cusp of multiplicity three. By \ref{lem:min(mu)>=4} this is a contradiction. Thus $C_1$ is not a tip of $D_0-E_0$.

Suppose $A\cdot V_3=0$. Then $V_3$ is not touched by $\psi$, so $V_3=[(2)_{v_3}]$ for some $v_3\geq 1$. Because  $D_0-E_0$ contracts to a point, we have $V_2=[(2)_{v_2},v_3+2]$ for some $v_2\geq 0$. Because the chain joining $C_1$ and $B$ is not touched by $\psi$, it consists of $(-2)$-curves, so $V_1=[(2)_{v_1},v_1'+2,1,(2)_{v_1'}]$ for some $v_1,v_1'\geq 0$. Then $B^2=-v_1-3\leq -3$, so $\psi$ touches $B$. Also, since $\cal C_+=0$, $\psi$ does not contract the component of $V_1$ of self-intersection $-v_1'-2$. It follows that $v_1\neq 0$, as otherwise $\psi$ would contract exactly $A+[(2)_{v_2}]$, so $\psi$ would not touch $B$. Then $v_1'\neq 0$. The curve $A$ meets exactly one component of $\Delta_0$, so we get $v_2=0$. But then $\psi$ touches $B$ at most once, so $\psi(B)^2\leq -v_1-2\leq -3$; a contradiction.

Thus $A\cdot V_3=1$. The twig $V_1$ is not touched by $\psi$, so $V_1-C_1$ consists of $(-2)$-curves. Because $C_1$ is not a tip of $D_0-E_0$, we have $V_1=[(2)_{v_1},1]$ for some $v_1\geq 1$. The divisor $D_1-E_1$ contracts to a cycle of rational curves, of which exactly all but one are $(-2)$-curves. It follows that $d(D_1-E_1)\neq 0$, so the components of $D_1-E_1$ are numerically independent. Because $D_1-E_1$ has $\rho(X_1)$ components, they are a basis of $\NS_\Q(X_1)$. Since the divisor $W=\psi(V_2+B+V_3)$ is a $(-2)$-cycle meeting $C_1'$ once, $K_1+W$ intersects all components of $D_1-E_1$ trivially, hence $K_1+W$ is numerically trivial. But its intersection with $K_1$ is $K_1\cdot(K_1+W)=K_1^2=\zeta-K_1\cdot E_1=3-\gamma_1<0$; a contradiction.
\end{proof}

We are left with the case when $D-E$ is a chain. Then $T_0=D_0-E_0$ is a chain. By Claim 3 $\#\cal L=0$, so $\eta=0$, so there exists a maximal $(-2)$-twig $[(2)_{v_1}]$, $v_1\geq 1$ in $T_0$, such that $A$ meets it in a tip of $D_0$. Let $V_0$ be the second component of $T_0$ met by $A$. Denote the maximal twigs of $T_0$ by $V_1, V_2$, where $V_1$ contains $[(2)_{v_1}]$. If $V_0$ is a tip of $T_0$ then $D_1-E_1$ is a $(-2)$-cycle, so the contraction of $D_1-E_1-L$ maps $X_1$ onto $\PP^2$ and $\bar E$ onto a unicuspidal rational curve with a cusp of multiplicity $\tau_1$. By claim 4 $\tau_1=2$, so \ref{lem:min(mu)>=4} implies that $V_0$ is not a tip of $T_0$. Since $\#\cal L=0$, $V_0$ does not meet $[(2)_{v_1}]$. Therefore, $\psi$ contracts exactly $A+[(2)_{v_1}]$. Then $V_0^2=-v_1-3$ and $T_0-C_1-V_0$ consists of one $(-3)$-curve and some number of $(-2)$-curves, so $T_0=[(2)_{v_1},3,(2)_{v_1'},1,(2)_{v_2'},v_1+3,(2)_{v_2}]$ or $T_0=[(2)_{v_1},3,(2)_{v_1'},v_1+3,(2)_{v_2'},1,(2)_{v_2}]$ for some $v_1',v_2'\geq 0$ and $v_2\geq 1$. In the second case the contractibility of $T_0$ to a point forces $v_\neq 0$, hence $v_2'=0$ and $v_2=v_1+1$, which implies that $T_0$ contracts to $[(2)_{v_1},3,(2)_{v_1'},1]$ and hence to a $(-2)$-chain; a contradiction. Similarly, if in the first case $v_2'=0$ then $v_1'=v_1+1$ and $v_2=0$, which is false. Thus $v_2'>0$. Then $v_1'=0$ and $v_2'=1$. Now again the contraction of $D_1-E_1-L$ maps $X_1$ onto $\PP^2$ and $\bar E$ onto a unicuspidal curve with a semi-ordinary cusps. By \ref{lem:min(mu)>=4} this is a contradiction.
\end{proof}

\med
\section{Process of length $n=0$.}\label{sec:n=0}

We keep the notation and assumptions from previous sections. In particular, the rational cuspidal curve $\bar E\subset \PP^2$ violates the Coolidge-Nagata conjecture and $\pi_0\:(X_0,D_0)\to (\PP^2,\bar E)$ is the minimal weak (see Subsection \ref{ssec:cuspidal}) resolution of singularities. By \ref{prop:fundamentals}(ii) $2K_0+E_0$ is effective. By \ref{prop:n=0} $2K_0+D_0^\flat$ is numerically effective. From the log MMP point of view this is the easiest possible situation ($n=0$). However, now the bounds on geometric parameters describing $\bar E\subset \PP^2$ are weaker, so ruling out remaining cases, and hence establishing the Coolidge-Nagata conjecture, requires considerable effort.

Recall that, by definition $p_2(\PP^2,\bar E)=h^0(2K_X+D)$, where $\pi\:(X,D)\to (\PP^2,\bar E)$ is the minimal log resolution of singularities and that $\zeta=K\cdot (K+E)=K_0\cdot (K_0+E_0)$. The reduced total inverse image of the cusp $q_j\in \bar E$ is, as before, denoted by $Q_j$, the image of $Q_j$ on $X_0$ by $\wt Q_j$ and the unique $(-1)$-curve of this image by $C_j$. By \ref{cor:fE>=4}(iv) $\gamma_0=-E_0^2\geq 4$ and by \ref{thm:CN1}(c) $\bar E$ has at most two cusps, i.e. $c\leq 2$. Let us recollect some bounds we have obtained.

\blem\label{lem:case_n=0} With the above notation: \benum[(i)]

\item $\gamma_0+\tau^*\leq 2p_2(\PP^2,\bar E)+2\zeta$,

\item $|\zeta|\leq p_2(\PP^2,\bar E)-2\leq 2$,

\item $\ds \sum_{j=1}^c K\cdot Q_j = p_2(\PP^2,\bar E)-\zeta$,

\item  $\ind(D) \leq 5-p_2(\PP^2,\bar E)\leq 2$.
\eenum
\elem

\begin{proof}(i) Arguing as in \ref{cor:n=1_basics}(iii) we get $0\leq (2K_0+D_0^\flat)\cdot (2K_0+E_0)=(2K_0+D_0)\cdot (2K_0+E_0)=2\zeta-\gamma_0+2p_2(\PP^2,\bar E)-\tau^*$. For (ii) see \ref{thm:CN1}(d) and (b). (iii) $\sum_{j=1}^c K\cdot Q_j =K\cdot (D-E)=K\cdot (K+D)-K\cdot (K+E)=p_2(\PP^2,\bar E)-\zeta$. (iv) is \ref{lem:mmp_basics}(iv).
\end{proof}

Note that $K_0\cdot \wt Q_j=K\cdot Q_j-\tau_j^*-1$, so \ref{lem:case_n=0}(iii) gives $$K_0\cdot \sum_{j=1}^c(\wt Q_j - C_j)=p_2(\PP^2,\bar E)-\zeta-\tau^*.$$

\blem\label{lem:zeta>=0_for_p2=3} If $p_2(\PP^2,\bar E)=3$ then $\zeta\neq -1$. \elem

\begin{proof} Suppose $\zeta=-1$. By \ref{lem:case_n=0} $\ind(D)\leq 2$ and $K_0\cdot \sum_{j=1}^c (\wt Q_j-C_j)=4-\tau^*$. Also, $\gamma_0+\tau^*\leq 4$, so $\gamma_0=4$ and $\tau^*=0$, hence $C_j\cdot E_0=\tau_j=2$ and $s_j=1$ for every $j\leq c$. In particular, $C_j$ is a tip of $\wt Q_j$. Furthermore, $K\cdot Q_1+K\cdot Q_2=4$, so $K_0\cdot \sum_{j=1}^c(\wt Q_j-C_j)=4$. We compute $K_0^2=\zeta-K_0\cdot E_0=-3$, hence by the Noether formula $\rho(X_0)=13$, so by \eqref{eq:rho(Xi)+i=|D|} $\#\wt Q_1+\#\wt Q_2=12$.

First consider the case $c=1$. Let $H$ be the unique component of $\wt Q_1-C_1$ meeting $C_1$. Since $\wt Q_1$ contracts to a smooth point, $H^2=-2$. Put $V=\wt Q_1-C_1-H$. Clearly, $K_0\cdot V=4$ and $\#V=10$. The divisor $V$ is vertical for the elliptic fibration $p\:X_0\to \PP^1$ induced by $|E_0+2C_1|$ (see \ref{lem:elliptic_ruling}). Since $K_0^2=-3$, by \ref{lem:elliptic_fibers} there exists a birational morphism $\alpha\:X_0\to Z$ which contracts $C_1$ and two other vertical components, such that the induced elliptic fibration $p_Z\:Z\to \PP^2$ is minimal. Note that if $L\subset X_0$ is a vertical $(-1)$-curve then it does not meet $H$, because otherwise the intersection of the nef divisor $C_1+H+L$ with the effective divisor $2K_0+E_0$ would be $-4+\tau_1<0$, which is impossible. The divisor $V$ has $\beta_{\wt Q_1}(H)-1\leq 2$ connected components. It does not contains fibers, because its intersection matrix is negative definite.

Suppose $V$ is not contained in one fiber. Denote the connected components of $V$ by $V_1$ and $V_2$. For $i=1,2$ let $F_i$ be the fiber containing $V_i$ and let $L_i$ be a component of $F_i-V_i$. Clearly, $L_i\not \subset D_0$. If, say, $F_1\neq V_1+L_1$ then the intersection matrix of $C_1+V+L_1$ is negative definite of rank $12$, so since $\rho(X_0)=13$, the components of $E_0+C_1+V+L_1$ are a basis of $\NS_\Q(X_0)$, contradicting the Hodge index theorem. Therefore, $F_i=V_i+L_i$ for $i=1,2$. The argument shows also that there are no singular fibers other than $E_0+C_1$, $F_1$ and $F_2$. Consider first the case when both $L_i$ are $(-1)$-curves. Since $\alpha_*F_i$ are minimal, $K_Z\cdot \alpha_*(F_1+F_2)=0$. But $K_Z\cdot \alpha_*(F_1+F_2)=\alpha^*K_Z\cdot V=K_0\cdot V-(L_1+L_2)\cdot V$, so $(L_1+L_2)\cdot V=4$. We have $L_i\cdot V=L_i\cdot D_0$, so by \ref{prop:fundamentals}(iii) $L_i$ meets $V$ in exactly two points and transversally. Let $W$ be the image of $X$ after the contraction of the proper transforms of $L_i$ and let $D_W$ be the image of $D$. The pair $(W,D_W)$ is an snc-minimal smooth completion of $W_0=X_0\setminus (D_0+L_1+L_2)$. We have $\kappa(W_0)\geq \kappa(X_0\setminus D_0)\geq 0$ and $(K_W+D_W)^2=(K+D+(\varphi^{-1})_*(L_1+L_2))^2=(K+D)^2+2=K\cdot (K+D)=p_2(\PP^2,\bar E)=3$. Since $W_0$ contains no lines, \ref{lem:BMY} gives $(K_W+D_W)^2+\ind(D_W)\leq 3\chi(W_0)=3\chi(X_0\setminus D_0)=3$, hence $\ind(D_W)=0$. But the image of $E$ is a maximal twig of $D_W$, so $\ind(D_W)>0$; a contradiction. It remains to consider the case when, say, $L_1$ is not a $(-1)$-curve. Then $L_2$ is a $(-1)$-curve and $V_1+L_1$ consists of $(-2)$-curves. Because $\wt Q_1$ contracts to a smooth point, $V_1$ is a $(-2)$-chain meeting $H$ in a tip. By \ref{prop:fundamentals}(iii) $L_i\cap D_0\geq 2$. Since $L_2\cdot H=0$, we get $L_2\cdot V_2\geq 2$. We have in fact $L_2\cdot V_2=2$. Indeed, if $L_2\cdot V_2\geq 3$ then $\alpha_*F_2$ is not snc, so by \ref{lem:elliptic_fibers}(ii) $F_2$ is a fork with three tips, hence $V_2$ is not connected, which is false. Let $W$ be the image of $X$ after snc-minimalization of $D+L_2'$. This minimalization factorizes $\alpha$ (hence is a composition of at most two contractions) and, because of the connectedness of $V_2$, is inner for $D+L_2'$. The pair $(W,D_W)$ is an snc-minimal smooth completion of $W_0=X_0\setminus (D_0+L_2)$. We have $\kappa(W_0)\geq \kappa(X_0\setminus D_0)\geq 0$ and $(K_W+D_W)^2=(K+D+L_2')^2=(K+D)^2+1=K\cdot (K+D)-1=p_2(\PP^2,\bar E)-1=2$. By \ref{lem:BMY} $\ind(D_W)\leq 1$. However, the images of $E$, $V_2$ and the $(-2)$-tip of $D$ contracted by $\varphi_1$ are maximal twigs of $D_W$, hence $\ind(D_W)\geq \frac{1}{6}+\frac{1}{2}+\frac{1}{2}$; a contradiction.

Thus $V$ is properly contained in some fiber $F$. Let $L$ be a component of $F-V$. We have in fact $F=V+L$, as otherwise the intersection matrix of $C_1+V+L$ is negative definite of rank $12$ and we get a contradiction with the Hodge index theorem as before. Then $F$ is the unique singular fiber other than $E_0+C_1$ and hence $\alpha$ contracts $L$ and one irreducible component $V_0$ of $V$.  It follows that $L^2=-1$, $V_0^2=-2$ and $L\cdot V_0=1$. Put $V'=V-V_0$. We have $\alpha^*K_Z\cdot V'=K_Z\cdot \alpha_*F_1=0$. Because $\alpha^*K_Z=K_0-V_0-2L-C_1$, we obtain $(V_0+2L)\cdot V'=K_0\cdot V'=4$. If $L\cdot V'>1$ then $\alpha_*F$ is not snc, so by \ref{lem:elliptic_fibers}(ii) $\#\alpha_*F\leq 3$, hence $10=\#V=\#\alpha_*F+1\leq 4$; a contradiction. Since $X_0\setminus D_0$ contains no lines, $L\cdot V=L\cdot D_0\geq 2$, so $L\cdot V'\geq 1$. We obtain $L\cdot V'=1$ and $V_0\cdot V'=2$. Then $\alpha_*F$ is not snc and has $\#V-1=9$ components. A contradiction by \ref{lem:elliptic_fibers}(ii).

Now consider the case $c=2$. By \ref{rmk:ind>half} the contribution of the first branch of each $Q_j$ to $\ind(D)$ is more than $\frac{1}{2}$. Since $\varphi_1$ contracts two $(-2)$-tips of $D$, in case both $Q_j$ are branched we get $\ind(D)>4\cdot\frac{1}{2}=2$, which contradicts \ref{lem:case_n=0}(iv). Thus, we may assume $Q_2$, and hence $\wt Q_2$, is a chain. Because $\tau_2^*=0$, $q_2\in \bar E$ is a semi-ordinary cusp, hence $Q_2=[2,1,3,(2)_k]$ for some $k\geq 0$. Let $\ind_j$ be the contribution of the maximal twigs of $D$ contained in $Q_j$ to $\ind(D)$. We get $\ind_1\leq 2-\ind_2=\frac{1}{2}+\frac{2}{2k+3}$. Since $K_0\cdot \wt Q_1=2-K_0\cdot \wt Q_2=3$, the cusp $q_1\in\bar E$ is not semi-ordinary, hence $\wt Q_1$ is not a chain. Then $Q_1$ has at least two branches. Let $\ind_1'$ be the contribution of the two maximal twigs of $D$ contained in the first branch of $Q_1$. We have $\ind_1\geq \frac{1}{2}+\ind_1'>1$, hence $k=0$ and $\ind_1'\leq \frac{2}{3}$. If $Q_1$ has more than one branching component then, assuming that $V_0$ is a tip of $Q_1$ contained in the second branch, we get $\frac{7}{6}\geq \ind_1\geq \ind_1'+\frac{1}{-V_0^2}+\frac{1}{2}>1+\frac{1}{-V_0^2}$, so $V_0^2<-6$. But because $Q_1$ contains a $(-3)$-curve, we would get $K\cdot Q_1+K\cdot Q_2=K\cdot Q_1\geq K\cdot V_0\geq 5$, which is false.

We infer that $Q_1$ has exactly one branching component. We have $\#\wt Q_1=12-\#\wt Q_2=11$. Let $\beta\:X_0\to Z$ be the contraction of the maximal twig of $\wt Q_1$ containing $C_1$ (it is of type $[1,(2)_m]$ for some $m\geq 0$). Then $\beta_*\wt Q_1$ is a chain with $K_Z\cdot \beta_*\wt Q_1=K_0\cdot \wt Q_1=3$, hence it is as in \ref{lem:Q_with_small_KQ}(v). Since $\ind_1'\leq \frac{2}{3}$, it is as in (v.8) with $k=0$, i.e. $\beta_*\wt Q_1=[3,2,2,2,1,5]$. Put $V=\wt Q_1-C_1$. The divisor $V$ has $10$ components, is connected and vertical for the elliptic fibration $p\:X_0\to \PP^2$ induced by $|E_0+2C_2|$. Let $F$ be the fiber containing $V$. The intersection matrix of $V$ is negative definite, so $F\neq V$. Let $L$ be a component of $F-V$. Since $K_0^2=-3$, there is a birational morphism $\alpha\:X_0\to Z$ which contracts $C_2$ and two other vertical components, such that the induced elliptic fibration $p_Z\:Z\to \PP^2$ is minimal. If $\#F>11$ then the intersection matrix of $C_2+V+L$ is negative definite of rank $12=\rho(X_0)-1$, so we have a contradiction with the Hodge index theorem as before. Therefore, $F=V+L$. The argument with Hodge theorem shows also that there are no singular fibers other than $F$ and $E_0+C_2$. It follows that $L$ is a $(-1)$-curve. Since $\#\alpha_*F\geq \#F-2=9$, \ref{lem:elliptic_fibers} implies that $\alpha_*F$ is snc, hence $L$ meets $V$ transversally, in two points belonging to two different components of $V$.  If $L$ does not meet a $(-2)$-curve in $V$ then, by the explicit description of $\wt Q_1$ we have, after the contraction of $L$ the fiber contains no $(-1)$-curve and does not consist of $(-2)$-curves, which is in contradiction with \ref{lem:elliptic_fibers}(i). Thus $L$ meets a $(-2)$-curve in $V$ and hence there is a chain of $(-2)$-curves $f\subset V$, whose tips meet $C_1$ and $L$. Then $C_1+f+L$ is nef, so $0\leq (C_1+f+L)\cdot (2K_0+E_0)=-4+\tau_1$; a contradiction.
\end{proof}

\med

Since $n=0$, the minimalization process for $(X_0,\frac{1}{2}D_0)$ (see Subsection \ref{ssec:logMMP}) contracts only curves in $D_0$, hence the log MMP has not much more to say. Still, we can minimalize the pair $(X_0,E_0)$. Let $$(Y_0=X_0,E_0)\xra{\sigma_1} (Y_1,E_1)\xra{\sigma_2} \ldots \xra{\sigma_t} (Y_t,E_t)$$ with $E_{i+1}=(\sigma_{i+1})_*E_i$ be the process of almost minimalization of the pair $(X_0,E_0)$, i.e. a maximal sequence of blowdowns, such that $\Exc \sigma_{i+1}\cdot E_i\leq 1$ and $E_t\neq 0$. Note that by \ref{prop:fundamentals}(ii) $2K_0+E_0\geq 0$, hence $K_i+E_i\geq 0$ (as a $\Q$-divisor). The definition implies that $Y_i$ and $E_i$ are smooth. Put $T=D-E=Q_1+Q_2$ and $T_0=D_0-E_0=\wt Q_1+\wt Q_2$. For $i=0,\ldots, t-1$ put $T_{i+1}=(\sigma_{i+1})_*T_i$, $L_i=\Exc \sigma_{i+1}\subset Y_i$, $K_i=K_{Y_i}$ and $\zeta_i=K_i\cdot (K_i+E_i)$. In particular, $\zeta=K\cdot (K+E)=\zeta_0$. By \cite[6.24]{Fujita-noncomplete_surfaces} $(K_t+E_t)^+=K_t+aE_t$, where $a=(1-\frac{2}{\gamma_t})$, $\gamma_t=-E_t^2$. For $j=0,1$ put $\Theta_j=\ds\sum_{i:L_i\cdot E_i=j}T_i\cdot L_i$ and $\theta_j=\#\{i:L_i\cdot E_i=j\}$. Clearly, $\theta_0+\theta_1=t$.

\blem\label{lem:(X,E)} For each $i$ we have $L_i\cdot (T_i+E_i)\geq 2$. In particular, $\Theta_0\geq 2\theta_0$ and $\Theta_1\geq \theta_1$. Furthermore, the following equations hold:
\beq \zeta_t=\zeta+\theta_0 \text{\ \ and \ \ }\gamma_t=\gamma_0-\theta_1\geq 4,\label{eq:1}\eeq
\beq p_2(\PP^2,\bar E)-\zeta+(1-\frac{4}{\gamma_t})c\geq \frac{2}{\gamma_t}(\Theta_1+\tau^*)+\Theta_0,\label{eq:2}\eeq  and \beq \zeta_t=((K_t+E_t)^+)^2+2-\frac{4}{\gamma_t}\geq 2-\frac{4}{\gamma_t}.\label{eq:3}\eeq
\elem

\begin{proof} From the definition of $\sigma_{i+1}$ we see that every pair $(Y_i,E_i)$ is smooth. The equations \eqref{eq:1} follow from the definition and basic properties of blowups. By \ref{cor:fE>=4}(iv) $\gamma_t\geq 4$. We have $((K_t+E_t)^+)^2=(K_t+E_t)^2-(\frac{2}{\gamma_t} E_t)^2=\zeta_t-2+\frac{4}{\gamma_t}$, hence \eqref{eq:3}.

We show that \eqref{eq:2} is equivalent to $T_t\cdot (K_t+E_t)^+\geq 0$. Put $t_i=T_i\cdot L_i$. We compute $$T_{i}\cdot (K_{i}+aE_i)-T_{i+1}\cdot (K_{i+1}+a E_{i+1})=T_{i}\cdot (K_{i}-\sigma_{i+1}^*K_{i+1})+aT_{i}\cdot (E_{i}-\sigma_{i+1}^*E_{i+1})=$$ $$=T_{i}\cdot L_i-a T_{i}\cdot (L_i\cdot E_i) L_i=t_i(1-aL_i\cdot E_i).$$ Since $T_t\cdot (K_t+aE_t)=T_t\cdot (K_t+E_t)^+$, we get $$T_0\cdot ( K_0+aE_0)-T_t\cdot (K_t+E_t)^+=\Theta_0+(1-a)\Theta_1,$$ so $T_0\cdot K_0+aT_0\cdot E_0\geq \Theta_0+ (1-a)\Theta_1.$ We have $T_0\cdot K_0+\tau^*+c=T_0\cdot K_0+\tau-s=T\cdot K=p_2(\PP^2,\bar E)-\zeta$ and $T_0\cdot E_0=\tau+c-s=\tau^*+2c$, hence $$T_0\cdot K_0+aT_0\cdot E_0=(p_2(\PP^2,\bar E)-\zeta-\tau^*-c)+ (1-\frac{2}{\gamma_t})(\tau^*+2c)=p_2(\PP^2, \bar E)-\zeta+(1-\frac{4}{\gamma_t})c-\frac{2}{\gamma_t}\tau^*.$$ The inequality \eqref{eq:2} follows.

\smallskip For every $i$ put $D_i=T_i+E_i$. It remains to prove that $L_i\cdot D_i\geq 2$.

\setcounter{claim}{0}
\bcl  If $L\subset X_0$ is a $(-1)$-curve for which $L\cdot E_0\leq 1$ then $L\not\subset D_0$. Moreover, if $L\cdot D_0\leq 2$ then $L\cdot (\Upsilon_0+\Delta_0)=0$. \ecl

\begin{proof} All $(-1)$-curves in $D_0$ meet $E_0$ at least twice, so $L$ is not a component of $D_0$. By \ref{prop:n=0} $2K_0+D_0^\flat$ is nef, so $L\cdot D_0-2=L\cdot (2K_0+D_0)\geq L\cdot (\Upsilon_0+\Delta_0^+)+L\cdot \Bk_{D_0}\Delta_0^-$ and the right hand side of the inequality is positive if $L$ meets $\Upsilon_0+\Delta_0$.
\end{proof}

For $p\in D_i$ denote by $\Exc_p\subset X_0$ the part of the reduced exceptional divisor of the contraction $(X_0,D_0)\to (Y_{i},D_{i})$ lying over $p$.

\bcl Let $p$ be a point of normal crossings of $D_i$ for which $\Exc_p\neq 0$. Then $p$ is not a smooth point of $D_i$, $D_i\setminus\{p\}$ is connected and $\Exc_p=[1]$.\ecl

\begin{proof} Let $U_1$ and $U_2$ be the analytic branches of $D_i$ at $p$. We assume $U_2=0$ if $p\in D_i$ is smooth. Let $L\subset X_0$ be some $(-1)$-curve in $\Exc_p$. Since $E_i$ is smooth, we have $L\cdot E_0\leq 1$, so $L$ is not a component of $D_0$. Because $p$ is a point of normal crossings of $D_i$, we have $L\cdot D_0\leq 2$ and the intersection is transversal. Since $\PP^2\setminus\bar E$ contains no affine lines, we get $L\cdot D_0=2$. Since $D_0$ is connected and $X_0\setminus D_0$ is affine, all blowups over $p$ are inner for $U_1+U_2$ (hence $U_2\neq 0$). We infer that $\Exc_p$ is a chain and $\Exc_p-L$ is contained in $D_0$. If $\Exc_p-L$ is nonzero then it contains a $(-2)$-curve meeting $L$, so $L\cdot \Delta_0>0$, which contradicts Claim 1. Therefore, $\Exc_p=[1]$. Since $D_0$ is connected, $D_i\setminus \{p\}$ is connected.
\end{proof}

Suppose now that $L_i\cdot D_i\leq 1$ for some $i$. We have $L_i\subset T_i$, because otherwise the proper transform of $L_i$ on $\PP^2\setminus \bar E$ is an affine line, which is impossible by \ref{prop:fundamentals}(iii). Recall that $(-1)$-curves in $T_0$ meet $E_0$ at least twice, so the proper transform of $L_i$ on $X_0$ is not a $(-1)$-curve. Hence a contraction of some $L_j$ with $j<i$ touches the proper transform of $L_i$. Note that the connected components of the exceptional divisor of the morphism $X_0\to Y_t$ contract to smooth points on $Y_t$, so their structure is well known. Let $L_i'$ be the proper transform of $L_i$ on $X_{i-1}$. By renaming $L_j$'s with $j<i$ if necessary we may assume that the contraction of $L_{i-1}$ touches $L_i'$. Then $\sigma_{i-1}^*L_i=L_i'+L_{i-1}=[2,1]$. By induction, we may assume that $L_j\cdot D_j\geq 2$ for every $j<i$. Using the projection formula we compute $$1\geq L_i\cdot D_i=L_i'\cdot D_{i-1}+L_{i-1}\cdot D_{i-1}=\beta_{D_{i-1}}(L_i')+(L_{i-1}\cdot D_{i-1}-2)\geq \beta_{D_{i-1}}(L_i').$$ But since $D_{i-1}$ is connected and $E_{i-1}\neq 0$ we have $\beta_{D_{i-1}}(L_i')\geq 1$, hence $\beta_{D_{i-1}}(L_i')=1$ and $L_{i-1}\cdot (D_{i-1}-L_i')=2-1=1$. If $L_{i-1}\not \subset D_{i-1}$ then applying Claim 2 to the image of $L_{i-1}$ on $D_i$ we see that the exceptional divisor over this point equals (the proper transform of) $L_{i-1}$, hence there is a $(-1)$-curve $L_{i-1}$ on $X_0$, for which $L_{i-1}\cdot D_0=2$ and (since $L_i'$ is a $(-2)$-tip of $D_i$) $L_{i-1}\cdot \Delta_0\geq 1$, in contradiction to Claim 1. Thus $L_{i-1}$ is a component of $D_{i-1}$, the unique component of $D_{i-1}$ meeting $L_i'$.

Let $U$ be the subdivisor of $D_{i-1}-L_i'-L_{i-1}$ consisting of these component which meet $L_{i-1}$. Since  $L_{i-1}\cdot U=L_{i-1}\cdot (D_{i-1}-L_i'-L_{i-1})=2$, the intersection $L_{i-1}\cap U$ consists of at most two points. By Claim 2 for all points $p\in L_{i-1}\cup L_i'$ not belonging to $U$ we have $\Exc_p=0$. Blow up over $L_{i-1}\cap U$ until the reduced total transform $M$ of $L_i'+L_{i-1}$ contains only nc-points of the total transform of $L_i'+L_{i-1}+U$. Note that if a $(-1)$-curve on $D_j$ meets $E_j$ at most once then $j>0$, hence the created surface is dominated by $X_0$. Denote the proper transform of $U$ by $U'$. If $\#L_{i-1}\cap U=1$ and the analytic branch of $U$ at $L_{i-1}\cap U$ is irreducible then $M=[2,2,\ldots,2,3,1,2]$ and $U'$ meets $M$ once in the unique $(-1)$-curve $L_M\subset M$. In other cases $M=[2,2,\ldots,2,1]$ and $U$ meets $L_M$ in two points. By the connectedness of $D_0$ the divisor $M$ is contained in the image of $D_0$. Since $L_M$ meets the proper transform of $E_0$ at most once, at least one more blowup over some point $q\in L_M$ is needed to recover $X_0$. By Claim 2 $\Exc_q=[1]$, $M=[2,2,\ldots,2,1]$ and $q\in U'$. Clearly, $\Exc_q$ is not a component of $D_0$ and $\Exc_q\cdot \Delta_0>0$; a contradiction with Claim 1.
\end{proof}

\blem\label{lem:zeta>=1_for_p2=3} If $p_2(\PP^2,\bar E)=3$ then $\zeta\neq 0$. \elem

\setcounter{claim}{0}
\begin{proof} Suppose $p_2(\PP^2,\bar E)=3$ and $\zeta=0$. The inequality \eqref{eq:2} reads as $$3+(1-\frac{4}{\gamma_t})c\geq \frac{2}{\gamma_t}(\Theta_1+\tau^*)+\Theta_0.$$

\bcl $\Theta_0=2$ and $\gamma_t=4$. \ecl

\begin{proof}
If $\Theta_0\geq 4$ then the inequality above gives $\gamma_t(c-1)\geq 2(\Theta_1+\tau^*+2c)>0$, so by \ref{thm:CN1}(3) $c=2$ and then $\gamma_0\geq \gamma_t\geq 8$, in contradiction to \ref{lem:case_n=0}(i). Thus $\Theta_0\leq 3$ and hence $\theta_0\leq 1$. But by \ref{lem:(X,E)} $\theta_0=\zeta_t\geq 2-\frac{4}{\gamma_t}\geq 1$, so $\theta_0=\zeta_t=1$ and hence $\gamma_t=4$. The above inequality gives \beq \Theta_1+\tau^*+2\Theta_0\leq 6.\label{eq:Theta_for_p2=3}\eeq
Suppose $\Theta_0\neq 2$. Then $\Theta_0=3$, so $\Theta_1=\tau^*=0$, and consequently $\gamma_0=4$, $L_0\cdot E_0=0$ and $L_0\cdot D_0=3$. Let $p\:X_0\to \PP^1$ be the elliptic fibration given by the linear system $|2C_1+E_0|$ (cf. \ref{lem:elliptic_ruling}). Since $L_0+C_j$ intersects $2K_0+E_0$ negatively, it cannot be nef, so $L_0\cdot C_j=0$ for every $j\leq c$. We have $K_0^2=\zeta-K_0\cdot E_0=-2$, so if $\alpha\:X_0\to Z$ is the contraction of $C_1+L_0$ the fibers of the induced elliptic fibration of $Z$ are minimal. Let $F$ be the reduced fiber of $p$ containing $L_0$. If $\wt Q_1\neq C_1$ then the unique component $H$ of $\wt Q_1$ meeting $C_1$ is horizontal for $p$. Then $L_0\cdot H=0$, because otherwise $C_1+H+L_0$ would be a nef divisor whose interesection with $2K_0+E_0$ is $-4+(C_1+L_0)\cdot E_0=L_0\cdot E_0-2<0$, which is impossible. It follows that $L_0$ does not meet horizontal components of $D_0$. Since $L_0\cdot T_0=3$, $\alpha_*F$ is not snc, so by \ref{lem:elliptic_fibers} $F$ is a fork with three tips (components of $T_0$) and $L_0$ as a branching component. Since $t=\theta_0=1$, none of the tips is a $(-2)$-curve, so by \ref{lem:elliptic_fibers} all of them are $(-3)$-curves. In case $c=2$ they do not meet $C_2$, because otherwise $\wt Q_2$ would not be contractible to a smooth point. Since $D_0$ is connected, it follows that $H$ meets all three $(-3)$-twigs, so $H\cdot F\geq 3$. But $H\cdot F\leq H\cdot (2C_1+E_0)=2$; a contradiction.
\end{proof}

The inequality \eqref{eq:Theta_for_p2=3} gives $\Theta_1+\tau^*\leq 2$.

\bcl $\tau^*\in\{1,2\}$. \ecl

\begin{proof}
Suppose $\tau^*=0$. By \eqref{eq:1} $\gamma_0=4+\theta_1$. By \ref{cor:I and II}(ii) $\gamma_0+d^2\equiv 0 \mod 4$, so $\gamma_0$ is congruent to $0$ or $3$ mod $4$. Because $\gamma_0\in \{4,5,6\}$ by \ref{lem:case_n=0}(i), we get $\gamma_0=4$, which implies that $\theta_1=0$. We get $L_0\cdot E_0=0$. As before we show that $L_0\cdot C_j=0$ for all $j\leq c$ and $L_0\cdot H=0$, where $H=0$ if $\wt Q_1=C_1$ and $H$ is the unique component of $\wt Q_1-C_1$ meeting $C_1$ otherwise. Thus, $L_0$ does not meet horizontal components of $D_0$. Let $F$ be the reduced fiber of the elliptic fibration of $X_0$ induced by $|2C_1+E_0|$ which contains $L_0$. Since $K_0^2=\zeta-K_0\cdot E_0=-2$, after the contraction of $C_1+L_0$ all fibers become minimal. Let $V$ be the subdivisor of $D_0$ consisting of components meeting $L_0$. We have $L_0\cdot V=L_0\cdot T_0=\Theta_0=2$. If $V$ is irreducible then the image of $F$ after the contraction of $L_0$ contains a node or a cusp, hence is irreducible by \ref{lem:elliptic_fibers}(ii), which implies that $V$ is a $(-4)$-curve. If $\#V=2$ then the image is a minimal reducible fiber, so by \ref{lem:elliptic_fibers}(i) it consists of $(-2)$-curves, hence $V$ consists of two $(-3)$-curves. In both cases we get $K_0\cdot V=2$. By \ref{lem:case_n=0}(iii) $K\cdot (Q_1+Q_2)=3$, which gives $K_0\cdot \sum_{j=1}^c (\wt Q_j-C_j)=3$. It follows that $K_0\cdot (\sum_{j=1}^c (\wt Q_j-C_j)-V)=1$, so $\sum_{j=1}^c (\wt Q_j-C_j)-V$ contains a $(-3)$-curve not meeting $L_0$. This curve is horizontal, hence it is $H$. But if $H^2=-3$ then $\wt Q_1$ does not contract to a smooth point; a contradiction.
\end{proof}

It follows that $\Theta_1\leq 2-\tau^*\leq 1$.

\bcl $\Theta_1=0$. \ecl

\begin{proof}
Suppose $\Theta_1=1$. Then $\theta_1=1$ and $\tau^*=1$. Claim 1 gives $\theta_0=1$, so $t=2$. By \eqref{eq:1} $\gamma_0=5$, which gives $\rho(X_0)=10-K_0^2=10+K_0\cdot E_0-\zeta=13$. We may assume $L_0\cdot E_0=0$ and $L_1\cdot E_1=1$. Recall that the contraction of $L_0$ is denoted by $\sigma_1$. Put $\ti L_1=(\sigma^{-1})_*L_1$ and suppose $L_0\cdot \ti L_1\neq 0$. Then $L_0+\ti L_1=[1,2]$. Since $L_0\cdot T_0=2$, we have $1=L_1\cdot T_1=\ti L_1\cdot ((\sigma^{-1})_*T_1+2L_0)$, so $\ti L_1\cdot (\sigma^{-1})_*T_1=-1$. It follows that $L_1\subset T_1$ and that $\ti L_1$ is a $(-2)$-tip of $T_0$, so $L_0\cdot \Delta_0>0$. Then $L_0\cdot (2K_0+D_0^\flat)<L_0\cdot (2K_0+D_0)=0$, so $2K_0+D_0^\flat$ is not nef, in contradiction to \ref{prop:n=0}. Therefore, $L_0\cdot \ti L_1=0$. In particular, $\ti L_1$ is a $(-1)$-curve, hence is not a component of $D_0$. Let $L_i'$, $i=1,2$ be the proper transform of $L_i$ on $X$. We have $L_i'\cdot D=2$, so since $X\setminus D$ contains no affine lines, $L_i'$ meets $D$ in two different points. Let $(Y,D_Y)$ be the image of $(X,D)$ after the snc-minimalization of $D+L_1'+L_2'$. The minimalization morphism is inner, because $L_i'$'s are not contained in the twigs of $D+L_1'+L_2'$. We compute $(K_Y+D_Y)^2=(K+D)^2+2=K\cdot (K+D)=p_2(\PP^2,\bar E)=3$ and $\chi(Y\setminus D_Y)=\chi(\PP^2\setminus \bar E)=1$. By \ref{lem:BMY} $\ind(D_Y)=0$, which implies that $D_Y$, and hence $D+L_1'+L_2'$, has no tips. In particular, $D$ has at most four tips. Also, $s=0$, so $\tau=1+c$. We claim that there is no $(-2)$-curve in $D_0$ which meets $L_0$ or $\wt L_1$ once. Indeed, if $V\subset D_0$ is a $(-2)$-curve meeting $L_0$ once then, because $s=0$, we have $\theta_1\geq 2$ or $\theta_0\geq 2$, which is in contradiction with Claim 1 or with with the inequality $\Theta_1\leq 1$ respectively.

Suppose $D$ has $4$ tips. Then each of them meets some $L_i'$. It follows that $c=1$, as otherwise $L_1'$ does not meet $E$, which is impossible, because $L_1$ meets $E_0$. We obtain $\tau_1=\tau=2$. We infer also that $Q_1$ is a fork with no $(-2)$-tips and such that $K\cdot Q_1=3$ (see \ref{lem:case_n=0}(iii)). By \ref{lem:type} $Q_1$ is of type $(r_1,r_2)$ for some $r_1+r_2=4$, $r_1,r_2\geq 1$. But since it has no $(-2)$-tip, we have $r_1,r_2\geq 2$, so $r_1=r_2=2$. Since $r_2=2$, by \ref{lem:Q_with_small_KQ}(iii) the unique $(-1)$-curve in $Q_1$ meets a $(-3)$-tip of $Q_1$. This tip meets some $L_i'$ and hence $C_1$ (which is the image of this tip on $X_0$) meets some $L_i$. By \ref{cor:fE>=4}(ii) $4\leq (C_1'+L_i)\cdot E_0\leq \tau_1+1$; a contradiction.

Therefore, $D$ has $3$ tips, so $c=1$ and $Q_1$ is a chain with $K\cdot Q_1=3$. We have $s_1=0$ and $\tau_1=2$. In particular, $\wt Q_1$ is a chain with $K_0\cdot \wt Q_1=3-\tau_1+s_1=1$, so it is as in \ref{lem:Q_with_small_KQ}(iii) for some $k\geq 0$. Since $\#D_0=\rho(X_0)=13$ we get $k=7$. But because $L_0+L_1$ does not meet any $(-2)$-curve of $D_0$, the fact that $k$ is positive implies that $D_Y+L_0'+L_1'$ has a tip; a contradiction.
\end{proof}

We obtain $t=1$ and $\gamma_0=\gamma_1=4$. We have also $K\cdot (Q_1+Q_2)=3$ by \ref{lem:case_n=0}(iii). Let $V$ be the divisor consisting of components of $D_0-E_0$ of self-intersection smaller than $(-2)$. Since $K_0\cdot \sum_{j=1}^c (\wt Q_j-C_j)=3-\tau^*\leq 2$ (see the remark after \ref{lem:case_n=0}), Claim 2 gives $V=[3]$, $V=[3]+[3]$ or $V=[4]$.

\bcl $L_0\cdot (D_0-V-\sum_{j=1}^cC_j)=0$ and $L_0\cdot V_i\leq 1$ for every $(-3)$-curve $V_i\subset V$. \ecl

\begin{proof} Suppose $L_0\cdot M\geq 2$ for some component $M\subset D_0-\sum_{j=1}^cC_j$. Since $M^2\geq -4$, the divisor $M+2L_0$ is nef, so $0\leq (M+2L)\cdot (2K_0+E_0)=2K_0\cdot M+ E_0\cdot M-4\leq 2K_0\cdot M-3$, hence $K_0\cdot M >1$, which means that $M=V=[4]$. If $L_0$ meets some component $M$ of $D_0-V-\sum_{j=1}^cC_j$ then $M^2=-2$ and by the above argument $M\cdot L_0=1$, so since $M\cdot E_0\leq 1$, we get $t>1$, which is false.
\end{proof}

Let $L_0'\subset X$ be the proper transform of $L_0$ and let $\alpha\:(X,D+L_0')\to (Y,D_Y)$ be the contraction of $L_0'$. By Claim 4 $D_Y$ is snc-minimal. We compute $(K_Y+D_Y)^2=(K+D)^2+1=K\cdot (K+D)-1=p_2(\PP^2,\bar E)-1=2$ and $\chi(Y\setminus D_Y)=\chi(\PP^2\setminus \bar E)=1$. By \ref{lem:BMY} $$\ind(D_Y)\leq 1.$$  Note also that $K_0^2=\zeta-K_0\cdot E_0=-2$ and $\#D_0=\rho(X_0)=10-K_0^2=12$.

\bcl $L_0\cdot (D_0-V)=0$, $K_0\cdot V=2$ and $\tau^*=1$. \ecl

\begin{proof} Suppose $L_0\cdot C_1>0$. Then $C_1+L_0$ is nef, so $0\leq (2K_0+E_0)\cdot (C_1+L_0)=\tau_1-4\leq \tau^*+s_1-3$. By Claim $2$ we obtain $\tau^*=\tau_1^*=2$ and $s_1=1$, hence $V=[3]$. Also, $L_0\cdot C_2=0$ if $c=2$. Then $\wt Q_1$ is a fork with maximal twigs $[3]$, $[2]$ and $[(2)_k,1]$ for some $k\geq 0$. The Claim $4$ gives $L_0\cdot (D_0-C_1)=0$, so $\ind(D_Y)\geq \frac{1}{2}+\frac{3}{4}>1$; a contradiction. Thus $L_0\cdot C_j=0$ for $j\leq c$ and we obtain  $L_0\cdot (D-V)=0$. In case $V=[3]$ we get $L_0\cdot V=L_0\cdot D=2$, which is impossible by Claim $4$. Hence $2\leq K_0\cdot V=K_0\cdot \sum_{j=1}^c (\wt Q_j-C_j)=3-\tau^*$, so by Claim 2 $\tau^*=1$. Then $V=[3]+[3]$ or $V=[4]$.
\end{proof}

\bcl $c=1$. \ecl

\begin{proof}
Suppose $c=2$. We have $K_0\cdot (\wt Q_1+\wt Q_2)=3-\tau+s=3-\tau^*-c=0$. Suppose, say, $K_0\cdot \wt Q_2<0$. Then $K_0\cdot \wt Q_2=-1$, so $\wt Q_2=[1,(2)_k]$ for some $k\geq 0$, and $K_0\cdot \wt Q_1=1$. Let $\ind_j$ be the contribution to $\ind(D_Y)$ of the twigs of $D_Y$ whose proper transform is contained in $Q_j$. By Claim 5 $L_0\cdot \wt Q_2=0$, so $\ind_2>\frac{1}{2}$ and hence $\ind_1<\frac{1}{2}$. It follows that $\tau_1^*\neq 0$, so $\tau_1^*=1$ and $\tau_2^*=0$. We get $1\geq \ind(D_Y)\geq \ind_2=\frac{1}{2}+\frac{2k+1}{2k+3}$, so $k=0$ and hence $\ind_1\leq \frac{1}{6}$. Because $K\cdot V=2$ and $\tau\leq 3$, there is no tip in $D$ whose intersection with $K$ is more than $3$, hence there is no tip with self-intersection smaller than $-5$. By \ref{lem:ind>=} we see that $D+L_0'$ has no tips contained in $Q_1$. Therefore, $Q_1$ is a chain and $L_0'$ meets its tips. Then $\wt Q_1$ is a chain and $L_0$ meets its tips. We infer that $s_1=0$, that $V$ consists of two $(-3)$-curves and that they are tips of $\wt Q_1$. Since $K_0\cdot \wt Q_1=1$, by \ref{lem:Q_with_small_KQ}(iii) $\wt Q_1=[3,2,1,3]$. Then $\#D_0=6$; a contradiction. Thus $K_0\cdot \wt Q_j\geq 0$ for $j\leq c$.

We obtain $K_0\cdot \wt Q_1=K_0\cdot \wt Q_2=0$. Then $V=[3]+[3]$, where each $(-3)$-curve is contained in a different $\wt Q_j$. Since $\tau^*=1$, we may assume $\tau_2^*=0$, so $C_2$ is a tip of $\wt Q_2$. Then $\wt Q_2$ is a fork with $[2]$ and $[(2)_k,3]$ as two maximal twigs. Since $L_0$ does not meet $(-2)$-curves of $D_0$, $D_0+L_0$ has a $(-2)$-tip contained in $\wt Q_2$ and hence $D+L_0'$ has two $(-2)$-tips contained in $Q_2$, so $\ind_2\geq 1$. We get $\ind_1=0$, so $D+L_0'$ has no tips in $Q_1$. But by Claim 5 $L_0'$ meets $Q_1$ only in the proper transform of the $(-3)$-curve from $V$, so the latter is impossible; a contradiction.
\end{proof}

\bcl $s_1=0$ and $\tau_1=2$. \ecl

\begin{proof} Since $c=1$, we have $\tau_1=\tau^*+1+s_1=2+s_1\leq 3$. Suppose $s_1=1$. Then $C_1$ is a tip of $\wt Q_1$. Since $K_0\cdot \wt Q_1=K_0\cdot Q_1-\tau_1^*-1=1$, we see that $\wt Q_1$ is not a chain. The contributions to $\ind(D_Y)$ of $E$ and of the twig contracted by $\psi_0$ are $\frac{1}{\gamma_0+\tau_1}=\frac{1}{7}$ and $\frac{2}{3}$ respectively, so the contribution of the remaining twigs of $D+L_0'$ is at most $\frac{4}{21}$, which is smaller than $\frac{1}{5}$, hence by \ref{lem:ind>=} their tips have self-intersections at most $(-6)$. But if such a tip exists then, since it is not touched by $\psi_0$, $D_0$ would contain a component with self-intersection at most $(-6)$, which is false. Thus $D+L_0'$ has exactly two maximal twigs and hence $D$ has at most four and $E$ is one of them. Because $\wt Q_1$ is not a chain, $\wt Q_1$ has exactly three maximal twigs. By Claim 5 their tips are $C_1$ and two $(-3)$-curves  (components of $V$) meeting $L_0$. Because $K_0\cdot \wt Q_1=1$, the first branch of $\wt Q_1$ is $[3,2,2,3]$. Let $V_1$ be the $(-3)$-curve meeting the branching component of $\wt Q_1$ (see Fig.\ \ref{fig:3}).
\begin{figure}[h]\centering\includegraphics[scale=0.5]{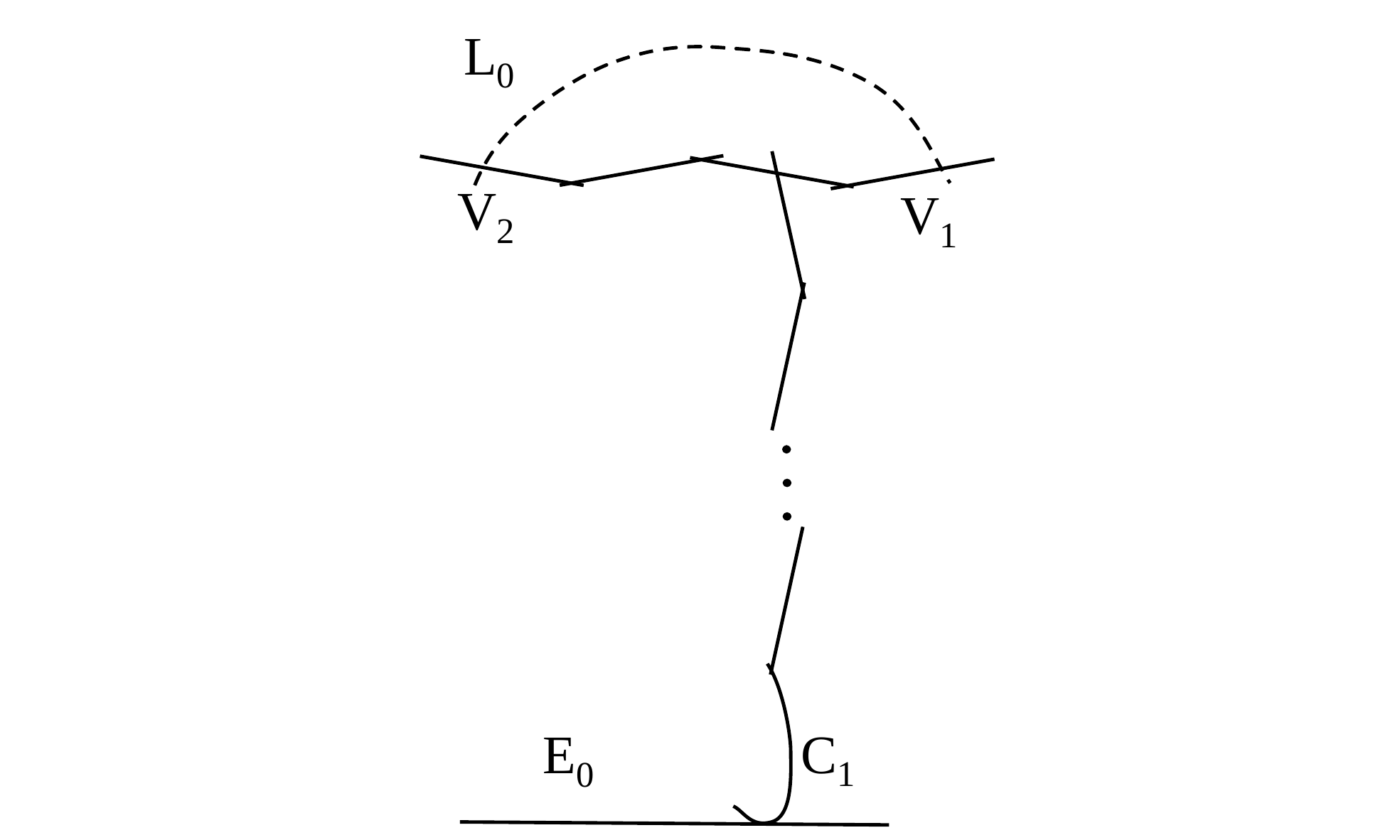}\caption{The divisor $D_0+L_0$ on $X_0$. Proof of Lemma \ref{lem:zeta>=1_for_p2=3}, Claim 7.}  \label{fig:3}\end{figure}
The contraction of $D_0-E_0+L_0-V_1$ maps $X_0$ onto $\PP^2$ and $E_0$ onto a unicuspidal curve with a cusp of multiplicity $\tau_1=3$; a contradiction with \ref{lem:min(mu)>=4}.
\end{proof}

By \ref{lem:type} $Q_1$ is of type $r=(r_1,\ldots,r_m)$ with $r_1+\ldots+r_m=4$. Because $E$ is a tip of $D_Y$, the inequality $\ind(D_Y)\leq 1$ implies that $D_Y$ has at most one $(-2)$-tip. By Claim 5 the number of $(-2)$-tips in $D$ and $D_Y$ is the same, so $D$ has at most one $(-2)$-tip, hence at most one $r_i$ equals $1$. Also, $r_m>1$, because $s_1=0$. Therefore, $r=(4)$, $(1,3)$ or $(2,2)$.

Suppose $r=(2,2)$. Since $\tau_1=2$, the second branch of $Q_1$ is $[(2)_{k_2},3,2,1,3]$, $k_2\geq -1$, so the second branch of $\wt Q_1$ is $[(2)_{k_2},2,1]$. Then $V$ is contained in the first branch. If this branch is $[(2)_{k_1},4,x,2,2]$, $x=2,3$, $k_1\geq 0$ then $V=[4]$ and we compute $\ind(D_Y)=(1-\frac{1}{k_1+1})+\frac{2}{3}+\frac{1}{3}+\frac{1}{\gamma}>1$, which is impossible. Thus the first branch of $\wt Q_1$ is $[(2)_{k_1},3,2,x,3]$, $x=2,3$, $k_1\geq 0$. The curve $L_0$ meats each of the $(-3)$-curves in this branch once. Let $V_1\subset V$ be the $(-3)$-tip meeting the branching component of $\wt Q_1$ (see Fig.\ \ref{fig:4}).
\begin{figure}[h]\centering\includegraphics[scale=0.5]{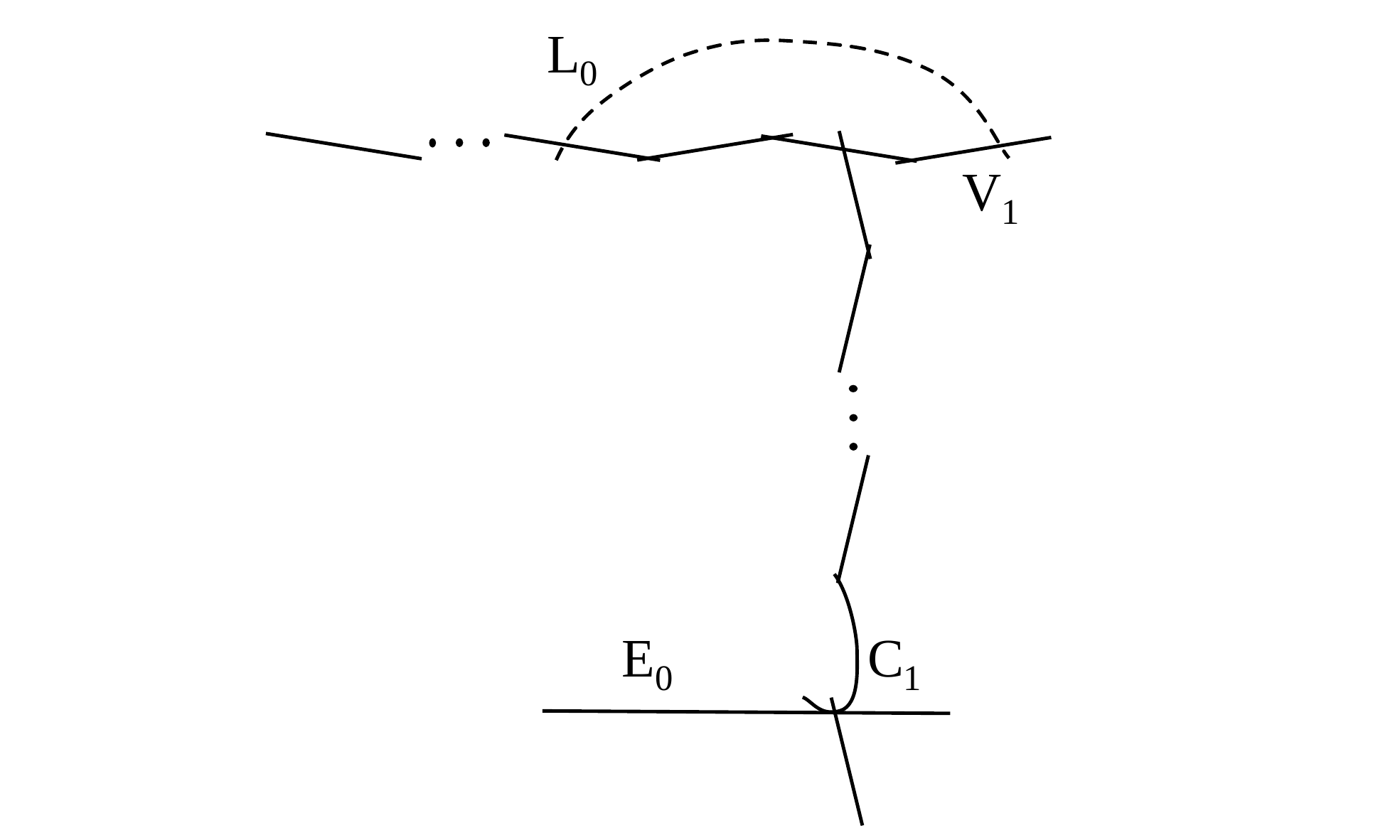}\caption{The divisor $D_0+L_0$ on $X_0$. Proof of Lemma \ref{lem:zeta>=1_for_p2=3}, Claim 7.}  \label{fig:4}\end{figure}
The contraction of $D_0-E_0-V_1+L_0$ maps $X_0$ onto $\PP^2$ and $E_0$ onto a unicuspidal curve with a cusp of multiplicity $3$, which contradicts \ref{lem:min(mu)>=4}.

Suppose $r=(1,3)$. Since $\tau_1=2$, the second branch of $Q_1$ is as in \ref{lem:Q_with_small_KQ} (iv.2) or (iv.3). Because $\gamma=\gamma_0+\tau_1=6$, in the first case $\ind(D_Y)\geq \frac{1}{2}+\frac{3}{5}+\frac{1}{6}$ and in the second $\ind(D_Y)\geq \frac{1}{2}+\frac{2}{5}+\frac{1}{6}$. In both cases $\ind(D_Y)>1$; a contradiction.

Thus $Q_1$ is a chain of type $(4)$. Since $K_0\cdot\wt Q_1=1$, $\wt Q_1$ is as in \ref{lem:Q_with_small_KQ}(iii.1) or (iii.2). Since $\#D_0=12$, we get $k=7$. Then $1\geq \ind(D_Y)>(1-\frac{1}{k+1}) +\frac{1}{\gamma}= 1-\frac{1}{8}+ \frac{1}{6}$; a contradiction.
\end{proof}

\bprop\label{prop:p2=4_for_n=0} $p_2(\PP^2,\bar E)=4$.  \eprop

\setcounter{claim}{0}
\begin{proof}By \ref{thm:CN1}(b) $p_2(\PP^2,\bar E)\in\{3,4\}$. Suppose $p_2(\PP^2,\bar E)=3$. By \ref{lem:case_n=0}(ii) $\zeta\in\{-1,0,1\}$. By \ref{lem:zeta>=0_for_p2=3} and \ref{lem:zeta>=1_for_p2=3} $\zeta=1$. Then \ref{lem:case_n=0}(iii) gives $K\cdot (D-E)=2$. We have $\rho(X)=10-K^2=8-\zeta+\gamma=7+\gamma$, hence $$\#D=7+\gamma.$$

Suppose $c=2$. Let $r=(r_1,\ldots,r_{m_1})$ and $v=(v_1,\ldots,v_{m_2})$ be the types of $Q_1$ and $Q_2$ respectively. We may assume $K\cdot Q_1\leq K\cdot Q_2$. Since $K\cdot Q_1+K\cdot Q_2=2$, we have $(\ds \sum_{i=1}^{m_1} r_i,\sum_{i=1}^{m_2} v_i)\in\{(1,3),(2,2)\}$. If $(r,v)=((1),(1,1,1))$ or $(r,v)=((1),(2,1))$ or $(r,v)=((1,1),(1,1))$ or $(r,v)=((1,1),(2))$ then we check easily using the explicit description in \ref{lem:Q_with_small_KQ} that $\ind(D)>2$, which contradicts \ref{lem:case_n=0}(iv). In case $(r,v)=((1),(1,2))$ the inequality $\ind(D)\leq 2$ holds only if $Q_1=[3,1,2]$, the first branch of $Q_2$ is $[3,x,2]$ and the second is $[(2)_{k-1},3,2,1,3]$ for some $k\geq 0$, where as usually by $[(2)_{-1},3,2,1,3]$ we mean $[2,1,3]$ (note that $x=2$, unless $k=0$). The characteristic pairs of the two cusps of $\bar E$ are $\binom{3}{2}$ and $\binom{9}{6},\binom{3}{3}_{k},\binom{3}{2}$, hence in the notation of \ref{cor:I and II} $I(q_1)=6$, $M(q_1)=4$, $I(q_2)=9k+60$, $M(q_2)=3k+16$. The equations \ref{cor:I and II}(i) and (ii) taken modulo $3$ give $\gamma\equiv 1$ and $\gamma+d^2\equiv 0$, hence $d^2\equiv 2 \mod 3$, which is a contradiction.

We are left with the cases when both $Q_1$ and $Q_2$ are chains. Then $r=(r_1)$, $v=(v_1)$ and $r_1+v_1=4$. For $i=1,2$ let $\ind_j$ be the contribution to $\ind(D)$ of the maximal twigs of $D$ contained in $Q_j$.

Consider the case $r_1=1, v_1=3$. Then $Q_1=[(2)_{k_1},3,1,2]$ for some $k_1\geq 0$, so $\ind_1\geq \frac{5}{6}$ and $\ind_2\leq 2-\ind_1=\frac{7}{6}$. In case $Q_2$ is as in \ref{lem:Q_with_small_KQ} (iv.1) or (iv.2) or (iv.3) with $k\geq 0$ we have $\ind_2\geq \frac{31}{40}$, so the inequality $\ind_2\leq \frac{7}{6}$ gives $k\leq 2$. Then $k_1+3=\#Q_1=\#D-1-\#Q_2=\gamma+1-k\geq \gamma-1\geq \gamma_0+3\geq 7$, so $\ind_1=\frac{1}{2}+\frac{2k_1+1}{2k_1+3}\geq \frac{29}{22}$ and consequently $\ind(D)\geq \frac{29}{22}+\frac{31}{40}>2$; a contradiction. Thus $Q_2=[(2)_{k_2},3,2,2,1,4]$, $k_2\geq 0$ and $Q_1=[(2)_{k_1},3,1,2]$, $k_1\geq 0$. Since $\#D=7+\gamma$, we have $k_1+k_2=\gamma-2=\gamma_0+3\geq 7$. The inequality $\ind(D)\leq 2$ gives $\frac{3}{8}\leq \frac{1}{2k_1+3}+\frac{2}{4k_2+7}$, which under the assumption $k_1+k_2\geq 7$ fails for $k_1\neq 0$. Thus $k_1=0$ and hence $\gamma=k_2+2$. The characteristic pairs of the two cusps of $\bar E$ are $\binom{3}{2}$ and $\binom{4k_2+7}{4}$, hence in the notation of \ref{cor:I and II} $M(q_1)=4$ and $M(q_2)=4k_2+10$. The equation \ref{cor:I and II}(i) gives $\gamma+3d=4k_2+16$, hence $3(d-k_2)=14$; a contradiction.

Consider the case $r_1=v_1=2$. Suppose $Q_1$ is as in \ref{lem:Q_with_small_KQ} (iii.1), i.e.\ $Q_1=[(2)_{k_1},4,1,2,2]$ for some $k_1\geq 0$. Then $\ind_1=\frac{2}{3}+\frac{3k_1+1}{3k_1+4}\geq \frac{11}{12}$. We have also $Q_2=[(2)_{k_2},4,1,2,2]$ or $[(2)_{k_2},3,2,1,3]$ for some $k_2\geq 0$, so $k_1+k_2=\#D-9=\gamma-2=\gamma_0+\tau-2\geq \tau+2\geq 7$. In the first case the inequality $\ind(D)\leq 2$ is equivalent to $\frac{1}{3k_1+4}+\frac{1}{3k_2+4}\geq \frac{4}{9}$, which is false for $(k_1,k_2)\neq (0,0)$. In the second case it is equivalent to $\frac{1}{3k_1+4}+\frac{1}{3k_2+5}\geq \frac{1}{3}$, which is also inconsistent with the inequality $k_1+k_2\geq 7$. Thus we may assume both $Q_j$'s are as in (iii.2), i.e.\ $Q_1=[(2)_{k_1},3,2,1,3]$ and $Q_2=[(2)_{k_2},3,2,1,3]$ for some $k_2\geq k_1\geq 0$, such that $k_1+k_2=\gamma-2\geq 6$. Then the inequality $\ind(D)\leq 2$ gives $\frac{2}{9}\leq \frac{1}{3k_1+5}+\frac{1}{3k_2+5}$, which for $k_1\neq 0$ is inconsistent with the inequality $k_1+k_2\geq 6$. Thus $k_1=0$ and $\gamma=k_2+2$. The characteristic pairs of the two cusps of $\bar E$ are $\binom{5}{3}$ and $\binom{3k_2+5}{3}$, hence in the notation of \ref{cor:I and II} $M(q_1)=7$, $I(q_1)=15$, $M(q_2)=3k_2+7$ and $I(q_2)=9k_2+15$. The equation \ref{cor:I and II}(iii) gives $d^2-3d+2=6k_2+16$, hence $d^2\equiv 2\mod 3$; a contradiction.

Thus, we proved that $c=1$. Let $r=(r_1,\ldots,r_m)$ be the type of $Q_1$. We have $\sum_{i=1}^m r_i=K\cdot Q_1+1=3$. Let $\ind_i$ be the contribution to $\ind(D)$ coming from the maximal twigs of $D$ contained in the $i$-th branch of $Q_1$. By \ref{lem:case_n=0}(iv) $$\ind_1+\ldots+\ind_m\leq 2-\frac{1}{\gamma}.$$

Consider the case $r=(1,1,1)$. The second and third branch of $Q_1$ are respectively $[2,x',3,(2)_{k_2-1}]$ and $[2,1,3,(2)_{k_3-1}]$ for some $k_2,k_3\geq 0$. We get $\ind_1\leq 2-\ind_2-\ind_3=1$, hence the first branch is $[3,x,2]$.  We have $k_2+k_3=\#D-8=\gamma-1=\gamma_0+1\geq 5$. The characteristic pairs of $\wt Q_1$ are $\binom{6}{4}, \binom{2}{2}_{k_2}, \binom{2}{1}, \binom{1}{1}_{k_3}$. We apply \ref{cor:I and II}(i) and (ii) to $X_0$. We compute $\rho_1=2$, $M(q_1)=10+2k_2+k_3$ and $I(q_1)=26+4k_2+k_3$, hence $3d=23+3k_2+k_3$ and $d^2=105+15k_2+3k_3$. The system of these two equations has only one solution in natural numbers $(d,k_2,k_3)=(12,0,13)$, for which $\gamma_0=k_2+k_3-1=12$, in contradiction to \ref{lem:case_n=0}(i).

Consider the case $r=(3)$. Then $Q_1$ is a chain as in \ref{lem:Q_with_small_KQ}(iv.1), (iv.2), (iv.3) or (iv.4), which gives $k=\#D-6=\gamma+1$. The characteristic pair of $Q_1$ is respectively $\binom{4k+5}{4}$, $\binom{5k+7}{5}$, $\binom{5k+8}{5}$ and $\binom{4k+7}{4}$. We check easily that in each case the system of equations \ref{cor:I and II}(i),(ii) has no integral solution.

Consider the case $r=(2,1)$. Then the second branch of $\wt Q_1$ is $[(2)_{k_2}]$ for some $k_2\geq 0$ and the first branch is either $[(2)_{k_1},4,x,2,2]$ or $[(2)_{k_1},3,2,x,3]$ for some $k_1\geq 0$, hence $k_1+k_2=\#D_0-5=\gamma_0+2$. In the first case the characteristic pairs are $\binom{3k_1+4}{3},\binom{1}{1}_{k_2}$, hence $\rho_1=2$, $M(q_1)=3k_1+k_2+6$ and $I(q_1)=9k_1+k_2+12$. Then the equations read as $3d=5k_1+k_2+16$ and $d^2=35k_1+3k_2+50$. There is a unique solution in natural numbers $(d,k_1,k_2)=(11,1,12)$, for which $\gamma_0=11$, in contradiction to \ref{lem:case_n=0}(i). In the second case the characteristic pairs are $\binom{3k_1+5}{3},\binom{1}{1}_{k_2}$, hence $\rho_1=2$, $M(q_1)=3k_1+k_2+7$ and $I(q_1)=9k_1+k_2+15$. Then the equations read as $3d=5k_1+k_2+18$ and $d^2=35k_1+3k_2+62$. There are no integral solutions.

Consider the case $r=(1,2)$. Then the first branch of $Q_1$ is $[(2)_{k_1},3,x,2]$ for some $k_1\geq 0$ and the second is either $[3,1,2,3,(2)_{k_2-1}]$ or $[2,2,1,4,(2)_{k_2-1}]$ for some $k_2\geq 0$, hence $k_1+k_2=\#D-7=\gamma$. In the first case the characteristic pairs are $\binom{3(2k_1+3)}{3\cdot 2},\binom{3}{3}_{k_2},\binom{3}{2}$, hence $M(q_1)=6k_1+3k_2+16$ and $I(q_1)=36k_1+9k_2+60$. Then the equations \ref{cor:I and II} read as $3d=5k_1+2k_2+18$ and $d^2=35k_1+8k_2+60$. There are no natural solutions. In the second case the characteristic pairs are $\binom{3(2k_1+3)}{3\cdot 2},\binom{3}{3}_{k_2},\binom{3}{1}$, hence $M(q_1)=6k_1+3k_2+15$ and $I(q_1)=36k_1+9k_2+57$. Then the equations read as $3d=5k_1+2k_2+17$ and $d^2=35k_1+8k_2+57$. There are two solutions in natural numbers $(d,k_1,k_2)=(11,0,8)$ and $(16,5,3)$. But we check that for the second solution $\ind(D)>2$.

We are therefore left with the case of a unicuspidal rational curve of degree $11$ with characteristic pairs $\binom{9}{6},\binom{3}{3}_{8},\binom{3}{1}$. We have $\gamma=k_1+k_2=8$, hence $\gamma_0=5$. The divisor $\wt Q_1$ is a fork with a branching $(-2)$-curve $B$ and three maximal twigs $V_1=[1,(2)_7]$, $V_2=[2]$ and $V_3=[3]$. We will show that $t=1$ (see \ref{lem:(X,E)}) and \beq 2K_0+E_0\equiv L_0.\label{eq:deg=11}\eeq By \eqref{eq:2} $(3-\Theta_0)\gamma_t\geq 2(\Theta_1+3)\geq 6$. Then $\Theta_0\leq 3-\frac{6}{\gamma_t}\leq 3-\frac{6}{5}<2$. Since $\Theta_0\geq 2\theta_0$, we get $\theta_0=\Theta_0=0$. By \eqref{eq:1} $\zeta_t=\zeta=1$, so by \eqref{eq:3} $\gamma_t=4$ and $((K_t+E_t)^+)^2=0$. By \eqref{eq:1} $\theta_1=\gamma_0-\gamma_t=1$, hence $\Theta_1\geq 1$.
Therefore, $L_0\cdot E_0=1$ and $L_0\cdot (D_0-E_0)=\Theta_1\in \{1,2,3\}$. Suppose $L_0$ meets some $(-2)$-curve $M$ in $D_0$. Since $\theta_1=1$, $L_0\cdot M\geq 2$, so $M+2L_0$ is nef and we get $0\leq (M+2L_0)\cdot (2K_0+E_0)=2L_0\cdot (2K_0+E_0)=-2$; a contradiction. Thus $L_0$ meets no $(-2)$-curve in $D_0$. Similarly, if $L_0$ meets $C_1$ more than once then $C_1+2L_0$ is nef and we get $0\leq (C_1+2L_0)\cdot (2K_0+E_0)=-4+C_1\cdot E_0=-1$; a contradiction.
\begin{figure}[h]\centering\includegraphics[scale=0.5]{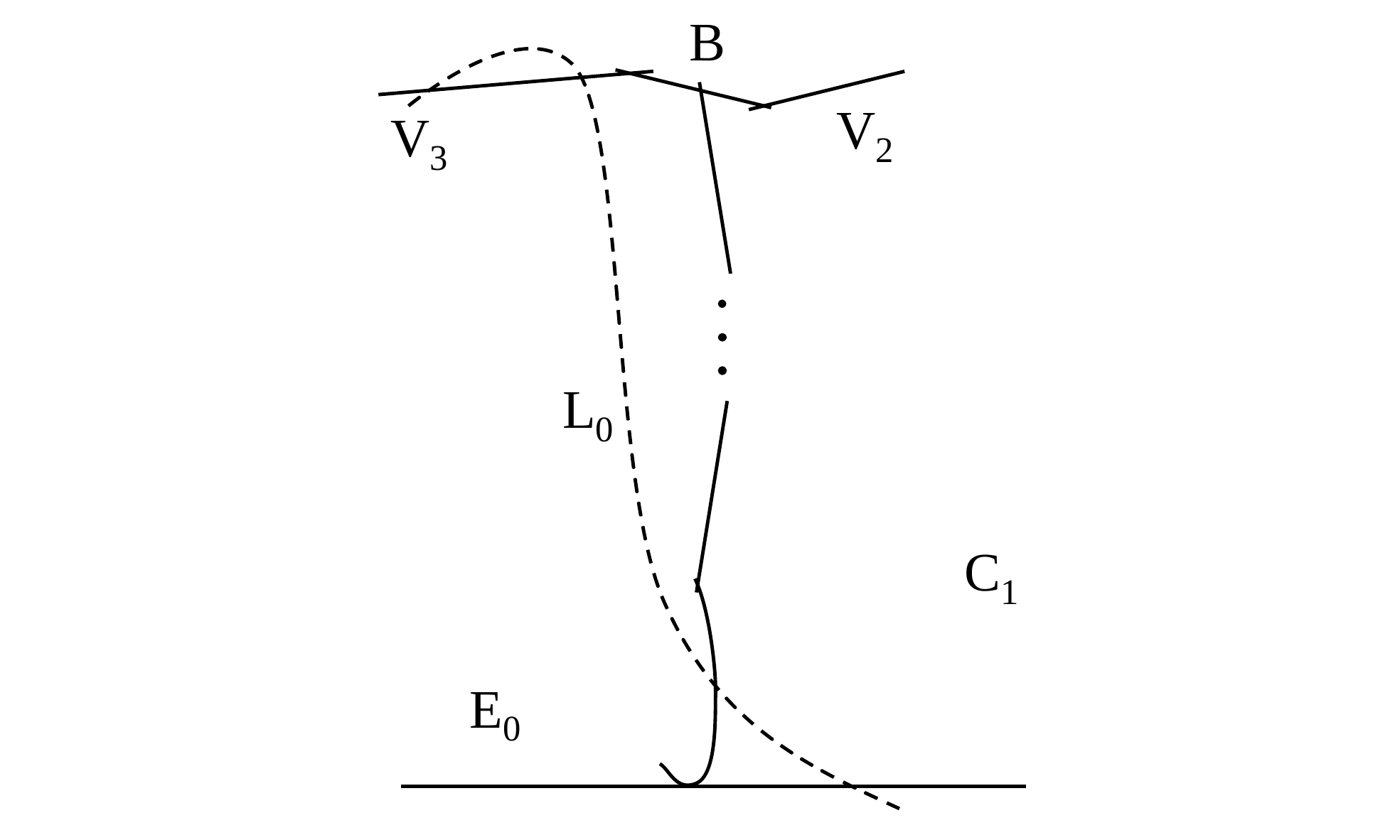}\caption{The divisor $D_0+L_0$ on $X_0$. Proof of Propositon \ref{prop:p2=4_for_n=0}.} \label{fig:5}\end{figure}
Finally, if $L_0\cdot V_3\geq 3$ then $V_3+3L$ is nef and we get $0\leq (V+3L_0)\cdot (2K_0+E_0)=2-3$; a contradiction. Thus, we have $\Theta_1=L_0\cdot (D_0-E_0)=L_0\cdot (V_3+C_1)\leq 3$. We compute $\pi_0(L_0)^2=-1+10L_0\cdot C_1+(L_0\cdot  V_3+2L_0\cdot C_1)^2$. The latter number is a nonzero square, so $L_0\cdot C_1=1$ and $L_0\cdot V_3=2$ (see Fig.\ \ref{fig:4}). We check now that all components of $\wt Q_1$ intersect $2K_0+E_0-L_0$ trivially, so since they generate $\NS_{\Q}(X_0)$, the proof of \eqref{eq:deg=11} is completed.

We now look back at $D$. Let $U$ be the unique $(-1)$-curve in $D$. We have $(\psi_0^*V_1)_{red}-U=V_1'+V_1''$, where $V_1'=[2,2]$ and $V_1''=[(2)_7,4]$. Clearly, the divisor $D-U$ is snc-minimal. We claim that the pair $(X,D-U)$ is almost minimal, i.e.\ that there is no $(-1)$-curve $\ell\subset X$ for which one of \cite[6.21]{Fujita-noncomplete_surfaces} holds. Suppose otherwise. Then $\ell\not\subset D-U$ and $\ell$ meets each connected component of $D-U$ at most once and meets at most two connected components in total. In particular, $\ell\neq U$. The curve $\ell$ does not meet $U+V_1'$, because otherwise we easily find a nef subdivisor of $\ell+V_1'+U$ whose intersection with the effective divisor $2K+E$ is at most $-3+(\ell+U)\cdot E\leq-1$, which is impossible. Therefore, the image of $\ell$ on $X_0$, which we denote by the same letter, is a $(-1)$-curve. By \eqref{eq:deg=11} $1\geq \ell\cdot E_0=\ell\cdot L_0+2$, so $\ell=L_0$. But $L_0\cdot V_3>1$; a contradiction.

Thus for the above unicuspidal curve $\bar E\subset \PP^2$ of degree $11$ the pair $(X,D-U)$ is almost minimal. Note that $V_2+B+V_3+V_1''$ does not contract to a quotient singularity. If $\kappa(X\setminus (D-U))\geq 0$ then the BMY inequality (we use here the version for non-connected boundaries stated in \cite[2.5(ii)]{Palka-exceptional}) gives $(K+D-U)^2+\ind(D-U)\leq 3(\frac{1}{3}+\frac{1}{8})=\frac{11}{8}$. But $(K+D-U)^2=(K+D)^2+U^2-2U\cdot (K+D)=(K+D)^2-3=-2$, so $-2+\ind(D-U)=-2+(\frac{1}{2}+\frac{1}{3}+\frac{8}{25})+\frac{6}{3}+\frac{4}{8}=\frac{124}{75}>\frac{11}{8}$, so we infer that $\kappa(X\setminus (D-U))=-\8$. Since $D-U$ has three connected components, not all contractible to quotient singularities, by structure theorems for almost minimal pairs \cite[2.3.15, 2.5.1.2]{Miyan-OpenSurf} $X\setminus (D-U)$ admits a $\C^1$- or a $\C^*$-fibration. Then \ref{cor:fE>=4}(i) fails; a contradiction.
\end{proof}

By Proposition \ref{prop:p2=4_for_n=0} to get a final contradiction, and hence to prove Theorem \ref{thm:CN} it remains to rule out the case $p_2(\PP^2,\bar E)=4$. By \ref{thm:CN1}(b) $\bar E$ is unicuspidal, i.e.\ $c=1$. Recall that $\zeta=K\cdot (K+E)=K_0\cdot (K_0+E_0)$ and that $\gamma_0=-E_0^2\geq 4$. Let $b\geq 0$ be the number of branching components of $Q_1$. Then $Q_1$ has $b+1$ branches. For $i\leq b+1$ let $\ind_i$ be the contribution of the maximal twigs of $D$ contained in the $i$-th branch of $Q_1$ to $\ind(D)$. Denote by $r=(r_1,\ldots,r_{b+1})$ the type of $Q_1$ and by $U$ the unique $(-1)$-curve of $Q_1$.

\bcor\label{lem:case_n=0,p2=4} If $p_2(\PP^2,\bar E)=4$ then the following hold:\benum[(i)]
    \item $\gamma_0+\tau^*\leq 8+2\zeta$,
    \item $\zeta\in \{-2,-1,0,1,2\}$,
    \item $r_1+r_2+\ldots+r_{b+1}=K\cdot Q_1+1=5-\zeta$,
    \item $\ind_2+\ldots+\ind_{b+1}+\frac{1}{\gamma}\leq 1-\ind_1<\frac{1}{2}$.
\eenum
\ecor

\begin{proof} For (i), (ii) and (iii) see \ref{lem:case_n=0} and \ref{lem:type}.  For (iv) see \ref{lem:case_n=0}(iv) and \ref{rmk:ind>half}.
\end{proof}

We have now $D-E=Q_1$, $\tau^*=\tau_1^*$ and $s=s_1$.

\setcounter{claim}{0}
\begin{proof}[Proof of Theorem \ref{thm:CN} (case $p_2(\PP^2,\bar E)=4$)]

\bcl $\tau^*\geq 1$ and $\zeta\geq -1$. \ecl

\begin{proof}Suppose $\tau^*=0$. Then the last branch of $Q_1$ contains a $(-2)$-tip, so $\ind_{b+1}\geq \frac{1}{2}$. By \ref{lem:case_n=0,p2=4}(iv) $Q_1$ is a chain. Then $\wt Q_1$ is a chain with a $(-1)$-tip, so $\wt Q_1=[1,(2)_k]$ for some $k\geq 0$, hence $\wt Q_1=[(2)_k,3,1,2]$. But $K\cdot Q_1=4-\zeta\geq 2$; a contradiction. Thus $\tau^*\geq 1$. By \ref{lem:case_n=0,p2=4}(i) $2\zeta\geq \tau^*-4\geq -3$, so $\zeta\geq -1$.
\end{proof}

For $i=2,\ldots,b+1$ let $T_i$ be the tip of $Q_1$ contained in its $i$-th branch and let $d_i=-T_i^2$.

\bcl $d_i\geq 4$ for $i\geq 2$. \ecl

\begin{proof} We may assume $b\geq 1$. By \ref{lem:ind>=}(ii) $\ind_i\geq \frac{1}{d_i}$ for $i\geq 2$. Suppose, say, $d_2\leq 3$. Since by \ref{lem:case_n=0,p2=4}(iv) $\ind_2<\frac{1}{2}$, we get $d_2=3$. Then $r_2\geq 2$ and $\ind_2\geq \frac{1}{3}$. Suppose $b\geq 2$. Then \ref{lem:case_n=0,p2=4}(iv) gives $\frac{1}{d_3}+\ldots+\frac{1}{d_{b+1}}+\frac{1}{\gamma}<\frac{1}{2}-\ind_2\leq \frac{1}{6}$, so $d_i\geq 7$ and hence $r_i\geq 5$ for $i\geq 3$. Then \ref{lem:case_n=0,p2=4}(iii) gives $0\leq r_1+\zeta\leq 5-5(b-1)-r_2<5(2-b)\leq 0$; a contradiction. Thus $b=1$. We have now $\ind_1<\frac{2}{3}$. By \ref{lem:ind>=}(ii) and \ref{lem:Q_with_small_KQ}(i)-(iv) $r_1\geq 4$. Then $r_2=5-\zeta-r_1\leq 6-r_1\leq 2$, so $r_2=2$ and hence $r_1=4$ and $\zeta=-1$. By \ref{lem:Q_with_small_KQ}(iii) the second branch of $Q_1$ is $[3,1,2,3,(2)_k]$ for some $k\geq -1$, so $\tau=2$ and $s=0$. We have $\gamma-1-s=\gamma_0+\tau^*\leq 8+2\zeta=6$, so $\gamma\leq 7$. Then $\frac{1}{2}<\ind_1\leq \frac{2}{3}-\frac{1}{\gamma}$, so $\gamma=7$ and $\ind_1\leq\frac{11}{21}$. By \ref{lem:Q_with_small_KQ}(v) the first branch of $Q_1$ contains a $(-2)$-tip and a $(\geq -6)$-tip or a $(-3)$-tip and a $(\geq -5)$-tip. But then $\ind_1\geq \frac{1}{2}+\frac{1}{6}$ and $\ind_1\geq \frac{1}{3}+\frac{1}{5}$ respectively, which is in both cases more than $\frac{11}{21}$; a contradiction.
\end{proof}

\bcl $Q_1$ has at most one branching component. \ecl

\begin{proof} Suppose $b\geq 2$. We have $\sum_2^{b+1}\frac{1}{d_i}\leq \sum_2^{b+1}\ind_i\leq 1-\ind_1<\frac{1}{2}$. Since $d_i\geq 4$ for all $i\geq 2$, we see that at least one of $d_i$'s is bigger than $4$. Since $r_i\geq d_i-2$, \ref{lem:case_n=0,p2=4}(iii) gives $r_1+\zeta\leq 5-\sum_2^{b+1}r_i\leq 5-\sum_2^{b+1}(d_i-2)$. Because $\zeta\geq -1$, we get $0\leq r_1+\zeta\leq 5-2b$, so $b=2$. Then the inequality gives $r_1+\zeta+d_1+d_2\leq 9$, so $\{d_1,d_2\}=\{4,5\}$ and $r_1=1$. Then $\ind_1+\ind_2+\ind_3\geq \frac{5}{6}+\frac{1}{4}+\frac{1}{5}>1$; a contradiction.
\end{proof}

\bcl $L\cdot D_0\geq 3$ for every $(-1)$-curve $L\not \subset D_0$. \ecl

\begin{proof} Suppose $L\cdot D_0\leq 2$. Then $L$ meets $D_0$ in exactly two points, because otherwise $X_0\setminus D_0$ contains a line, in contradiction to \ref{prop:fundamentals}(iii). It follows that the proper transform $L'$ of $L$ on $X$ is a $(-1)$-curve and the snc minimalization of $D+L'$, which we denote by $(X,D+L')\to (Y,D_Y)$, is inner. We compute $(K_Y+D_Y)^2=(K+D+L')^2=(K+D)^2+1=p_2(\PP^2,\bar E)-1=3$ and $\chi(Y\setminus D_Y)=\chi(X\setminus D)=1$. The BMY inequality \ref{lem:BMY} gives $(K_Y+D_Y)^2+\ind(D_Y)\leq 3\chi(Y\setminus D_Y)$, i.e.\ $\ind(D_Y)=0$. Then $D_Y$ has no tips, hence $D$ has at most two tips; a contradiction.
\end{proof}

\bcl $\zeta\geq 0$. \ecl

\begin{proof} By Claim 1 we have $\zeta\geq -1$. Suppose $\zeta=-1$. By \ref{lem:(X,E)} $\theta_0=\zeta_t+1\geq 3-\frac{4}{\gamma_t}\geq 2$ and $6-\Theta_0\geq \frac{2}{\gamma_t}(\Theta_1+\tau^*+2)>0,$ so $\theta_0\geq 2$ and $\Theta_0\leq 5$. We may assume $L_0\cdot E_0=0$, hence $L_0\not\subset D_0$. By Claim 4 $L_0\cdot (D_0-E_0)\geq 3$, so \ref{lem:(X,E)} gives $\Theta_0\geq 2\theta_0+1$. It follows that $\Theta_0=5$ and $\theta_0=2$. Then $\zeta_t=1$ and $\gamma_t=4$, so $\Theta_1+\tau^*+2\leq 2$. But then $\tau^*=0$, in contradiction to Claim 1.
\end{proof}

\bcl $Q_1$ is a chain. \ecl

\begin{proof} By Claim 3 the divisor $Q_1$ has at most one branching component. Suppose it has one.  Then it is of type $(r_1,r_2)$ with $r_1+r_2=5-\zeta\leq 5$. Since $\ind_2<\frac{1}{2}$, the second branch of $Q_1$ does not contain a $(-2)$-tip of $Q_1$, so $r_2\geq 2$. Suppose it contains a $(-3)$-tip. Then $\ind_1<1-\ind_2\leq \frac{2}{3}$. But using \ref{lem:Q_with_small_KQ} we check easily that for $r_1\leq 5-r_2\leq 3$ the inequality $\ind_1<\frac{2}{3}$ fails; a contradiction. Thus the tip of the second branch of $Q_1$ has self-intersection at most $(-4)$, hence $r_2\geq 3$. If $r_2=4$ then $r_1=1$, so $\ind_1\geq \frac{5}{6}$ and $\ind_2\geq \frac{1}{5}$ (the second branch contains no tip with self-intersection smaller than $(-5)$), hence $\ind_1+\ind_2>1$, which is impossible. Therefore, $r_2=3$ and the second branch of $Q_1$ is as in \ref{lem:Q_with_small_KQ}(iv.4). Then $\ind_2=\frac{1}{4}$ and hence $\ind_1<\frac{3}{4}$. This is possible only if $r_1=2$ and the first branch of $Q_1$ is $[3,2,x,3]$. By \ref{lem:case_n=0,p2=4}(iv) we obtain $\frac{1}{\gamma}\leq 1-\frac{1}{4}-\frac{11}{15}$, so $\gamma\geq 60$. Then $\gamma_0+\tau^*=\gamma-1-s\geq 58$, which contradicts \ref{lem:case_n=0,p2=4}(i).
\end{proof}

Let $\binom{c}{p}$ be the characteristic pair of $Q_1$ and let $d=\deg \bar E$. Because $2K_X+E\geq 0$, we have $2K_{\PP^2}+\bar E\geq 0$, so $d\geq 6$. We compute $\ind_1=\frac{c-p}{c}+\frac{p-r}{p}$, where $r$ is the remainder of division of $c$ by $p$. By \ref{lem:case_n=0,p2=4}(iv)
\beq\frac{c-p}{c}+\frac{p-r}{p}+\frac{1}{\gamma}\leq 1.\label{eq:gamma_final}\eeq
The equations \ref{cor:I and II}(i),(ii) give $c=3d+\gamma-p-1$ and $d^2=cp-\gamma$, hence we get a quadratic equation for $d$:
\beq d^2-(3p)d=(\gamma-p-1)p-\gamma.\label{eq:d_final} \eeq
We have $\gamma=\gamma_0+\tau^*+1+s\leq \gamma_0+\tau^*+2\leq 10+2\zeta\leq 14$, so $6\leq \gamma\leq 14$. Since $K\cdot Q_1+3=7-\zeta\leq 7$, \ref{lem:type}(ii) gives $p=\mu(q_1)\leq F_7=13$. In particular, the coefficients of the equation \eqref{eq:d_final} are bounded. Integral solutions of \eqref{eq:d_final} satisfying these bounds do exist, but for all of them the condition \eqref{eq:gamma_final} fails; a contradiction.
\end{proof}

\vfill\eject
\bibliographystyle{amsalpha}
\bibliography{C:/KAROL/PRACA/PUBLIKACJE/BIBL/bibl2}
\end{document}